\documentclass[reqno,12pt]{amsart}
\usepackage{amsmath, amssymb, amsthm, epsfig, tikz, caption,comment}
\usepackage{hyperref, latexsym, times, xurl, color, xcolor, stackrel, setspace, array, booktabs}
\usetikzlibrary{arrows.meta,decorations.pathreplacing,calc,positioning}
\usetikzlibrary{decorations.markings, arrows,angles,quotes}
\usepackage{wrapfig,lipsum}
\usepackage[shortlabels]{enumitem}
\usepackage[mathscr]{euscript}
\usepackage[left=2.5cm, right=2.5cm, bottom=2cm, top=2cm]{geometry}
\def\today{\ifcase\month\or
  January\or February\or March\or April\or May\or June\or
  July\or August\or September\or October\or November\or December\fi
  \space\number\day, \number\year}
\usepackage{dsfont}
\newtheorem{proposition}{Proposition}[section]
\newtheorem{lemma}[proposition]{Lemma}
\newtheorem{corollary}[proposition]{Corollary}
\newtheorem{theorem}[proposition]{Theorem}

\newtheorem*{theorem*}{Theorem}
\theoremstyle{definition}

\theoremstyle{remark}
\newtheorem*{remark}{Remark}
\numberwithin{equation}{section}

\DeclareMathOperator{\sgn}{\mathrm{sgn}}

\DeclareMathOperator*{\Res}{Res}
\DeclareMathOperator{\TV}{TV}
\DeclareMathOperator{\li}{li}
\DeclareMathOperator{\err}{err}
\DeclareMathOperator{\tri}{tri}

\DeclareMathOperator{\sech}{sech}
\DeclareMathOperator{\sinc}{sinc}
\DeclareMathOperator{\csch}{csch}
\DeclareMathOperator{\Ei}{Ei}
\DeclareMathOperator{\arsinh}{arsinh}

\newcommand{\C}{\mathbb{C}}
\newcommand{\R}{\mathbb{R}}

\newcommand{\Z}{\mathbb{Z}}

\colorlet{darkred}{red!70!black}
\colorlet{darkblue}{blue!70!black}
\colorlet{darkgreen}{black!60!white}
\tikzset{
  contour/.style={
    postaction={decorate},
    decoration={markings, mark=at position 0.5 with {\arrow{latex}}}
  },
  contourend/.style={->, >=latex},
  motion/.style={->},
  motiondotted/.style={dotted, ->},
}

\begin{document}

\title{Optimal bounds for sums of non-negative arithmetic functions}
\author[Chirre and Helfgott]{Andr\'{e}s Chirre and  Harald Andr\'{e}s Helfgott}

\date{\today}

\address{Departamento de ciencias - secci\'on matem\'aticas, Pontificia Universidad Cat\'olica del Per\'u, Lima, Per\'u}
\email{cchirre@pucp.edu.pe}

\address{IMJ-PRG, Université Paris Cité, Bâtiment Sophie Germain, 8 Place Aurélie Nemours, 75205 Paris Cedex 13, France.} 
\email{harald.helfgott@gmail.com}

\allowdisplaybreaks
\numberwithin{equation}{section}


\begin{abstract}
Let $A(s) = \sum_n a_n n^{-s}$ be a Dirichlet series admitting meromorphic continuation
to the complex plane. Assume we know the location of the poles of $A(s)$ with $|\Im s| \leq T$, and their residues, for some large constant $T$. It is natural to ask how such finite spectral information may be best used to estimate partial sums $\sum_{n\leq x} a_n$.

  Here, we prove a sharp, general result on sums $\sum_{n\leq x} a_n n^{-\sigma}$ for 
  $a_n$ non-negative, giving an
  optimal way to use information on the poles of $A(s)$ with $|\Im s|\leq T$,
  with no need for zero-free regions. We give not just bounds, but an
  explicit formula with compact support.
  Our bounds on $\psi(x)-x$ are, unsurprisingly, better and often simpler than
  a long list of existing explicit versions of the Prime Number Theorem.
 We treat the case of $M(x)$ and similar functions in a companion paper.
  
  Our solution mixes a Fourier-analytic approach in the style of Wiener--Ikehara with contour-shifting, using optimal approximants of Beurling--Selberg type found in (Graham--Vaaler,
  1981). 
\end{abstract}

\maketitle

  \section{Introduction}

  \subsection{Basic problem}

  Many problems in analytic number theory involve
  estimating sums $\sum_{n\leq x} a_n$ of arithmetic functions.
  Here ``arithmetic function'' means ``a sequence $\{a_n\}_{n=1}^\infty$ that number 
  theorists study'' or, most often, a sequence $\{a_n\}$
  such that the Dirichlet
  series $\sum_n a_n n^{-s}$ converges absolutely on a right half-plane and has
  meromorphic continuation to a function $A(s)$ on $\mathbb{C}$.

  Two basic examples to keep in mind are:
  \begin{itemize}
    \item $a_n=\Lambda(n)$, where
      $\Lambda$ is the von Mangoldt function;
      then $A(s) = -\zeta'(s)/\zeta(s)$;
    \item $a_n = \mu(n)$, where $\mu$ is the Möbius function; then
      $A(s) = 1/\zeta(s)$.
  \end{itemize}
  We will focus on $\Lambda(n)$ and other non-negative functions here.
For $\Lambda(n)$ in particular, there were several sorts of useful estimates, thanks to the fact
that, for any meromorphic function $f$, the residue of $f'(s)/f(s)$ at a zero $\rho$ of $f(s)$ is simply the multiplicity of $\rho$. The situation was nevertheless unsatisfactory, in that the best way to
use information on the zeros of $\zeta(s)$ up to a height $T$ was not known. By ``zeros up to a height'',
we mean those with $|\Im s|\leq T$; they can be determined by rigorous computational means up to large,
finite $T$.

 A naïve student might set out by first looking for weight functions whose Mellin transform is compactly supported. There is
no such thing, but, as we will see, there is a conceptually clean way to proceed
that is sound and amounts to the same.

 \subsection{Results}








 We will need notation for two very mild technical conditions.
 We will ask for a function to be bounded on a 
``ladder'', that is, a union of segments
\begin{equation}\label{eq:sapli}
S = ((-\infty,1] \pm i T)\cup
\bigcup_n (\sigma_n+i[-T,T])\;\;\;\;
\text{for some 
$\{\sigma_n\}_{n=0}^\infty$ with $\sigma_0 = 1$ and $\sigma_n\to -\infty$,}
\end{equation}
so as to be able to conveniently shift a contour to $\Re s = -\infty$.
(``Bounded'' here implies in particular that the function has no poles on $S$.)
Sums of the form 
$\sum_{\rho\in \mathcal{Z}}$ in \eqref{eq:quentino} should be read as $\lim_{n\to\infty}
\sum_{\rho\in \mathcal{Z}: \Re \rho > \sigma_n}$. 

\begin{figure}[ht] 
\scalebox{.6}{
\begin{tikzpicture}[>=latex,scale=0.6,
                    horizaux/.style={line width=1pt},   
                    vertaux/.style={gray,line width=0.25pt},
                    unionarrow/.style={-Latex,line width=0.8pt}]
\def\T{4}
\def\Left{-24}
\def\CrossSz{0.08}
\def\RightScale{1.6}

\draw[vertaux,->] (0,-5) -- (0,5);

\draw[vertaux] (\Left,0) -- (0,0);                 

\foreach \y in {\T,-\T}
  \draw[horizaux] (0,\y) -- (\Left,\y);      

\foreach \x in {-1,-4.5, -19}
  \draw[horizaux] (\x,-\T) -- (\x,\T);        

\filldraw[black] (-1, 0) circle (1.2pt) node[above right] {\(\sigma_1\)};
\filldraw[black] (-4.5, 0) circle (1.2pt) node[above right] {\(\sigma_2\)};
\filldraw[black] (-19, 0) circle (1.2pt) node[above right] {\(\sigma_3\)};

\begin{scope}[xscale=\RightScale,clip]

  \filldraw[black] (1, 0) circle (1.2pt) node[above right] {\(\sigma_0=1\)};

  \draw[vertaux,->] (0,0) -- (2.2,0);               

  \draw[horizaux] (1,-\T) -- (1,\T);      

  \foreach \y in {\T,-\T}{
    \begin{scope}[shift={(1,\y)},xscale={1/\RightScale}]
      \fill (0,0) circle[radius=1pt];
    \end{scope}
  }

  \node[below right=2pt] at (1,-\T) {$1-iT$};
  \node[above right=2pt] at (1,\T)  {$1+iT$};

  \foreach \y in {\T,-\T}
    \draw[horizaux] (1,\y) -- (0,\y);

  \coordinate (SUnionTarget) at (1,2.2);
\end{scope}

\node[font=\Large\bfseries] (Slabel) at ($(SUnionTarget)+(6,1)$) {$S$};
\draw[unionarrow] (Slabel.west) to[out=160,in=-20] (SUnionTarget);

\end{tikzpicture}
}
\end{figure}

The second bit of nomenclature, needed for the same shift, is as follows. An {\em admissible contour} will be a continuous path $s(t)$ such that the intersection of the contour with any strip of width $1$ consists of finitely many smooth arcs of bounded total length.
 
Lastly, we write $O^*(R)$ to mean a quantity of absolute value at most $R$.
 
 \begin{theorem}\label{thm:mainthmB}
Let $\{a_n\}_{n=1}^\infty$, $a_n\geq 0$ for all $n$. Assume that
 $A(s)=\sum_{n}a_n n^{-s}$ converges absolutely for $\Re{s}>1$ and 
 extends meromorphically to $\mathbb{C}$, with a simple pole at $s=1$ and no other poles
 with $\Re s = 1$, $|\Im s|\leq T$ for some $T>0$.
 
 Assume that $A(s) T^s$ is bounded on some ladder $S$ as in \eqref{eq:sapli} and on some admissible contour
 $\mathcal{C}$ going from $1$ to $\Re s = -\infty$
  within $R_{1/4}=(-\infty,1] + i [0,\frac{T}{4}]$ in such a way that
 all real poles of $A(s)$ lie below it and all poles $\rho$ of $A(s)$ with $\Im \rho>0$ lie above it.

Then, for any $\sigma\in \mathbb{R}$ and any $x>T$,
\begin{equation}\label{eq:quentino}\begin{aligned}
\frac{\sum_{n\leq x} a_n n^{-\sigma}}{x^{1-\sigma}}  &=
\frac{\pi}{T}
{\sum_{\rho\in \mathcal{Z}_{A,\mathbb{R}} \cup \{\sigma\}}
  \Res_{s=\rho} \coth \frac{\pi (s-\sigma)}{T} A(s) x^{s-1}} +
  \frac{2\pi}{T}
\Im\!\!\!\sum_{\rho\in \mathcal{Z}_{A}^+(T)} \omega^+_{T,\sigma}(\rho) 
x^{\rho-1} \cdot \Res\limits_{s=\rho} A(s) 
\\&+ O^*\left(\frac{\pi}{T}
\sum_{\rho\in \mathcal{Z}_{A,\mathbb{R}}} \Res_{s=\rho} A(s) x^{s-1} + \frac{2 I_{+,\mathcal{C}}(T)}{T^2} + 
\frac{2\pi}{T}  \Re\!\!\!\sum_{\rho\in \mathcal{Z}_{A}^+(T)} 
\theta_{T,1}(\rho) x^{\rho-1}\cdot
\Res\limits_{s=\rho} A(s) \right) 
,
\end{aligned}\end{equation}
where $\mathcal{Z}_{A}^+(T)$ is the set of poles $\rho$
   of $A(s)$ with $0<\Im \rho< T$, $\mathcal{Z}_{A,\mathbb{R}}$ is the set of real poles of $A(s)$,
   $$\omega^+_{T,\sigma}(s) =\!-\theta_{T,\sigma}(s) \cot \pi \theta_{T,\sigma}(s) + c_{T,\sigma},\;\,
   \theta_{T,\sigma}(s)=1-\frac{s-\sigma}{i T},\;\,
   c_{T,\sigma} = \theta_{T,\sigma}(1+iT) \cot \pi \theta_{T,\sigma}(1+ i T),$$
   $$I_{+,\mathcal{C}}(T) =
\int_0^\infty t |F(1 - t + i T)| x^{-t} dt +
\left|\int_{\mathcal{C}} \Phi(s) F(s) x^{s-1} ds\right|,\;\;\;\; F(s) = A(s) - \frac{\Res_{s=1} A(s)}{s-1},
$$
where $\Phi(s)$ is holomorphic on a neighborhood of $R_{1/4}$,
with $\Phi(1)=0$ and $|\Phi'(s)|\leq 1$ for $s\in R_{1/4}$, and
the restriction $\Phi_{(-\infty,1)}$ is real-valued and of constant sign.
   \end{theorem}

This is an explicit formula, that is, a result expressing a partial sum
as what may be called a sum over the spectral side -- classically,
a sum $\sum_\rho x^\rho$ over zeros $\rho$ of $\zeta(s)$. Here we have two sums on the
first line: the first one goes over real poles -- corresponding, for $A(s) = -\zeta'(s)/\zeta(s)$, to
the pole at $s=1$, the pole at $s=\sigma$ of the weight function, and the
trivial (that is, real) zeros of $\zeta(s)$ -- and the second one over complex poles,
corresponding to non-trivial zeros.
Then there are the third and fourth sums, on the second line, within $O^*()$. Error terms are inevitable because of the restriction $|\Im s|\leq T$; we shall soon see in what sense the ones here are best possible. Notice one can still 
get cancellation in the terms within $O^*()$.

The terms in the first
sum other than $\rho=1$ and $\rho=\sigma$ will typically be negligible. The second and fourth sums -- that is, the sums over complex poles $\rho \in \mathcal{Z}_A^+(T)$ --   will give the main secondary term in a broad but finite range; they will be dominated by the term $\rho=1$ in the third
sum for $x$ very large.  The term $\rho=1$ in the third sum gives us the term $\frac{\pi}{T} \Res_{s=1} A(s)$, dominant for large $x$.

For $\sigma\ne 1$, the total contribution of $\rho=1$ will be 
$\left(\frac{\pi}{T} \coth \frac{\pi (1-\sigma)}{T} +O^*\left(\frac{\pi}{T}\right)\right) \Res_{s=1} A(s)$. We will see that
$\frac{\pi}{T} \coth \frac{\pi (1-\sigma)}{T}$ is the right term -- not an artifact of the method -- and $O^*\left(\frac{\pi}{T}\right)$
is optimal in a strong sense: Proposition \ref{prop:counterxe} states that, for $\sigma=0$, we can 
construct $A(s)$ satisfying the conditions such that $O^*\left(\frac{\pi}{T}\right)$ is sharp; moreover, that error 
is centered on $\frac{\pi}{T} \coth \frac{\pi}{T}$, which comes clearly through.

The contribution of $\rho=\sigma$ for $\sigma\ne 1$, $\sigma\notin \{-2,-4,\dotsc\}$ will be 
$\frac{A(\sigma)}{x^{1-\sigma}}$. This term is negligible for $\sigma=0$, say, but it becomes dominant for
$\sigma>1$, as makes sense: $\sum_{n\leq x} a_n n^{-\sigma}$ then converges.

The weights $\omega_{T,\sigma}^+(s)$, $\theta_{T,1}$ are optimal for our support $[-T,T]$. 
They come from the Fourier transform of an extremal function in the sense of Beurling-Selberg found by Graham-Vaaler \cite{zbMATH03758875}.

The second line in \eqref{eq:quentino} does not depend on $\sigma$. In particular, $\theta_{T,1}$ is not a typo for $\theta_{T,\sigma}$.

Let us see what Thm.~\ref{thm:mainthmB} yields for the primes. 
When we say that the Riemann hypothesis holds up to height $T$, we mean that all non-trivial 
zeros of $\zeta(s)$ with $0<\Im s \leq T$ have real part $\Re s = \frac{1}{2}$.
 \begin{corollary}\label{cor:psi}
 Assume the Riemann hypothesis holds up to height $T\geq 10^7$. For $x> \max(T,10^9)$,
   \[\left|\psi(x)-x\cdot \frac{\pi}{T} \coth \frac{\pi}{T}\right|\leq \frac{\pi}{T-1}\cdot x + \left(\frac{1}{2\pi} \log^2 \frac{T}{2\pi} - \frac{1}{6\pi} \log \frac{T}{2\pi}\right) \sqrt{x},\]
   \[\left|\sum_{n\leq x} \frac{\Lambda(n)}{n} - (\log x - \gamma)\right|\leq  
   \frac{\pi}{T-1} + \left(\frac{1}{2\pi} \log^2 \frac{T}{2\pi} - \frac{1}{6\pi} \log \frac{T}{2\pi}\right) \frac{1}{\sqrt{x}}
   ,\]
where $\gamma=0.577215\dotsc$ is Euler's constant.
\end{corollary}
   
   
We know that the Riemann hypothesis holds up to height $3\cdot 10^{12}+1+\frac{\pi}{3}$ (in fact, height $3000175332800$) thanks to Platt and Trudgian \cite{zbMATH07381909}. We obtain the following immediately.
 \begin{corollary}\label{cor:psiexpl}
 For any $x\geq 1$,
 \[|\psi(x)-x|\leq \frac{\pi}{3\cdot 10^{12}}\cdot x + 113.67 \sqrt{x},\]
   \[\sum_{n\leq x} \frac{\Lambda(n)}{n} = \log x - \gamma + O^*\left(\frac{\pi}{3\cdot 10^{12}} + \frac{113.67}{\sqrt{x}}\right).\]
     \end{corollary}

     The constant $\frac{\pi}{T}$ in front of the main error term is, as we said, provably optimal in a rather precise sense -- if we remain agnostic
     as to what happens above height $T$. In contrast, the factor we are approximating as
     $C_T = \frac{1}{2\pi} \log^2 \frac{T}{2\pi} - \frac{1}{6\pi} \log \frac{T}{2\pi}$ 
     is ``best'' in a much weaker sense: if the ordinates of the zeros of $\zeta(s)$ are linearly independent, as is believed, then the contribution of the zeros up to $T$ should really be that large for some (most likely 
     very large, very rare) $x$. Of course, for such $x$, the main error term overwhelms the contribution of the zeros. For moderate $x$, the factor $C_T$ can most likely be reduced to a small constant by means of FFT-based bounds. The basic idea there is known, but we outline an improved method in \S \ref{subs:companal}. 
     
 


We could use Theorem \ref{thm:mainthmB} to cover the case of bounded functions as well, in that
a bounded function becomes non-negative when one adds a constant. That, however,
would be suboptimal by a factor of $2$; it is more logical to proceed as we do in the
companion paper \cite{Moebart}.

\subsection{Context and methods}
\subsubsection{Existing results}\label{subs:predif}
 There are several kinds of explicit estimates on
  $\psi(x)$:
  \begin{enumerate}[(i)]
  \item\label{it:brute} For $x$ relatively small, we can use computational methods;
  \item\label{it:interpsi} For $x$ in a broad intermediate range (roughly
    $10^{19}\leq x\lesssim e^{2280}$ prior to this work), the best bounds are of type
    $|\psi(x)-x|\leq \epsilon x$, and rely in part on finite verifications of the Riemann hypothesis, that is, computations showing that all zeros of $\zeta(s)$ with $0<|\Im s|\leq T$ (for some large constant $T$) lie on the critical line $\Re s = \frac{1}{2}$. 
  \item For $x$ beyond that intermediate range,
    the best bounds are of the form \begin{equation}\label{eq:classicpnt}|\psi(x)-x|\leq C x (\log x)^\beta \exp(-c \sqrt{\log x}),\end{equation} which should be familiar from classical, non-explicit proofs of the Prime Number Theorem (PNT), based on
    a zero-free region. 
  \end{enumerate}
 When one first encounters explicit estimates for $\psi(x)$, the existence of the
 intermediate range \ref{it:interpsi} can come as a surprise. The familiar
  bound \eqref{eq:classicpnt} is really useful only once $x$ is very large. There is actually a range between \ref{it:brute} and \ref{it:interpsi} where yet another, partly computational approach can be used \cite{zbMATH06864192}; we will discuss
  how one could do better in \S \ref{subs:companal}.

We will discuss the best bounds $|\psi(x)-x|\leq \epsilon x$ to date in \S \ref{subs:compmeth}.
The situation for $M(x)$ was far more dire; see \cite[\S 1.3]{Moebart}. 

\subsubsection{Strategy}
It is by now a commonplace observation that it is often best to estimate sums
$\sum_{n\leq x} a_n$ by first approximating them by smoothed sums
$\sum_{n\leq x} a_n \eta(n/x)$, where $\eta$ is continuous. One way to proceed then is to take Mellin transforms to obtain $$\sum_{n\leq x} a_n \eta(n/x) = \int_{\sigma-i\infty}^{\sigma+i\infty} A(s) x^s M\eta(s) ds,$$ where $A(s)$ is the Dirichlet series $\sum_n a_n n^{-s}$. 
The difficulty here is that the restriction of $M\eta(s)$ to a vertical
line cannot be compactly supported, as $M\eta(s)$ is meromorphic.

Matters are clearer if, instead of defining $\eta$, we choose a weight
$\varphi:\mathbb{R}\to \mathbb{R}$ and work out
\begin{equation}\label{eq:schokolade}\int_{\sigma-i\infty}^{\sigma+i\infty} A(s) x^s \varphi\left(
\frac{s-\sigma}{i}\right) ds.\end{equation} The integral in \eqref{eq:schokolade} can be expressed as a sum involving
$a_n$ and $\widehat{\varphi}$ (Lemma \ref{lem:basicfour}). It is clear that we can
take $\varphi$ to be compactly supported and still have $\widehat{\varphi}$ be in
$L^1(\mathbb{R})$. This is not a new insight; it underlies 
the first half of the proof of the Wiener--Ikehara theorem (see, e.g., \cite[p. 43--44]{zbMATH05203418}). That has also been combined -- in a different context -- with
shifting the contour to the left: Ramana and Ramaré \cite{zbMATH07182422} worked
with a piecewise polynomial $\varphi$, and of course polynomials are entire.

 We can state, more generally: it is enough
for a function $\varphi$ supported on a compact interval $\mathbf{I}$ to equal a holomorphic
or meromorphic function $\Phi$ {\em on $\mathbf{I}$} (as opposed to: on all of $\mathbb{R}$). We can then replace $\varphi$ by $\Phi$ in $\int_{\sigma + i\mathbf{I}}
A(s) x^s \varphi\left(\frac{s-\sigma}{i}\right)$, and then shift the contour.

Our way of estimating the resulting terms is different from that in
\cite{zbMATH07182422}. It leads us to an optimization problem -- how to best approximate a given function by a band-limited function. 
This is a problem of a kind first solved by Beurling,
and later by Selberg; depending on the function being approximated, the solution
can be that found by Beurling (and Selberg), or one given by Graham and Vaaler \cite{zbMATH03758875}, or something else. There is by now a rich literature on optimal approximants whose results we can use.

{\em Note.} The Beurling-Selberg approximant is familiar to many analytic number theorists 
through Selberg's proof of the optimal version of the large sieve; see \cite[\S 20]{Sellec} or, e.g., \cite[Thm.~9.1]{MR2647984}. Somewhat closer to us -- there is a literature combining
such approximants with the Guinand–Weil explicit formula to give explicit bounds for quantities
associated with $\zeta(s)$ \cite{MR2331578, MR2781205, MR3063902, MR3778242, MR3980935}; these results concern the line $\Re s = \frac{1}{2}$, and assume the Riemann hypothesis.

\subsection{Structure of the paper}
We start with Fourier-based replacements for Perron's formula (\S \ref{subs:perfour})
and use the non-negativity of $a_n$ to reduce the problem of estimating our partial sums to that
of estimating sums with Fourier transforms as weights (\S \ref{subs:abovebelow}).

We show by an explicit construction that the leading term in our results is optimal, in \S \ref{sec:asideopt}, which is independent from the rest of the paper.

We work out the solutions to
our optimization problems in \S \ref{sec:modidon}. 
In \S \ref{sec:contour}, we shift contours, first replacing our function by two holomorphic functions, each of which is identical to it on one half of the segment on which we are integrating. We are left with horizontal integrals at $\Im s = \pm T$ and a seam line down the middle (Figure \ref{fig:kolobrz2}).
Proving our main result, Thm.~\ref{thm:mainthmB}, is then rather easy (\S \ref{sec:proofmain}); we estimate our sums $\sum_n a_n/n^\sigma$,
dealing with $\sigma=1$ by passing to a limit. 

The main result can be applied immediately to give a finite explicit formula for $a_n = \Lambda(n)$. What remains is to use that explicit formula to obtain a clean estimate on $\sum_{n\leq x} \Lambda(n) n^{-\sigma}$.  While our weights are not complicated, giving a fair approximation to the contribution of the zeros on the critical line (Prop.~\ref{prop:vihuela}) still takes some work (\S \ref{subs:nontrivz}). We could have carried this task out computationally, as in \cite{Moebart}, but we wish to give results automatically applicable to all $T$ (Cors.~\ref{cor:psi} and \ref{cor:psiexpl}). For the same reason, we estimate the integrals (\S \ref{sec:horint}) on the contours in Figure \ref{fig:kolobrz2}, rather than bounding them computationally. One can thus see \S \ref{sec:zetres}--\ref{sec:horint} (and Appendix \ref{sec:zeroszeta}) as in some sense optional, or rather as necessitated by a choice.

We finish the proofs of our estimates on $\Lambda(n)$ in \S \ref{sec:lambconcl}.
We discuss work past and future in \S \ref{sec:finremark}. 


Appendices~\ref{sec:estzeta} and \ref{sec:zeroszeta} are devoted to explicit estimates on $\zeta(s)$. Appendix~\ref{sec:darf} gives useful estimates on other functions, and a convenient expression for our weight on the integers.

\subsection{Notation} We define the Fourier transform 
$\widehat{f}(x) = \int_{-\infty}^\infty f(t) e^{-2\pi i x t} dt$ for $f\in L^1(\R)$,
extended to $f\in L^2(\R)$ in the usual way (e.g., Thm.~9.13 in \cite{zbMATH01022658}, which, however, puts the factor of $2\pi$ elsewhere). We write $\|f\|_1$ and $\|f\|_{\infty}$ for the $L^1$-norm and $L^\infty$-norm respectively, and $\|f\|_{\TV}$ for the total variation of a function $f:\mathbb{R}\to \mathbb{C}$.

As above, we use $O^*(R)$ to mean a quantity of absolute value at most $R$ (Ramaré's notation).

When we write $\ll_{N,g}$ (say),  we mean that the implied constant depends on $N$ and on the definition of $g$, and nothing else. 

When we write $\sum_{n}$, we mean a sum over positive integers; we use
$\sum_{n\in \mathbb{Z}}$ for a sum over all integers. We write $\mathbb{Z}_{>0}$ for the set of positive integers. We let $\mathds{1}_S$ be the characteristic function of a set $S\subset \mathbb{R}$, that is, $\mathds{1}_S(x)=1$ for $x\in S$, $\mathds{1}_S(x)=0$ for $x\not\in S$.

\subsection{Acknowledgements}
We are much obliged to David Platt, who shared his files of low-lying zeros of $\zeta(s)$, and to
several contributors, often anonymous, to MathOverflow and Mathematics Stack Exchange, who were of particular help with the material in Appendix \ref{sec:estzeta}. 
 We are also grateful to Kevin Ford, Habiba Kadiri and Nathan Ng for their feedback and encouragement.

       




  
  

  


  \section{From a complex integral to an $L^1(\R)$ approximation problem}\label{sec:coeur}
  \subsection{A smoothed Perron formula based on the Fourier transform}\label{subs:perfour}

  We want to work with a fairly arbitrary
  weight function $\varphi$
  on a vertical integral, and work out what will happen on the
  side of the sum, knowing that the Fourier transform $\widehat{\varphi}$ will appear.


  The following proposition is close to several in the literature; it is a natural starting point
  for the Wiener--Ikehara Tauberian method, and it is also in some sense akin to the Guinand-Weil formula.
  Statements like Lemma \ref{lem:basicfour}
  are often given
  with $\varphi$ and $\widehat{\varphi}$ switched; that is, of course, logically equivalent, but then
  the author may be tempted to assume  $\widehat{\varphi}$ (in our sense) to
  be compactly supported (see, e.g., \cite[Prop.~7]{taoblog254a2}). Curiously, a statement in 
  the formal-proof project ``Prime Number Theorem And\dots''
  \cite{PNT+} is very close to Lemma~\ref{lem:basicfour}.\footnote{Indeed, it has now
  become equivalent to it, since we contacted the project participants to show them that one of the assumptions of their Lemma 1 was superfluous (March 3, 2025).}
  See also \cite[Thm. 2.1]{zbMATH07182422}. At any rate, we give a proof from scratch, as it is brief and straightforward.
  
  \begin{lemma}\label{lem:basicfour}
    Let $A(s) = \sum_n a_n n^{-s}$ be a Dirichlet series converging absolutely for
    $\Re s = \sigma$.
    Let $\varphi:\mathbb{R}\to \mathbb{C}$ be in $L^1(\R)$. Then, for any $x>0$ and any $T>0$,
    \begin{equation}\label{eq:perreq1}\frac{1}{2\pi iT} \int_{\sigma-i\infty}^{\sigma+i\infty} \varphi\left(\frac{\Im s}{T}\right)
    A(s) x^s ds = \dfrac{1}{2\pi}\sum_{n} a_n \left(\frac{x}{n}\right)^\sigma
    \widehat{\varphi}\left(\frac{T}{2\pi} \log \frac{n}{x}\right).\end{equation}
  \end{lemma}
  \begin{proof}
    By the dominated convergence theorem,
    $$
    \int_{\sigma-i\infty}^{\sigma+i\infty} \!\varphi\left(\frac{\Im s}{T}\right)
    A(s) x^s ds = 
        \int_{\sigma-i\infty}^{\sigma+i\infty} \!\varphi\left(\frac{\Im s}{T}\right)
        \sum_{n} a_n n^{-s} x^s ds =
\sum_{n} a_n \int_{\sigma-i\infty}^{\sigma+i\infty} \!\varphi\left(\frac{\Im s}{T}\right)
\left(\frac{x}{n}\right)^s ds$$
since $\varphi\in L^1(\R)$
and, for any $u$, $|\sum_{n\leq u} a_n n^{-s}|$ is bounded by $\sum_n |a_n| n^{-\sigma}$. Clearly,
$$\frac{1}{i T} \int_{\sigma-i\infty}^{\sigma+i\infty} \varphi\left(\frac{\Im s}{T}\right)
\left(\frac{x}{n}\right)^s ds =
\left(\frac{x}{n}\right)^\sigma \int_{-\infty}^\infty
\varphi(t) e^{i T t \log \frac{x}{n}} dt =
\left(\frac{x}{n}\right)^\sigma \widehat{\varphi}\left(\frac{T}{2\pi} \log \frac{n}{x}\right).
$$
  \end{proof}

  It will be useful to be able to integrate on the very edge of the region of absolute convergence of $A(s)$. We will have
  to be careful, as there is a discontinuity on the edge.

We will need the following simple lemma to deal with what in effect is a pole.
\begin{lemma}\label{lem:ariosto}
Let $\varphi:\mathbb{R}\to\mathbb{C}$ be such that $\varphi, \widehat{\varphi}\in L^1(\mathbb{R})$. Let $T>0$. Define $\Phi_{T,\epsilon}(t) = \frac{\varphi(t)}{i T t + \epsilon}$ for $\epsilon>0$. Then
\begin{equation}\label{eq:huckleb}
\widehat{\Phi_{T,\epsilon}}(\xi) = \frac{2\pi}{T} \int_\xi^\infty e^{-{2\pi(y-\xi)\epsilon}/{T}} \widehat{\varphi}(y) dy.\end{equation} 
\end{lemma}
\begin{proof}
Write $\Phi_{T,\epsilon} = \varphi\cdot g_{T,\epsilon}$,
where $g_{T,\epsilon}(t) = \frac{1}{i T t + \epsilon}$.
    By a table of Fourier transforms (e.g., \cite[App.~2]{zbMATH05232956}) 
\begin{equation}\label{eq:agaragar} \widehat{g_{T,\epsilon}}(x)
 = H(-x)\cdot
	\frac{2\pi}{T} e^{\frac{2\pi \epsilon x}{T}},\end{equation}
where $H(x)$ is the Heaviside function $\mathds{1}_{x>0} + \frac{1}{2}
\mathds{1}_{x=0}$. In particular, $\widehat{g_{T,\epsilon}}\in L^1(\mathbb{R})$.

Hence, by \cite[Thm.~9.2(c)]{zbMATH01022658} and a couple of applications of the Fourier inversion theorem
\cite[Thms.~9.11 and 9.14]{zbMATH01022658}, 
$\widehat{\Phi_{T,\epsilon}}$ equals the convolution of $\widehat{\varphi}$ with
$\widehat{g_{T,\epsilon}}$ .
 (We are being careful because $g_{T,\epsilon}$ is not in $L^1(\mathbb{R})$.) In other words,
 \eqref{eq:huckleb} holds.
\end{proof}

  \begin{proposition}\label{prop:edgepole}
	Let $A(s) = \sum_n a_n n^{-s}$ be a Dirichlet series converging absolutely for
	$\Re s > 1$. Assume that $A(s)- {1}/{(s-1)}$ extends continuously to
    $1 + i [-T,T]$. Let $\varphi:\mathbb{R}\to \mathbb{C}$ be supported on $[-1,1]$, in $L^1(\R)$, with $\widehat{\varphi}(y) = O(1/|y|^\beta)$ for
    some $\beta > 1$ as $y\to \pm\infty$.
Assume  $\sum_{n>1} \frac{|a_n|}{n\log^\beta n} < \infty$.

Then, for any $x>0$,
\begin{equation}\label{eq:passiflora}\begin{aligned}
\dfrac{1}{2\pi}\sum_{n} a_n \,\frac{x}{n}\,
\widehat{\varphi}\left(\frac{T}{2\pi} \log \frac{n}{x}\right)
&= \frac{1}{2\pi iT} \int_{1-i T}^{1+i T} \varphi\left(\frac{\Im s}{T}\right)
	 \left(A(s) - \frac{1}{s-1}\right) x^s ds\\
     &+ \left(\varphi(0)-\int_{-\infty}^{-\frac{T}{2\pi} \log x}\widehat{\varphi}(y) dy\right)\frac{x}{T}
\end{aligned}\end{equation}
\end{proposition}
\begin{proof}
We can apply Lemma \ref{lem:basicfour} for $\Re s=1+\epsilon$, $\epsilon>0$ arbitrary. Then
  \begin{equation}\label{eq:perreq122}\frac{1}{2\pi iT} \int_{1+\epsilon-i T}^{1+\epsilon+i T} \varphi\left(\frac{\Im s}{T}\right)
	A(s) x^s ds = \dfrac{1}{2\pi}\sum_{n} a_n \left(\frac{x}{n}\right)^{1+\epsilon}
	\widehat{\varphi}\left(\frac{T}{2\pi} \log \frac{n}{x}\right)\end{equation}
    since $\varphi$ is supported on 
    $[-1,1]$. Let  
$\Phi_{T,\epsilon}(t) = \frac{\varphi(t)}{i T t + \epsilon}$. Clearly, $\Phi_{T,\epsilon}(\frac{\Im s}{T}) = \varphi(\frac{\Im s}{T})\tfrac{1}{s-1}$ when $\Re s = 1+\epsilon$. Thus,
	\begin{equation}\label{eq:gorgoroth}
	\frac{1}{2\pi iT} \int_{1+\epsilon-i \infty}^{1+\epsilon+i \infty} 
	\frac{\varphi\left(\frac{\Im s}{T}\right) x^s}{(s-1)} ds = 
    \frac{x^{1+\epsilon}}{2\pi T} \int_{-\infty}^{\infty} \Phi_{T,\epsilon}\left(\frac{t}{T}\right)
	e^{i t \log x} dt =
    \dfrac{x^{1+\epsilon}}{2\pi}\widehat{\Phi_{T,\epsilon}}\left(-\frac{T}{2\pi} \log x\right),
\end{equation}
and so, subtracting from \eqref{eq:perreq122}, we get
  \begin{equation*}\begin{aligned}\frac{1}{2\pi iT} \int_{1+\epsilon-i T}^{1+\epsilon+i T} &\varphi\left(\frac{\Im s}{T}\right)
	 \left(A(s) - \frac{1}{s-1}\right) x^s ds \\&= \dfrac{1}{2\pi}\sum_{n} a_n \left(\frac{x}{n}\right)^{1+\epsilon}
	\widehat{\varphi}\left(\frac{T}{2\pi} \log \frac{n}{x}\right)
    - \dfrac{x^{1+\epsilon}}{2\pi}\widehat{\Phi_{T,\epsilon}}\left(-\frac{T}{2\pi} \log x\right).
    \end{aligned}\end{equation*} 

Now let
  $\epsilon\to 0^+$. By the continuity of $s\mapsto A(s)-1/(s-1)$, the fact that $\varphi$ is compactly supported, and by dominated convergence, we arrive at 
\begin{equation}\label{eq:droctulft}\begin{aligned}\frac{1}{2\pi iT} \int_{1-i T}^{1+i T} &\varphi\left(\frac{\Im s}{T}\right)
	 \left(A(s) - \frac{1}{s-1}\right) x^s ds \\&= \dfrac{1}{2\pi}\sum_{n} a_n \frac{x}{n}
	\widehat{\varphi}\left(\frac{T}{2\pi} \log \frac{n}{x}\right)
    - \dfrac{x}{2\pi}\lim_{\epsilon\to 0^+}\widehat{\Phi_{T,\epsilon}}\left(-\frac{T}{2\pi} \log x\right).
    \end{aligned}\end{equation}

By Lemma \ref{lem:ariosto},
$$\widehat{\Phi_{T,\epsilon}}\left(- \frac{T}{2\pi} \log x\right) 
= \frac{2\pi}{T} \int_\xi^\infty  e^{-{2\pi (y-\xi)\epsilon}/{T}} \widehat{\varphi}(y) dy
$$
for $\xi = - \frac{T}{2\pi} \log x$.
  Since $\widehat{\varphi}(y) = O(1/|y|^\beta)$ as $y\to \pm\infty$ and $\beta > 1$,
  this integral converges to
$$\int_\xi^\infty \widehat{\varphi}(y) dy=\varphi(0)-\int_{-\infty}^\xi\widehat{\varphi}(y) dy.$$ 
\end{proof}

\subsection{Bounding unsmoothed sums from above and below}\label{subs:abovebelow}
For $\{a_n\}_{n=1}^\infty$ and $\sigma\in \mathbb{R}\setminus \{1\}$, 
let
\begin{equation}\label{eq:sotodef}S_\sigma(x) = \sum_{n\leq x} \frac{a_n}{n^\sigma}\;\;\;\text{if $\sigma<1$,}\;\;\;\;\;\;\; \;\;\;\;\;\;\; S_\sigma(x)= \sum_{n\geq x} \frac{a_n}{n^\sigma}\;\;\;\text{if $\sigma>1$.}\end{equation}
Our task is to estimate these sums.

 For $\lambda\in \mathbb{R}\setminus \{0\}$, we define $I_\lambda$ to be the truncated exponential 
\begin{equation}\label{eq:truncexp}I_\lambda(u) = \mathds{1}_{[0,\infty)}(\sgn(\lambda) u)\cdot e^{-\lambda u}.\end{equation} 
The motivation for this definition is that, for any $\sigma\neq 1$ and $x\geq 1$:
\begin{equation}\label{eq:sombrerero}
S_\sigma(x) = x^{-\sigma} \sum_{n} a_n \frac{x}{n} I_{\lambda}\left(\frac{T}{2\pi}\log \frac{n}{x}\right),
\end{equation}
where $T>0$, and $\lambda = 2\pi (\sigma- 1)/T$.

\begin{proposition}\label{prop:gennonneg}
Let $\{a_n\}_{n=1}^\infty$, $a_n\geq 0$ for all $n$. Assume that
 $A(s)=\sum_{n}a_n n^{-s}$ converges absolutely for $\Re{s}>1$, and 
$A(s)-1/(s-1)$ extends continuously to $1+i [-T,T]$ for some $T>0$.

Let $S_\sigma$ be as in \eqref{eq:sotodef} for $\sigma\ne 1$. Let $I_\lambda$ be as in \eqref{eq:truncexp} with $\lambda = \frac{2\pi (\sigma-1)}{T}$,
$\sigma\ne 1$. Let
$\varphi_{\pm}:\mathbb{R}\to\mathbb{C}$ be supported on $[-1,1]$ and in $L^1$, with $\widehat{\varphi_{\pm}}(y)=O(1/|y|^\beta)$ as $y\to \pm\infty$
for some $\beta> 1$. Assume 
\begin{equation}\label{eq:lowup} \widehat{\varphi_{-}}(y)\leq I_{\lambda}(y)\leq \widehat{\varphi_{+}}(y)\end{equation} for all $y\in\R$. 
Then, for any $x\geq 1$,
\begin{equation}\label{eq:rotterdam2}\begin{aligned}
S_\sigma(x) &\leq  
     \frac{2\pi x^{1-\sigma}}{T} \varphi_+(0) + \frac{x^{-\sigma}}{iT} \int_{1-i T}^{1 +i T} \varphi_+\left(\frac{\Im s}{T}\right)
	 \left(A(s) - \frac{1}{s-1}\right) x^s ds -
     \frac{\mathds{1}_{(-\infty,1)}(\sigma)}{1-\sigma},\\
S_\sigma(x) &\geq
     \frac{2\pi x^{1-\sigma}}{T} \varphi_-(0) + \frac{x^{-\sigma}}{iT} \int_{1-i T}^{1 +i T} \varphi_-\left(\frac{\Im s}{T}\right)
	 \left(A(s) - \frac{1}{s-1}\right) x^s ds -
     \frac{\mathds{1}_{(-\infty,1)}(\sigma)}{1-\sigma}.
\end{aligned}
\end{equation}
\end{proposition}
\begin{proof}
By \eqref{eq:sombrerero}, since $a_n\geq 0$ for all $n\geq 1$,
$$\begin{aligned} 
x^{-\sigma} \sum_{n} a_n \frac{x}{n} \widehat{\varphi_{-}}\left(\frac{T}{2\pi} \log \frac{n}{x}\right)\leq S_\sigma(x) \leq 
x^{-\sigma} \sum_{n} a_n \frac{x}{n} \widehat{\varphi_{+}}\left(\frac{T}{2\pi} \log \frac{n}{x}\right).
\end{aligned}$$
Since $a_n$ is non-negative,
the condition $\sum_{n>1} \frac{|a_n|}{n(\log n)^{\beta}} < \infty$
holds for every $\beta>1$ by \cite[Thm.~5.11]{MR2378655} 
(Hardy--Littlewood's real Tauberian theorem) and summation by parts or Abel summation. Hence, we may apply Proposition \ref{prop:edgepole}:
\begin{align*}
\sum_{n} a_n \,\frac{x}{n}\,
\widehat{\varphi_\pm}\left(\frac{T}{2\pi} \log \frac{n}{x}\right)
&= \frac{1}{iT} \int_{1-i T}^{1+i T} \varphi_\pm\left(\frac{\Im s}{T}\right)
	 \left(A(s) - \frac{1}{s-1}\right) x^s ds\\
     &+ \frac{2\pi x}{T} \varphi_\pm(0) -
     \frac{2\pi x}{T} \int_{-\infty}^{-\frac{T}{2\pi} \log x}\widehat{\varphi_\pm}(y) dy .
\end{align*}
If $\lambda>0$, then $I_\lambda(y)=0$ for $y$ negative,
and so, by \eqref{eq:lowup},
$$-\int_{-\infty}^{-\frac{T}{2\pi} \log x}\widehat{\varphi_+}(y)dy \leq 0
\;\;\;\;\;\;\text{and}\;\;\;\;
-\int_{-\infty}^{-\frac{T}{2\pi} \log x}\widehat{\varphi_-}(y)dy \geq 0.
$$
If $\lambda<0$, then $I_\lambda(y) = e^{-\lambda y}$ for $y$ negative,
and so
$$\int_{-\infty}^{-\frac{T}{2\pi} \log x}I_\lambda(y)dy =
\frac{e^{\lambda \frac{T}{2\pi} \log x}}{-\lambda} = 
\frac{x^{\sigma-1}}{-\lambda};$$
hence, again by \eqref{eq:lowup},
$$-\int_{-\infty}^{-\frac{T}{2\pi} \log x}\widehat{\varphi_+}(y)dy \leq -\frac{x^{\sigma-1}}{-\lambda}\;\;\;\;\;\;\text{and}\;\;\;\;
-\int_{-\infty}^{-\frac{T}{2\pi} \log x}\widehat{\varphi_-}(y)dy \geq -\frac{x^{\sigma-1}}{-\lambda}.
$$

Since $x^{-\sigma}\cdot \frac{2\pi x}{T} \cdot \frac{x^{\sigma-1}}{-\lambda}=
\frac{1}{1-\sigma}$, we are done.
\end{proof}

\section{An aside on optimality}\label{sec:asideopt}
We will now show that the leading term of Thm.~\ref{thm:mainthmB} is sharp. The construction is inspired by the well-known example $A(s) = (\zeta(s+i)+\zeta(s-i))/2 = \sum_n \cos(\log n)\cdot n^{-s}$, often used to show that one cannot derive asymptotics for $\sum_{n\leq x} a_n$ just from the behavior of $\sum_n a_n n^{-\sigma}$ for real $\sigma\to 1^+$. 

\begin{proposition}\label{prop:counterxe}
Let $T\geq 1$. For every $\epsilon>0$, there are $\{a_n\}_{n=1}^\infty$, $a_n\geq 0$, such that
\begin{itemize}
\item $A(s) = \sum_n a_n n^{-s}$ converges absolutely for $\Re s > 1$, with meromorphic continuation to
$\mathbb{C}$, 
\item $A(s)$ has a simple pole at $s=1$, with residue $1$, and no other poles with $|\Im s|\leq T$, 
\end{itemize}
and \[\limsup_{x\to\infty} \frac{1}{x}\sum_{n\leq x}a_n > 
\frac{\pi}{T} \coth \frac{\pi}{T} + \frac{\pi}{T} - \,\epsilon,\quad
\liminf_{x\to\infty} \frac{1}{x}\sum_{n\leq x}a_n < 
\frac{\pi}{T} \coth \frac{\pi}{T} - \frac{\pi}{T} + \,\epsilon.\]
\end{proposition}
The main term in Theorem \ref{thm:mainthmB} for $\sigma=0$ is precisely
$\frac{\pi}{T} \coth \frac{\pi}{T} + O^*\left(\frac{\pi}{T}\right)$: the contribution of $\rho=1$ to the first sum on the right of \eqref{eq:quentino} is $\frac{\pi}{T} \coth \frac{\pi}{T}$,
whereas its contribution to the third sum on the right is $\frac{\pi}{T}$. In other words,
Theorem \ref{thm:mainthmB} is optimal ``to the dot'', among all methods that use only information 
on poles with $|\Im s|\leq T$.
\begin{proof}
It is clearly enough to prove the statement with the weaker condition that $|\Im s|<T$, as we can then
apply that statement with $T_+$ instead of $T$, where $T_+$ tends to $T$ from above.

Let $
F_K(t):=\sum_{k=-K}^{K}\Big(1-\frac{|k|}{K+1}\Big)e^{ikt}
$ be the Fejér kernel, which is non-negative (see, e.g., \cite[\S III.3]{zbMATH01881986} or \cite[Lem.~2.2 (i)]{zbMATH07503892}). Define 
$a_n = F_K(T_{}\log n)$. Then
\begin{equation}\label{eq:dududu}\begin{aligned}\sum_n a_n n^{-s} &= 
\sum_n n^{-s} \sum_{k=-K}^K\left(1-\frac{|k|}{K+1}\right)n^{ikT_{}}
= \sum_{k=-K}^K\left(1-\frac{|k|}{K+1}\right)\zeta(s-i kT_{})\end{aligned}\end{equation}
for $\Re s >1$, and so $\sum_n a_n n^{-s}$ has the meromorphic continuation \[A(s) = \sum_{k=-K}^K\left(1-\frac{|k|}{K+1}\right)\zeta(s-i kT_{})\] for all $s$. Clearly, $A(s)$ has a simple pole at $s=1$ with residue $1$ (from $k=0$) and it has no poles with $|\Im s|< T_{}$. Now, let $x = \exp\left(\frac{2\pi N +\delta}{T}\right)$, $N\in \mathbb{Z}_{>0}$ and $0<\delta<\pi$. Then
\begin{align} \label{10_30pm}
\sum_{n\leq x}a_n = \sum_{n\leq x}\sum_{k=-K}^K\left(1-\frac{|k|}{K+1}\right)n^{ikT_{}}  = \sum_{k=-K}^K\left(1-\frac{|k|}{K+1}\right)\sum_{n\leq x}n^{ikT_{}}.
\end{align}
By Euler--Maclaurin in degree $1$, 
\begin{equation}\label{eq:fandango}
\sum_{n\leq x}n^{ikT_{}} = \dfrac{x^{ikT_{}+1}}{ikT_{}+1} + O(|k|T_{}\log x+1).
\end{equation}
Thus, the sum on the right-hand side of \eqref{10_30pm} is
\begin{align*} 
x\cdot \sum_{|k|\leq K}\left(1-\frac{|k|}{K+1}\right)\dfrac{x^{ikT_{}}}{ikT_{}+1} + O\left(\sum_{|k|\leq K}\left(1-\frac{|k|}{K+1}\right)\left(|k|T_{}\log x+1\right)\right). 
\end{align*}
We bound  the error term easily by $O(K^2 T_{}\log x+K)$. 

Define $g(t)=\frac{1}{T} e^{(t-\delta)/T} \mathds{1}_{(-\infty,\delta]}(t)$.
Then, for any $k\ne 0$, \[\begin{aligned}\int_{-\infty}^\infty g(t) e^{i k t} dt &=
\frac{e^{-\delta/T}}{T} \int_{-\infty}^\delta e^{\left(i k + \frac{1}{T}\right) t} dt
= \frac{e^{-\delta/T}}{i k T + 1} e^{\left(i k + \frac{1}{T}\right) \delta} =
\frac{e^{i k \delta}}{i k T + 1}.
\end{aligned}\]
Since $x^{ikT_{}} = e^{ik \delta}$, we can thus write
\[\begin{aligned}\sum_{|k|\leq K}\left(1-\frac{|k|}{K+1}\right)\dfrac{x^{ikT_{}}}{ikT_{}+1} 
&= \int_{-\infty}^{\delta}g(t) F_{K}(t) dt .\end{aligned}\]
Since $F_K$ has period $2\pi$, 
$\int_{-\infty}^{\delta} g(t) F_{K}(t) dt = \int_{-\pi}^{\pi} G(t) F_K(t) dt$, where, for
$-\pi\leq t<\pi$,
\[G(t) = g(t) + \frac{e^{-\delta/T}}{T} \sum_m e^{\frac{t-2\pi m}{T}} = 
\frac{e^{(t-\delta)/T}}{T} \left(\mathds{1}_{[-\pi,\delta]}(t) + \frac{1}{e^{2\pi/T}-1}\right).\]  
In particular, $G$ is continuous at $0$, and so (\cite[III, (3.4)]{zbMATH01881986} or
\cite[Thm.~2.3]{zbMATH07503892})
\[\lim_{K\to \infty} \frac{1}{2\pi} \int_{-\pi}^{\pi} G(t) F_K(t) dt = G(0) = 
\frac{1}{T} \frac{e^{(2\pi-\delta)/T}}{e^{2\pi/T}-1}.\]
In other words, by taking $K$ sufficiently large, we can make sure that 
$ \sum_{|k|\leq K}\left(1-\frac{|k|}{K+1}\right)\frac{x^{ikT_{}}}{ikT_{}+1}$ is larger
than $\frac{2\pi}{T} \frac{e^{(2\pi -\delta)/T}}{e^{2\pi/T}-1} - \frac{\epsilon}{3}$. We fix $K$ and let
$N\to \infty$, and so $x$ goes to infinity; then the term $O(K^2 T \log x + K)$ becomes $< (\epsilon/3) x$ in absolute value for $N$ sufficiently large. Hence
\[\frac{1}{x} \sum_{n\leq x} a_n > \frac{2\pi}{T} \frac{e^{(2\pi-\delta)/T}}{e^{2\pi/T}-1}  - \frac{2\epsilon}{3}
= \left(\frac{\pi}{T}  + \frac{\pi}{T} \coth \frac{\pi}{T}\right) e^{-\frac{\delta}{T}} - \frac{2\epsilon}{3}
\]
for $N$ sufficiently large. Set $\delta$ small enough for the right side to be
$\geq \frac{\pi}{T}  + \frac{\pi}{T} \coth \frac{\pi}{T} - \epsilon$. We have proved that
$\limsup_{x\to\infty} \frac{1}{x}\sum_{n\leq x}a_n > 
\frac{\pi}{T} \coth \frac{\pi}{T} + \frac{\pi}{T} - \epsilon$.

To obtain the statement on $\liminf_{x\to \infty}$, we proceed just as above, but with
$x = \exp\left(\frac{2\pi N -\delta}{T}\right)$.
\end{proof}

{\em Comparison with the construction for bounded $a_n$.} In the companion paper \cite{Moebart},
we construct a bounded sequence $\{a_n\}_{n=1}^\infty$ such that 
$\limsup_{x\to \infty} \frac{1}{x}\sum_{n\leq x} a_n >  \tanh \frac{\pi}{2 T} - \epsilon$.
Here we just managed to construct a sequence with 
$\limsup_{x\to \infty} \left(\frac{1}{x}\sum_{n\leq x} a_n -1\right)$ greater than $\frac{\pi}{T} \coth \frac{\pi}{T} - 1 + \frac{\pi}{T} - \epsilon$, which is about twice $\tanh \frac{\pi}{2 T}$. Of course
what we have done is to show that the main result in this paper is sharp, just as our construction in \cite{Moebart} shows that the main result in 
\cite{Moebart} is sharp.

What is the difference between the construction here and the construction in \cite{Moebart}, and why does the construction here give a lower
bound that is about twice that in \cite{Moebart}?
Our sequence $\{a_n\}$ here has mean $1$ and large peaks around $\exp\left(\frac{2\pi N}{T}\right)$, $N>0$. There is thus a large imbalance in $\sum_{n\leq x} a_n$ for $x$ right after and right before each peak. If  $a_n$ is bounded, we cannot have large peaks; rather, $a_n$ can
approximate a square wave, which also has an imbalance -- but a smaller one, by about half --
right before or right after $a_n$ goes from about $1$ to about $-1$, or vice versa.

\section{Extremal approximants to the truncated exponential}\label{sec:modidon}
Our task is now to give band-limited approximations in $L^1$-norm
to a given function $I:\mathbb{R}\to \mathbb{C}$.
By ``band-limited'' we mean that our approximation is the Fourier transform
$\widehat{\varphi}$ of a function $\varphi$ supported on a compact interval (in our case, $[-1,1]$).



\vspace{0.1cm}

To be precise: let $I:\mathbb{R}\to \mathbb{C}$ be in $L^1(\R)$.
We want to find $\varphi:\mathbb{R}\to \mathbb{C}$ supported on $[-1,1]$, with 
$\varphi, \widehat{\varphi}\in L^1(\R)$, such that
\begin{equation}
\|\widehat{\varphi}-I\|_1
\end{equation}
is minimal. This is the {\em approximation} problem;
the function $\widehat{\varphi}$ here is sometimes called a {\em two-sided approximant}. If we add
the constraint that $\widehat{\varphi}-I$ is non-negative (or,
non-positive), we speak of a {\em majorization}
(or, respectively, {\em minorization}) problem; the majorant or minorant
$\widehat{\varphi}$ is called a {\em one-sided approximant}.




Let $\lambda\in \mathbb{R}\setminus \{0\}$. We will consider the functions
$I = I_\lambda$ defined in
\eqref{eq:truncexp}. For those functions, the majorization/minorization problem was solved by Graham and Vaaler \cite{zbMATH03758875}. 
Our task will be mainly to work out the rather nice Fourier transforms 
$\varphi_{\pm,\lambda}$ of the approximants.

Results in the literature are often phrased in terms 
of {\em exponential type}. An entire function $F$ is of exponential type $2\pi\Delta$, with $\Delta>0$, if
$|F(z)| \ll_\epsilon e^{(2\pi\Delta + \epsilon) |z|}$. The Paley--Wiener
theorem states that, if $\varphi:\mathbb{R}\to\mathbb{C}$
is supported on $[-\Delta,\Delta]$ and in $L^2(\mathbb{R})$, then $\widehat{\varphi}$ is entire 
and of exponential type $2\pi\Delta$;
conversely, if $F$ is
exponential type $2\pi\Delta$, and the restriction of $F$ to $\mathbb{R}$ lies
in $L^2(\mathbb{R})$, then $F$ is the Fourier transform of some $\varphi\in L^2(\mathbb{R})$ supported on $[-\Delta,\Delta]$ (\cite[\S 5]{zbMATH03026314}, \cite[Ch. XVI, Thm.~7.2]{zbMATH01881986},  or
\cite[Thm.~19.3]{zbMATH01022658}).

{\em Remark.} As we shall see, when $\lambda\to 0^+$, the optimal one-sided approximants to $I_\lambda$ tend to the optimal one-sided
  approximants to $I_0$ found by Beurling
  and rediscovered by Selberg\footnote{The result may be most familiar
to number theorists due to its use by Selberg to prove an optimal form of the large
sieve, matching Montgomery and Vaughan's. See, e.g.,
\cite[\S 9.1]{MR2647984}.}; see the comments in \cite[pp. 226]{Sellec}
  on the non-publication history. 
  This is a ``cultural'' comment, in that we will be
  working only with $\lambda\ne 0$, letting $\lambda\to 0$ only once we
  reach our applications, so as to avoid special cases.

\subsection{Graham--Vaaler's one-sided approximants and their transforms}

\begin{proposition}\label{prop:grahamvaaler}
Let $F(z)$ be an entire function of exponential type $2\pi$. Let $I_\lambda$ be as in \eqref{eq:truncexp}, where $\lambda \in \mathbb{R}\setminus \{0\}$.
\begin{enumerate}[(i)]
\item If $F(x)\geq I_\lambda(x)$ for all $x\in \R$, then
\begin{equation}\label{eq:agora1}\|F-I_\lambda\|_1 \geq \frac{1}{1-e^{-|\lambda|}}-\frac{1}{|\lambda|},\end{equation}
with equality if and only if $F =\widehat{\varphi_{+,\lambda}}$.
\item If $F(x)\leq I_\lambda(x)$ for all $x\in \R$, then
\begin{equation}\label{eq:agora2}\|F-I_\lambda\|_1 \geq \frac{1}{|\lambda|}-\frac{1}{e^{|\lambda|}-1},\end{equation}
with equality if and only if $F = \widehat{\varphi_{-,\lambda}}$. 
\end{enumerate}
Here
$\varphi_{\pm,\lambda}(t) = \varphi_{|\lambda|}^\pm(\sgn(\lambda) t)$, where,
for $\nu>0$,
\begin{equation}\label{eq:arnor1}\varphi_{\nu}^\pm(t) =
\mathds{1}_{[-1,1]}(t) \cdot (
\Phi_\nu^{\pm,\circ}(t) + \sgn(t) \Phi_\nu^{\pm,\star}(t)),
\end{equation}
\begin{equation}\label{eq:arnoros}
\Phi_\nu^{\pm,\circ}(z) = \frac{1}{2} \left(\coth \frac{w}{2} \pm 1\right),
\;\;\;\;\;
\Phi_\nu^{\pm,\star}(z) = 
\frac{i}{2\pi}\left(
\frac{\nu}{2} \coth \frac{\nu}{2} - \frac{w}{2} \coth \frac{w}{2} \pm \pi i z\right),
 \end{equation}
and $w = w(z) = -2\pi i z  + \nu$.
\end{proposition}
\begin{proof}
Apply \cite[Theorem 2]{zbMATH06384942} with $c=0$ and $\delta=1$:
if $F(x)\geq I_\lambda(x)$ for all $x$, then
\begin{align*}  \label{L1CarneiroLittmann2}
\|F-I_\lambda\|_1 & = \int_{-\infty}^\infty (F(\sgn(\lambda) u)-E_{|\lambda|}(u))du  \geq \frac{1}{1-e^{-|\lambda|}}-\frac{1}{|\lambda|},
\end{align*} 
with equality if and only if $F(u)=M_{|\lambda|}(\sgn(\lambda) u)$, where, for $\nu>0$, $M_{\nu}$ is the entire function of exponential type $2\pi$ given by
\begin{equation}\label{eq:harambe}
M_{\nu}(z)=\left(\dfrac{\sin\pi z}{\pi}\right)^2\left\{\sum_{n}\left(\dfrac{E_{\nu}(n)}{(z-n)^2}+\dfrac{E_{\nu}'(n)}{z-n}-\dfrac{E_{\nu}'(n)}{z}\right)+\dfrac{1}{z^2}\right\};\end{equation}
if $F(x)\leq I_\lambda(x)$ for all $x$, then \eqref{eq:agora2} holds instead,
with equality iff $F(u) = L_{|\lambda|}(\sgn(\lambda) u)$,
where $L_\nu(z) = 
M_\nu(z) - \left(\frac{\sin \pi z}{\pi z}\right)^2$ \cite[(3.21)]{zbMATH06384942}.
(It was already proved in \cite[Thm.~9]{zbMATH03758875} that $M_\nu$ and $L_\nu$
were the unique majorant and minorant minimizing $\|F-I_\lambda\|_1$; we
are referring to \cite{zbMATH06384942} because they give the specific values
on the right sides of \eqref{eq:agora1} and \eqref{eq:agora2}.) By
\cite[Thm.~6]{zbMATH03758875}, the restrictions of $M_\nu$ and $L_\nu$ to $\mathbb{R}$ are in $L^1\cap L^2$, since $I_\lambda$ and $\left(\frac{\sin \pi z}{\pi z}\right)^2$ are.

By the Fourier inversion formula, if $F(u) = M_{|\lambda|}(\sgn(\lambda) u)$, then
$F=\widehat{\varphi_{+,\lambda}}$, where $\varphi_{+,\lambda}(t)=
\widehat{M_{|\lambda|}}(-\sgn(\lambda) t)$; if
$F(u) = L_{|\lambda|}(\sgn(\lambda) u)$, then
$F=\widehat{\varphi_{-,\lambda}}$, where $\varphi_{-,\lambda}(t)=
\widehat{L_{|\lambda|}}(-\sgn(\lambda) t)$.
By \cite[Theorem~9]{zbMATH03919007} and $M_\nu\in L^1(\mathbb{R})$,
for $\nu>0$,
\begin{align} \label{12_51am}
\widehat{M_\nu}(t)=(1-|t|)\sum_{n\in\mathbb{Z}} M_\nu(n)e^{-2\pi i nt}+\dfrac{1}{2\pi i}\sgn(t)\sum_{n\in\mathbb{Z}} M'_\nu(n)e^{-2\pi i nt}.
\end{align} 
for all $t\in [-1,1]$.
It follows from \eqref{eq:harambe} that $M_\nu(n)=E_\nu(n)$ for all $n\ne 0$ and $M_\nu(0) = 1$. Hence
$$
\sum_{n\in\Z} M_\nu(n)e^{-2\pi i nt}=\sum_{n=0}^\infty e^{-\nu n}e^{-2\pi i nt} = \dfrac{e^{2\pi it+\nu}}{e^{2\pi it+\nu}-1} = 
\dfrac{1}{1-e^{-2\pi it-\nu}}.
$$
Again by \eqref{eq:harambe}, $M'_\nu(n)=E'_\nu(n)$ for all $n\neq 0$ and $M'_\nu(0)=-\sum_{m\neq 0}E'_\nu(m)$. Thus
\begin{align*} 
	\sum_{n\in\Z} M'_\nu(n)e^{-2\pi i nt} & = -\sum_{m}E'_\nu(m) + \sum_{n} E'_\nu(n)e^{-2\pi i nt}  \\
	& = \nu \sum_{m}e^{-\nu m} -\nu  \sum_{n} e^{-\nu n}e^{-2\pi i nt} 
    = \dfrac{\nu}{e^\nu -1} -\dfrac{\nu}{e^{2\pi it+\nu}-1}. 
\end{align*} 
Therefore, by \eqref{12_51am} we conclude that, for all $|t|\leq 1$:
\begin{equation}\label{eq:pomme1}
	\widehat{M_\nu}(t)= \dfrac{(1-|t|)e^{2\pi it+\nu}}{e^{2\pi it+\nu}-1}+\dfrac{\sgn(t) \nu}{2\pi i}\left( \dfrac{1}{e^\nu -1} -\dfrac{1}{e^{2\pi it+\nu}-1}\right),
\end{equation}

It is not hard to see from \eqref{eq:harambe} that $M_\nu$ is bounded. Since 
$M_\nu\in L^1(\mathbb{R})$, it follows that $M_\nu \in L^2(\mathbb{R})$.
The Paley--Wiener theorem then tells us that $M_\nu$ is the Fourier transform in the $L^2$ sense of a function
$g\in L^2(\mathbb{R})$ supported on $[-1,1]$. By 
\cite[Thm.~9.14]{zbMATH01022658} and $M_\nu\in L^1(\mathbb{R})$, $\widehat{M_\nu}(t) =
g(-t)$ almost everywhere. As $\widehat{M_\nu}$ is continuous (because $M_\nu\in L^1(\mathbb{R})$), it follows that $g$ has a continuous representative supported on $[-1,1]$.

Now, in general, for any $w$,
$$\frac{e^w}{e^w-1} = 
\frac{1}{2} \left(\coth \frac{w}{2} + 1\right),\;\;\;\;\;\;\;\;\;\;
\frac{1}{e^w-1} 
= \frac{1}{2} \left(\coth \frac{w}{2} - 1\right).
$$
It follows that, for $w= 2\pi i t + \nu$ and $|t|\leq 1$,
$$\begin{aligned}
\widehat{M}_\nu(t) &= (1-|t|)\cdot \frac{1}{2} \left(\coth \frac{w}{2} + 1\right)
+ \frac{\sgn(t) \nu}{2\pi i} \cdot \frac{1}{2}\left(\coth \frac{\nu}{2} - 
\coth \frac{w}{2}\right)\\
&= \frac{1}{2} \left(\coth \frac{w}{2} + 1\right) + \frac{\sgn(t)}{2\pi i}\left(
\frac{\nu}{2} \coth \frac{\nu}{2} - \frac{w}{2} \coth \frac{w}{2} - t\pi i\right).
\end{aligned}
$$
where we note that $|t| = t \sgn(t)$.
Since the Fourier transform of 
$\left(\frac{\sin \pi z}{\pi z}\right)^2$ is $(1-|t|) \cdot \mathds{1}_{[-1,1]}$,
$$\widehat{L_\nu}(t) = \widehat{M_\nu}(t) - (1-|t|) = 
\frac{1}{2} \left(\coth \frac{w}{2} - 1\right) + \frac{\sgn(t)}{2\pi i}\left(
\frac{\nu}{2} \coth \frac{\nu}{2} - \frac{w}{2} \coth \frac{w}{2} + t\pi i\right)
$$
for $|t|\leq 1$, and $\widehat{L_\nu}(t)=0$ for $|t|>1$.

Finally, we define $\varphi_\nu^+(t) = \widehat{M_\nu}(-t)$,
$\varphi_\nu^-(t) = \widehat{L_\nu}(-t)$, and obtain
$$\varphi_\nu^\pm(t) = \mathds{1}_{[-1,1]}(t)\cdot \left(
\frac{1}{2} \left(\coth \frac{w}{2} \pm 1\right) - \frac{\sgn(t)}{2\pi i}\left(
\frac{\nu}{2} \coth \frac{\nu}{2} - \frac{w}{2} \coth \frac{w}{2} \pm t\pi i\right)
\right)
$$
for $w = -2\pi i t + \nu$. Then $\varphi_{\pm,\lambda}(t) = 
\varphi_{|\lambda|}^\pm(\sgn(\lambda) t)$.
We may write $\varphi_\nu^\pm(t)$ as \eqref{eq:arnor1} and \eqref{eq:arnoros}.
\end{proof}

\begin{figure}[ht]
\begin{minipage}[c]{0.45\linewidth}
\centering
\includegraphics[width=\textwidth]{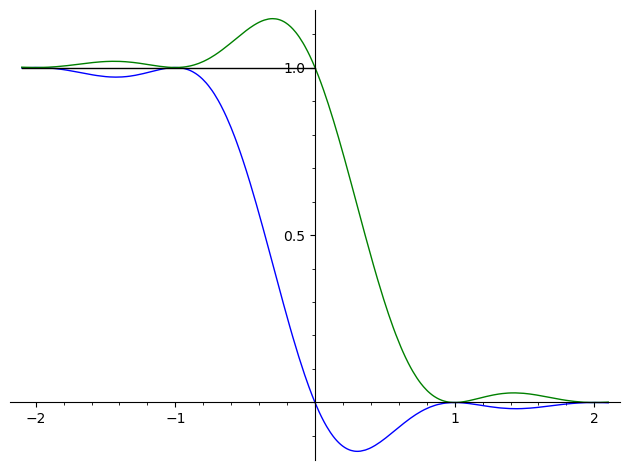}
\caption{Beurling--Selberg majorant and minorant of $\mathds{1}_{(-\infty,0]}(u)$}
\end{minipage}
\begin{minipage}[c]{0.45\linewidth}
\centering
\includegraphics[width=\textwidth]{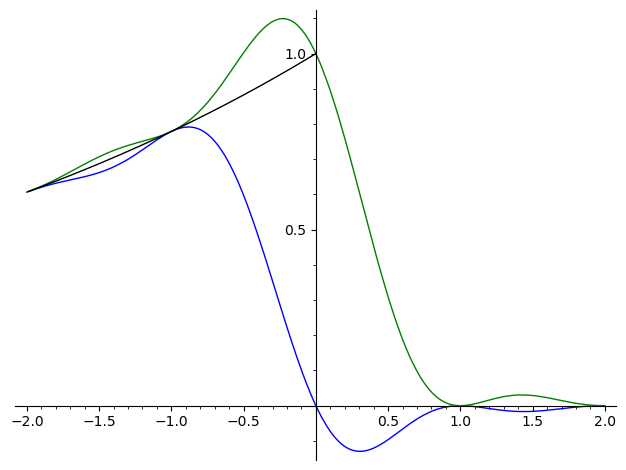}
\caption{Graham--Vaaler majorant and minorant of $\mathds{1}_{(-\infty,0]}(u)\cdot e^{\lambda u}$ for $\lambda = 1/4$}
\label{fig:grahamvaaler}
\end{minipage}
\end{figure}

 We may also write, for $\xi = \pm 1$,
\begin{equation}\label{eq:arnoroid}(\Phi_{\nu}^{\pm,\circ}+ \xi \Phi_\nu^{\pm,\star})(z) = 
\frac{i \xi}{2\pi} \left(\frac{\nu}{2} \coth \frac{\nu}{2} -\frac{w(z-\xi)}{2}\coth \frac{w(z-\xi)}{2} \pm \pi i (z-\xi)\right),\end{equation}
which follows from
\eqref{eq:arnoros}: $\coth \frac{w(z\pm 1)}{2} = \coth\left(\frac{w(z)}{2}\pm \pi i\right) = \coth \frac{w(z)}{2}$.
\begin{remark}
Vaaler finds the optimal one-sided approximants for $\sgn(x)$ of exponential type $2\pi$ \cite[Thm.~8]{zbMATH03919007}: the optimal majorant $B(x)$ is defined in \cite[Eq (1.1)] {zbMATH03919007} as
\begin{equation}\label{eq:fibel}B(x) = 
\left(\frac{\sin \pi x}{\pi}\right)^2\left(\sum_{n\in \mathbb{Z}}  \frac{\sgn(n)}{(x-n)^2}+\dfrac{2}{x} + \dfrac{1}{x^2}\right).\end{equation}
Since $\sgn(x)$ is odd, the optimal minorant is given by $-B(-x)$. 

Thus, the optimal majorant for $\mathds{1}_{[0,\infty)}(x)$ is  $\frac{B(x)+1}{2}$. By the formula
$\frac{\pi^2}{\sin^2 \pi x} = \sum_{n\in \Z}\frac{1}{(x-n)^2}$, 
\begin{equation}\label{eq:gawat}
\dfrac{B(x)+1}{2}=\left(\frac{\sin \pi x}{\pi}\right)^2\left(\sum_{n}  \frac{1}{(x-n)^2}+\dfrac{1}{x}+\dfrac{1}{x^2}\right).
\end{equation}
By \eqref{eq:harambe}, if $\lambda>0$, the optimal majorant of exponential type $2\pi$ for
$\mathds{1}_{[0,\infty)}(x)\, e^{-\lambda x}$  is
\begin{equation}\label{eq:gowot}
\widehat{\varphi_{+,\lambda}}(x)
= \left(\dfrac{\sin\pi x}{\pi}\right)^2\left\{\sum_{n}\left(\dfrac{e^{-\lambda n}}{(x-n)^2}-\dfrac{\lambda e^{-\lambda n}}{x-n}+\dfrac{\lambda e^{-\lambda n}}{x}\right)+\dfrac{1}{x^2}\right\}.
\end{equation}
For $x\in \Z$, both \eqref{eq:gawat} and \eqref{eq:gowot} have their values
defined by continuity, and equal the function they majorize:
$\widehat{\varphi_{+,\lambda}}(n) = \frac{B(n)+1}{2} = 0$ for $n<0$,
and, for $n\geq 0$, as $\lambda\to 0^+$,
$\widehat{\varphi_{+,\lambda}}(n) = e^{-\lambda n}\to 1 = \frac{B(n)+1}{2}$.

Fix $x\notin \Z$. 
As $\lambda\to 0^+$, $\sum_n \lambda e^{-\lambda n} = \frac{\lambda}{e^\lambda - 1} \to 1$, and,
by monotone convergence, $ \sum_{n}\frac{e^{-\lambda n}}{(x-n)^2} \to \sum_{n}\frac{1}{(x-n)^2}$. For any $N$,\; $\sum_{n\leq x+N} \frac{e^{-\lambda n}}{x-n}$ is bounded
independently of $\lambda\geq 0$, and so
$\sum_{n\leq x+N} \frac{\lambda e^{-\lambda n}}{x-n}$ tends to $0$. 
The remainder $\sum_{n> x+N} \frac{\lambda e^{-\lambda n}}{x-n}$ is bounded by
$\frac{1}{N} \sum_n \lambda e^{-\lambda n} = \frac{1}{N} \frac{\lambda}{e^\lambda - 1}
< \frac{1}{N}$. Hence, $\lim_{\lambda \to 0^+}
\sum_{n\leq x+N} \frac{\lambda e^{-\lambda n}}{x-n} = 0$. We conclude that 
$$\lim_{\lambda\to 0^+} \widehat{\varphi_{+,\lambda}}(x) = \frac{B(x)+1}{2}.$$
Moreover, our argument gives uniform convergence on compact intervals. The same argument gives $\lim_{\lambda\to 0^+} \widehat{\varphi_{-,\lambda}}(x) = \frac{-B(-x)+1}{2}$,
or else we can deduce that from $-B(-x) = B(x) - 2 \left(\frac{\sin \pi x}{\pi x}\right)^2$ and $ \widehat{\varphi_{-,\lambda}}(x) =
\widehat{\varphi_{+,\lambda}}(x) - \left(\frac{\sin \pi x}{\pi x}\right)^2$.

\end{remark}

\subsection{Useful bounds and properties}
\begin{lemma}\label{lem:cothder}
If $|\Im z|\leq \frac{\pi}{4}$, then $|(z\coth z)'|< 1$. If $|\Im z|\leq \frac{\pi}{2}$, then $|(z\coth z)'|\leq |z|$.
\end{lemma}
\begin{proof}
Since $z\coth(z)$ is regular at $0$ and an even function,
we see 
that $f(z) := (z \coth z)'$ and $f(z)/z$ are regular at $0$, and hence analytic on the strip $|\Im z|\leq \frac{\pi}{2}$.
We see from 
 $f(z) = \coth z - z \csch^2 z$ that $f(z)$ has at most exponential
 growth as $\Re z\to \pm \infty$ within the strip.
Hence, by Phragmén-Lindelöf, 
it is enough to verify the inequalities
$|f(z)/z|\leq 1$ for $\Im z = \pm \frac{\pi}{2}$ and 
$|f(z)|\leq 1$ for $\Im z = \pm \frac{\pi}{4}$; 
by complex conjugation, it suffices to check them for
$\Im z = \frac{\pi}{2}$ and $\Im z =\frac{\pi}{4}$.

By the above, $f(z) = \frac{(\sinh 2 z)/2 - z}{\sinh^2 z}$. 
Now, for $z = x + i \frac{\pi}{4}$ with $x\in \mathbb{R}$, we have $\sinh 2 z = i \cosh 2 x$ and $\sinh^2 z = - \frac{1}{2} + \frac{i}{2} \sinh 2 x$, and so $|f(z)|^2 = \frac{\left(\cosh 2 x - \pi/2\right)^2 + 4 x^2}{1 + \sinh^2 2 x}$. 
By $1 + \sinh^2 2 x = \cosh^2 2 x$, $$|f(z)|^2 = 
1 - \frac{\pi \cosh 2 x - \pi^2/4 - 4 x^2} {\cosh^2 2 x}.$$
Since $\cosh 2 x = 1 + 2\sinh^2 x\geq 1 + 2 x^2$,
$\pi>\frac{\pi^2}{4}$ and $2 \pi>4$, the numerator here is positive. We conclude that $|f(z)|^2<1$ for $z = x+i\frac{\pi}{4}$, as was desired.

For $z = x+ i\frac{\pi}{2}$ with $x\in \R$, we have $\coth z=\tanh x$ and $\csch^2 z=-\sech^2x$. Then $|f(z)|^2 = (\tanh x + x\sech^2x)^2 + (\frac{\pi}{2}\sech^2x)^2$. Since $\sech^2x -1 = -\tanh^2x$, this
is equal to
$$\tanh^2x\sech x\left(\cosh x+2x\csch x -|z|^2(\sech x+ \cosh x)\right)+|z|^2,$$
Since $|z|^2 \geq \frac{\pi^2}{4}>2$, it suffices to show that
$2x\csch x - 2\sech x-\cosh x\leq 0$ for all $x\in\R$; by parity,
it is enough to check all $x\geq 0$. The statement is then equivalent to $g(x)=2x-2\tanh x-
\sinh x \cosh x\leq 0$, since $\sinh(x)\geq 0$. That which follows
from $g'(x)= 2 \tanh^2 x - \cosh^2 x - \sinh^2 x =  
-2\sinh^2 x\tanh^2x-1 \leq 0$
(by $1-\cosh^2x =- \sinh^2x$) and $g(0)=0$. 

\end{proof} 

\begin{lemma}\label{lem:benvenuto2}
Let $\Phi_{\nu}^{\pm,\circ}(z)$ and $\Phi_{\nu}^{\pm,\star}(z)$ be as in \eqref{eq:arnoros} for $\nu>0$. Then
\begin{itemize}[-]
    \item $\Phi_{\nu}^{\pm,\circ}(z)$ is a meromorphic function whose poles, all of
them simple, are at $n -  \frac{i\nu}{2\pi}$, $n\in \mathbb{Z}$; the residue at every pole
is $\frac{i}{2\pi}$. Moreover, $\overline{\Phi_{\nu}^{\pm,\circ}(z)}=\Phi_{\nu}^{\pm,\circ}(-\overline{z})$.
\item $\Phi_{\nu}^{\pm,\star}(z)$ is a meromorphic function whose poles, all of
them simple, are at $n -  \frac{i\nu}{2\pi}$, $n\in \mathbb{Z}\setminus\{0\}$;
the residue at $n- \frac{i\nu}{2\pi}$ is $-\frac{i n}{2\pi}$. Moreover, $\overline{\Phi_{\nu}^{\pm,\star}(z)}=-\Phi_{\nu}^{\pm,\star}(-\overline{z})$.
\end{itemize}
On every region $\{z: \Im z\geq c\}$, $c>-\frac{\nu}{2\pi}$,
or $\{z:\Im z\leq c\}$, $c<-\frac{\nu}{2\pi}$, 
the function $\Phi_\nu^{\pm,\circ}(z)$ is bounded and 
$\Phi_\nu^{\pm,\star}(z) = O(|z|+1)$. Moreover, these bounds hold uniformly
for all $\nu$ in an interval $[\nu_0,\nu_1]$, with conditions $c>-\frac{\nu_0}{2\pi}$,
$c<-\frac{\nu_1}{2\pi}$, respectively.

We have $\Phi_{\nu}^{\sigma,\star}(0) = 0$.
For $z$ with $0\leq \Re z\leq \frac{1}{4}$, and
for either sign $\sigma=\pm$,
$$\left|(\Phi_\nu^{\pm,\star})'(z)\right|\leq 1,\quad\quad
\left|\Phi_{\nu}^{\sigma,\star}(\pm z)\right|\leq |z|,\quad\quad
\left|(\Phi_{\nu}^{\sigma,\circ}\pm \Phi_{\nu}^{\sigma,\star})(\pm 1 \mp z)\right|\leq |z|.$$
Moreover, for $z$ purely imaginary, $(\Phi_\nu^{\sigma,\star})'(\pm z)$, which
is purely real, is of constant sign.

Note that $\Phi_\nu^{\sigma,\circ}(z)\pm \Phi_\nu^{\sigma,\star}(z)$
is regular at $\pm 1 - \frac{i\nu}{2\pi}$, since the residues cancel out.
\end{lemma}
Our convention is that all signs denoted by $\pm$ in the same equation are the same, $\mp$ is the opposite sign, and $\sigma$ denotes a sign that may or not may be the same.
\begin{proof}
The statements on poles and residues follow directly from \eqref{eq:arnoros};
so do the statements on $\overline{\Phi_{\nu}^{\pm,\circ}(z)}$ and
$\overline{\Phi_{\nu}^{\pm,\star}(z)}$.
The statements on the boundedness of $\Phi_\nu^{\sigma,\circ}(z)$ and the growth
of $\Phi_\nu^{\sigma,\star}(z)$ follow from \eqref{eq:arnoros} and the
fact that $\coth(w)$ is bounded on $\Re w\geq c$ for $c>0$ arbitrary and on
$\Re w \leq c$ for $c<0$ arbitrary. Since $|\Phi_\nu^{\sigma,\star}(-z)| = 
 |\Phi_\nu^{\sigma,\star}(\overline{z})|$ and
$\left|(\Phi_{\nu}^{\sigma,\circ}-\Phi_{\nu}^{\sigma,\star})(-1+z)\right|
= \left|(\Phi_{\nu}^{\sigma,\circ}+\Phi_{\nu}^{\sigma,\star})(1-\overline{z})\right|$,
it is left to check that 
$\left|\Phi_{\nu}^{\sigma,\star}(z)\right|\leq |z|$ and
$\left|(\Phi_{\nu}^{\sigma,\circ}+ \Phi_{\nu}^{\sigma,\star})(1-z)\right|\leq |z|$.

By \eqref{eq:arnoros}, $\Phi_\nu^{\pm,\star}(0) = 0$ and
$(\Phi_\nu^{\pm,\star})'(z) = -\frac{d}{dw} \left(\frac{w}{2} \coth \frac{w}{2}\right) \mp 1/2$ at $w = -2\pi i z +\nu$. Hence, for $0\leq \Re z\leq 1/4$, by Lemma \ref{lem:cothder},
$\left|(\Phi_\nu^{\pm,\star})'(z)\right|\leq 1$, and so
$\left|(\Phi_\nu^{\pm,\star})(z)\right|\leq |z|$; moreover,
$(\Phi_\nu^{\pm,\star})'(z)$ does not change sign for $z$ purely imaginary, as thn $w$ is real, and the term
$\mp \frac{1}{2}$ always dominates.
By \eqref{eq:arnoroid},
$(\Phi_{\nu}^{\pm,\circ}+ \Phi_{\nu}^{\pm,\star})(1)=0$ and
$(\Phi_{\nu}^{\pm,\circ}+ \Phi_{\nu}^{\pm,\star})'(z) =
-\frac{d}{dw} \left(\frac{w}{2} \coth \frac{w}{2}\right)
\mp \frac{1}{2}$ at $w = - 2\pi i (z-1) + \nu$. Hence, again by Lemma \ref{lem:cothder},
for $0\leq \Re z\leq \frac{1}{4}$,
$\left|(\Phi_{\nu}^{\sigma,\circ}+ \Phi_{\nu}^{\sigma,\star})(1-z)\right|\leq |z|$.
\end{proof}

\begin{lemma}\label{lem:omlet}
    For $z\in \mathbb{C}$, 
    $\lambda\in \mathbb{R}\setminus \{0\}$,
    define
    $$\Phi_\lambda^\pm(z) =  \Phi_{|\lambda|}^{\pm,\circ}(\sgn(\lambda) z) + \sgn(\lambda)
\sgn(\Re z) \Phi_{|\lambda|}^{\pm,\star}(\sgn(\lambda) z),$$
where $\Phi_{|\lambda|}^{\pm,\circ}$, $\Phi_{|\lambda|}^{\pm,\star}$ are as in
\eqref{eq:arnoros}, and $\sgn(0)=0$. 
Let $T>0$, and let $z(s) = \frac{s-1}{i T}$.

Then, for $s\in \mathbb{C}$,
\begin{equation}\label{eq:Phicong}
\Phi_\lambda^{\pm}(z(\overline{s}))
= \overline{\Phi_\lambda^{\pm}(z(s))}.
\end{equation}
Let $\sigma\in \mathbb{R}\setminus \{1\}$. Let
$\lambda = \frac{2\pi}{T} (\sigma-1)$ and 
 write $\theta(s) = 1 - \frac{s-\sigma}{i T}$. If $\Im s >0$,
\begin{equation}\label{eq:explorm}
\Phi_\lambda^{\pm}(z(s)) = i \sgn(\lambda)
\left(- \frac{\theta(s)}{2} \cot(\pi \theta(s)) + \frac{\theta(1+i T)}{2} \cot(\pi \theta(1+ i T))\right)
 \pm \frac{1 - z(s)}{2}.
\end{equation}
\end{lemma}
\begin{proof}
When we evaluate $\Phi_\lambda^\pm$ at $z(s)$,
we evaluate $\Phi_{|\lambda|}^{\pm,\circ}$ and
$\Phi_{|\lambda|}^{\pm,\star}$ at $\sgn(\lambda) z(s)$, and so 
the variable $w$ in \eqref{eq:arnoros} is given by
\begin{equation}\label{eq:firmow}\begin{aligned} w &= 
-2\pi i \sgn(\lambda) \frac{s-1}{i T} + |\lambda|
= \sgn(\lambda) \left(-\frac{2\pi}{T} (s-1) + \lambda\right) = - \sgn(\lambda) \frac{2\pi}{T} (s- \sigma).
\end{aligned}\end{equation}
In particular, when we conjugate $s$, we conjugate $w$.
We thus see from \eqref{eq:arnoros} that 
\begin{equation}\label{eq:mechner}\Phi_{|\lambda|}^{\pm,\circ}(\sgn(\lambda) z(\overline{s}))
= \overline{\Phi_{|\lambda|}^{\pm,\circ}(\sgn(\lambda) z(s))},\;\;\;\;\;\;\;
\Phi_{|\lambda|}^{\pm,\star}(\sgn(\lambda) z(\overline{s}))
= - \overline{\Phi_{|\lambda|}^{\pm,\star}(\sgn(\lambda) z(s))},\end{equation}
and thus, since $\sgn(\Re z(\overline{s})) = - \sgn(\Re z(s))$,
\eqref{eq:Phicong} holds.

If $\Im s > 0$,
\begin{equation}\label{eq:kalinka}\begin{aligned}\Phi_\lambda^\pm(z(s)) &=  \Phi_{|\lambda|}^{\pm,\circ}(\sgn(\lambda) z(s)) + 
\sgn(\lambda) \Phi_{|\lambda|}^{\pm,\star}(\sgn(\lambda) z(s))
\end{aligned}\end{equation}
because $\Im s > 0$ implies $\Re z(s)>0$. Since
$\coth$ is an odd function, \eqref{eq:arnoros} and
\eqref{eq:firmow} give us
$$\Phi_{|\lambda|}^{\pm,\circ}(\sgn(\lambda) z(s)) =
\frac{1}{2} \left(-\sgn(\lambda) \coth \frac{\pi (s-\sigma)}{T}
\pm 1\right),$$
$$\Phi_{|\lambda|}^{\pm,\star}(\sgn(\lambda) z(s)) =
\frac{i}{2\pi} \left(\frac{\lambda}{2} \coth \frac{\lambda}{2} - \frac{\pi (s-\sigma)}{T} \coth
 \frac{\pi (s-\sigma)}{T} \pm \sgn(\lambda) \pi i z(s)\right).$$
 Thus, for $\Im s> 0$, \eqref{eq:kalinka} gives us
 $$\Phi_\lambda^\pm(z(s)) = 
-\sgn(\lambda)\left( \left(\frac{i (s-\sigma)}{2 T} + \frac{1}{2}\right)
 \coth \frac{\pi (s-\sigma)}{T} - \frac{i \lambda}{4 \pi}
 \coth \frac{\lambda}{2}\right)
 \pm \frac{1- z(s)}{2}.
 $$
So, by $\coth u = - i \cot(u/i)$, $\coth(-u) = - \coth u$,
$\cot(\pi-u) = -\cot u$
and $\theta(s) = 1 -  \frac{s-\sigma}{i T}$,
 $$\Phi_\lambda^\pm(z(s)) = 
i \sgn(\lambda)\left( - \frac{\theta(s)}{2}
 \cot (\pi \theta(s)) - \frac{i \lambda}{4 \pi}
 \cot \frac{\lambda}{2 i}\right)
 \pm \frac{1- z(s)}{2}.
 $$
 Since
 $\theta(1+ i T) = \frac{\sigma-1}{i T} = \frac{\lambda}{2\pi i}$,
we have $\cot \frac{\lambda}{2 i} = 
 \cot (\pi \theta(1+ i T))$. 
\end{proof}

\section{Shifting contours}\label{sec:contour}


It is now time to shift our contours of integration to $\Re s = -\infty$.
For the sake of clarity, we will do our contour-shifting in some generality.
We will consider an integral of the form
$$
\int_{1-iT}^{1+iT} (G^\circ(s) + \sgn(\Im{s})G^\star(s))x^sds,
$$
where $G^\circ$ and $G^\star$ are meromorphic, and $G^\star(\overline{s}) = - \overline{G^\star(s)}$. Shifting the contour to the left for  $G^\circ(s)$ is straightforward; when we shift it to the left for $G^\star(s)$,
we will stop short of the real axis, on either side. We do not separate $G^\circ(s)$ and $G^\star(s)$ from the start because they
may have poles on the $\pm 1 +i \mathbb{R}$ that $G^\circ(s)+\sgn(\Im s) G^\star(s)$ does not have.

In the end, we obtain a contour $\mathcal{C}_\infty$ consisting of a straight path from $1-i T$ to $-\infty -i T$ and another from $-\infty + i T$ to $1 + i T$, and also a contour
$\mathcal{C}$ going from $1$ to $\Re s = -\infty$; on the latter contour, we integrate only
$G^\star(s)$, not $G^\circ(s)$.

``Shifting a contour to $\Re s = -\infty$'' truly means
shifting it first to a vertical line, and then to another further to the left, etc.,
taking care that the contour never goes through poles. We write 
    \begin{equation}\label{eq:defL}
      L \;=\; \bigcup_{n=1}^{\infty}\!\bigl(\sigma_n + i[-T,T]\bigr)
    \end{equation}
for these lines; here $1>\sigma_1>\sigma_2>\dotsc$ is a sequence of our choice, tending to $-\infty$ as $n\to \infty$.


\begin{lemma}\label{lem:kolobrz2}
Let $G(s) = G^\circ(s) + \sgn(\Im{s}) G^\star(s)$, where
$G^\circ(s)$ and $G^\star(s)$ are meromorphic functions on $R = (-\infty,1] + i [-T,T]$. 

 Let $\mathcal{C}$ be an admissible contour contained in $(-\infty,1]+i [0,T]$, 
 going from 
$1$ to $\Re s = -\infty$, and let $R_{\mathcal{C}}$ be the closed subregion of $R$ between 
$\mathcal{C}$ and $\overline{\mathcal{C}}$. 
Assume that, for some $x_0\geq 1$, $G(s) x_0^s$ is bounded on 
$\partial R$, and both
$G^\circ(s) x_0^s$ and $G^\star(s) x_0^s$ are bounded on $L$ and on $\mathcal{C}$, where
$L$ is as in \eqref{eq:defL}. 

Assume as well that
$G^\star(\overline{s}) = -\overline{G^\star(s)}$.
Then, for any $x>x_0$,
\begin{equation}\label{eq:runada}\begin{aligned}
\!\!\frac{1}{2\pi i}\int_{1-iT}^{1+iT} G(s)x^sds
 =  &\frac{1}{2\pi i} \int_{\mathcal{C}_\infty} G(s) x^s ds +
 \frac{1}{\pi} \Im \int_{\mathcal{C}}  G^\star(s) x^s ds
 \\
 + &\sum_{\substack{\text{$\rho$ a pole of $G$}\\ \rho\in R\setminus R_{\mathcal{C}}}}
\Res\limits_{s=\rho} G(s) x^s +
\sum_{\substack{\text{$\rho$ a pole of $G^\circ$}\\ \rho\in R_{\mathcal{C}}}} 
\Res\limits_{s=\rho} G^\circ(s) x^s
\end{aligned}\end{equation}
where $\mathcal{C}_\infty$ is a straight path from $1-iT$ to
$-\infty- i T$, and another from $-\infty +i T$ to $1+iT$.
\end{lemma}
Here $\overline{\mathcal{C}}$ means the image of $\mathcal{C}$ under
complex conjugation.
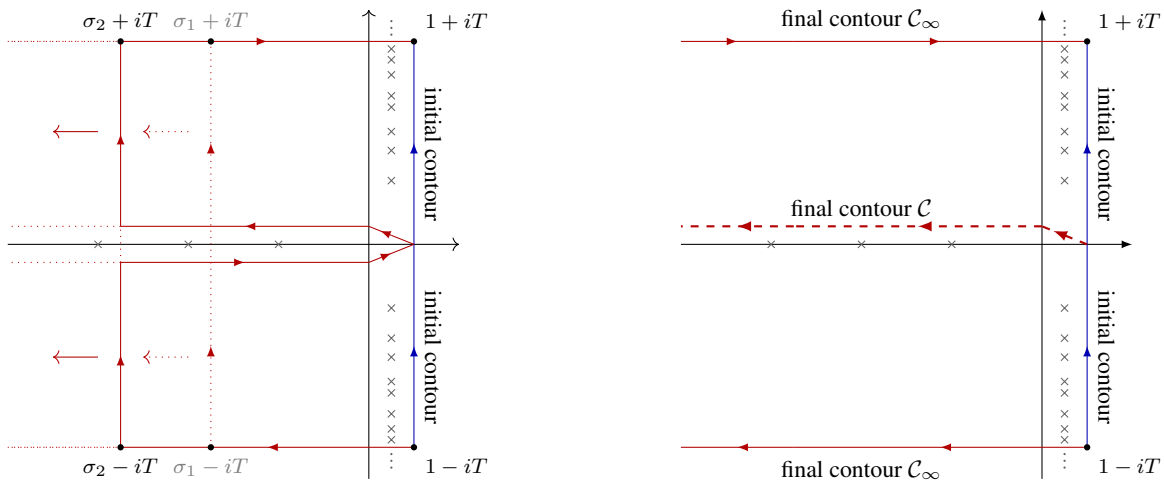
\begin{figure}[ht]
  \centering
  \begin{minipage}{0.45\textwidth}
    \centering
\begin{tikzpicture}[scale=0.6,
                    arrow/.style={postaction={decorate},
                  decoration={markings, mark=at position 0.5 with {\arrow{>}}}}]   
\useasboundingbox (-9,-6) rectangle (3,6);

\draw[->] (-8,0) -- (2,0) node[above] {};
\draw[->] (0,-5.2) -- (0,5.2) node[left] {};


\draw[darkblue,contour] (1,-4.5) -- (1,0);
\draw[darkblue,contour] (1,0) -- (1,4.5);

\draw[darkred,dotted] (-8,4.5) -- (-5.5,4.5);
\draw[darkred,dotted] (-8,-4.5) -- (-5.5,-4.5);

\draw[darkred,contour] (1,-4.5)  -- (-5.5,-4.5);
\draw[darkred,contour] (-5.5,-4.5) -- (-5.5,-0.4);
\draw[darkred,dotted] (-5.5,-0.4) -- (-5.5,0.4);
\draw[darkred,contour] (-5.5,0.4) -- (-5.5,4.5);
\draw[darkred,contour] (-5.5,4.5) -- (1,4.5);
\draw[darkred,dotted,contour] (-3.5,0) -- (-3.5,4.5);
\draw[darkred,dotted,contour] (-3.5,-4.5) -- (-3.5,0);

\draw[darkred,dotted] (-5.5,4.5) -- (-8,4.5);
\draw[darkred,dotted] (-5.5,-4.5) -- (-8,-4.5);

\draw[darkred,motiondotted] (-4,-2.5) -- (-5,-2.5);
\draw[darkred,motiondotted] (-4,2.5) -- (-5,2.5);
\draw[darkred,motion] (-6,-2.5) -- (-7,-2.5);
\draw[darkred,motion] (-6,2.5) -- (-7,2.5);

\fill (1,4.5) circle (.07) node[above right] {\tiny $1+iT$};
\fill (1,-4.5) circle (.07) node[below right] {\tiny $1-iT$};
\fill (-3.5,4.5) circle (.07) node[above] {\textcolor{gray}{\tiny $\sigma_1+iT$}};
\fill (-3.5,-4.5) circle (.07) node[below] {\textcolor{gray}{\tiny $\sigma_1-iT$}};
\fill (-5.5,4.5) circle (.07) node[above] {\tiny $\sigma_2+iT$};
\fill (-5.5,-4.5) circle (.07) node[below] {\tiny $\sigma_2-iT$};

\foreach \x in {-2,-4,-6} { \draw[very thin,darkgreen]
  ({\x+.08},.08) -- ({\x-.08},-.08) ({\x+.08},-.08) -- ({\x-.08},.08); }
\foreach \y in {1.413,2.0802,2.501,3.042,3.294,3.759,4.092,4.333} {
    \draw[darkgreen]
        (.58,{\y+.08}) -- (.42,{\y-.08}) (.42,{\y+.08}) -- (.58,{\y-.08})
        (.58,{-\y+.08}) -- (.42,{-\y-.08}) (.42,{-\y+.08}) -- (.58,{-\y-.08}); }
\foreach \y in {4.65,4.8,4.95} {
    \fill[darkgreen] (.5,\y) circle (.025) (.5,{-\y}) circle (.025); }
\node[rotate=270, anchor=south] at (1,2) {\scriptsize initial contour};
\node[rotate=270, anchor=south] at (1,-2.5) {\scriptsize initial contour};

\draw[darkred] (1,0) -- (0.5,0.2);
\draw[darkred,contour]  (0.5,0.2) -- (0,0.4);
\draw[darkred,contour] (0,0.4) -- (-5.5,0.4);
\draw[darkred,dotted] (-5.5,0.4) -- (-8,0.4);

\draw[darkred,contour]  (0,-0.4) -- (1,0);
\draw[darkred,contour] (-5.5,-0.4) -- (0,-0.4);
\draw[darkred,dotted] (-5.5,-0.4) -- (-8,-0.4);

\end{tikzpicture}
\end{minipage}
\hspace{40pt}
\begin{minipage}{0.45\textwidth}
\begin{tikzpicture}[>=latex,scale=0.6,
                    arrow/.style={postaction={decorate},
                  decoration={markings, mark=at position 0.5 with {\arrow{>}}}}]   
\useasboundingbox (-9,-6) rectangle (3,6);
\draw[->] (-8,0) -- (2,0) node[above] {};
\draw[->] (0,-5.2) -- (0,5.2) node[left] {};
\draw (-2,2.5) node {};
\draw[darkblue,arrow] (1,-4.5) -- (1,0);
\draw[darkblue,arrow] (1,0) -- (1,4.5);
\draw[darkred,arrow] (-8,4.5) -- (-5.5,4.5);
\draw[darkred,arrow] (-5.5,-4.5) -- (-8,-4.5);
\draw[darkred,arrow] (1,-4.5) -- (-5.5,-4.5);
\draw[darkred,arrow] (-5.5,4.5) -- (1,4.5);
\fill (1,4.5) circle (.07) node[above right] {\tiny $1+iT$};
\fill (1,-4.5) circle (.07) node[below right] {\tiny $1-iT$};
\foreach \x in {-2,-4,-6} { \draw[thin,darkgreen]
  ({\x+.08},.08) -- ({\x-.08},-.08) ({\x+.08},-.08) -- ({\x-.08},.08); }
\foreach \y in {1.413,2.0802,2.501,3.042,3.294,3.759,4.092,4.333} {
    \draw[darkgreen]
        (.58,{\y+.08}) -- (.42,{\y-.08}) (.42,{\y+.08}) -- (.58,{\y-.08})
        (.58,{-\y+.08}) -- (.42,{-\y-.08}) (.42,{-\y+.08}) -- (.58,{-\y-.08}); }
\foreach \y in {4.65,4.8,4.95} {
    \fill[darkgreen] (.5,\y) circle (.025) (.5,{-\y}) circle (.025); }
\node[rotate=270, anchor=south] at (1,2) {\scriptsize initial contour};
\node[rotate=270, anchor=south] at (1,-2.5) {\scriptsize initial contour};
\node at (-4,5) {\scriptsize final contour $\mathcal{C}_\infty$};
\node at (-4,-5) {\scriptsize final contour $\mathcal{C}_\infty$};    

\draw[thick,dashed,darkred] (1,0) -- (0.5,0.2);
\draw[thick,dashed,darkred,contour]  (0.5,0.2) -- (0,0.4);
\draw[thick,dashed,darkred,contour] (0,0.4) -- (-5.5,0.4);
\draw[thick,dashed,darkred,contour] (-5.5,0.4) -- (-8,0.4);
\node at (-4,0.8) {\scriptsize final contour $\mathcal{C}$};
\end{tikzpicture}
\end{minipage}
\captionof{figure}{Lemma \ref{lem:kolobrz2}: contour-shifting and result. Only $G^\star(s)$ is integrated on $\mathcal{C}$.}\label{fig:kolobrz2} 
\end{figure}

\begin{proof}
We start by separating the integral on the left side of \eqref{eq:runada} into
an integral from $1$ to $1+iT$ and an integral from $1-i T$ to $1$.
For $n\geq 1$, let $\mathcal{C}_n^+$ be a contour following
$\mathcal{C}$ leftwards up to $\Re s =\sigma_n$, and then going
upwards up to $\Im s = T$, and then to the right up to
$s = 1 + i T$. Then
\begin{align*}
\dfrac{1}{2\pi i}\int_{1}^{1+iT}G(s) x^s ds = \frac{1}{2\pi i}\int_{\mathcal{C}_{n}^+} G(s) x^s ds + \sum_{\substack{\text{$\rho$ a pole of $G$}\\ \rho\in R^+,\;\Re \rho > \sigma_n}}   \Res\limits_{s=\rho} G(s) x^s,
\end{align*}
where we write $R^+$ for the subregion of $R$ above $\mathcal{C}$.

Similarly, for $\mathcal{C}_n^-$ equal to the complex conjugate
of $\mathcal{C}_n^+$ traversed backwards (go on a straight line from $1-iT$ to $\sigma_n-i T$, then upwards up to $\mathcal{C}$, then follow $\mathcal{C}$ backwards up to $1$),
\begin{align*} 
\dfrac{1}{2\pi i}\int_{1-iT}^1 G(s) x^s ds = \frac{1}{2\pi i}\int_{\mathcal{C}_{n}^-} G(s) x^s ds + \sum_{\substack{\text{$\rho$ a pole of $G$}\\ \rho\in \overline{R^+},\;\Re \rho > \sigma_n}}  \Res\limits_{s=\rho} G(s) x^s,
\end{align*}
where $\overline{R^+} = \{\overline{s}: s\in R^+\}$, i.e., $\overline{R^+}$ is
the subregion of $R$ below $\overline{\mathcal{C}}$.

We split
$\mathcal{C}_n^+$ into a horizontal contour from $\sigma_n+iT$ to $1+ i T$, and
the rest $\mathcal{C}_{n,1}^+$. Similarly, we separate
$\mathcal{C}_n^-$ into a horizontal contour from $1-iT$ to $\sigma_n - i T$, and 
$\mathcal{C}_{n,1}^-$.
For $G^\circ(s)$, we can just shift the union of the contours
$\mathcal{C}_{n,1}^+$, $\mathcal{C}_{n,1}^-$ 
wholly to
$\Re s = \sigma_n$:
$$\frac{1}{2\pi i} \left(\int_{\mathcal{C}_{n,1}^+} + 
\int_{\mathcal{C}_{n,1}^-}\right)
G^\circ(s) x^s ds 
= \frac{1}{2\pi i} \int_{\sigma_n-i T}^{\sigma_n + i T} 
G^\circ(s) x^s ds +
\sum_{\substack{\text{$\rho$ a pole of $G^\circ$}\\ \rho\in R_{\mathcal{C}},\;\Re \rho > \sigma_n}} 
 \Res\limits_{s=\rho} G^\circ(s) x^s.
$$

For $G^\star(s)$, we simply note that, since $G^\star(\overline{s}) = 
- \overline{G^\star(s)}$, 
$$\begin{aligned}
\int_{\mathcal{C}_{n,1}^+} G^\star(s) x^s ds -
\int_{\mathcal{C}_{n,1}^-} G^\star(s) x^s ds &= 
\int_{\mathcal{C}_{n,1}^+} (G^\star(s) x^s ds +G^\star(\overline{s}) x^{\overline{s}} d\overline{s})  =
2 i
\Im \int_{\mathcal{C}_{n,1}^+} G^\star(s) x^s ds.\end{aligned}$$

We conclude that $\frac{1}{2\pi i}\int_{1-iT}^{1+iT} G(s)x^s ds$
equals
\begin{align*} 
 \frac{1}{2\pi i} \left(\int_{1-i T}^{\sigma_n-i T} +
 \int_{\sigma_n+i T}^{1+ i T}\right)G(s) x^s ds &+
\frac{1}{2\pi i} \int_{\sigma_n-i T}^{\sigma_n + i T} 
G^\circ(s) x^s ds + \frac{1}{\pi}
\Im \int_{\mathcal{C}_{n,1}^+} G^\star(s) x^s ds\\
+ \sum_{\substack{\text{$\rho$ a pole of $G$}\\ \rho\in R\setminus R_{\mathcal{C}},\;\Re \rho > \sigma_n}} 
\Res\limits_{s=\rho} G(s) x^s &+
\sum_{\substack{\text{$\rho$ a pole of $G^\circ$}\\ \rho\in R_{\mathcal{C}},\;\Re \rho > \sigma_n}} 
\Res\limits_{s=\rho} G^\circ(s) x^s.
\end{align*}
Finally,  we let $n\to \infty$. Since $G(s)x_0^s$ is bounded on $\partial R$,
$|G(s) x^s|\leq \left|G(s) x_0^s (x/x_0)^s\right| \ll (x/x_0)^{\Re s}$, and so $$\lim_{n\to \infty} \int_{\sigma_n\pm i T} ^{1\pm i T} G(s) x^s ds = \int_{-\infty\pm i T} ^{1\pm i T} G(s) x^s ds.$$
Similarly, since $G^\circ(s) x_0^s$ is bounded on $L$, then, as $n\to \infty$,
$$\left|\int_{\sigma_n-i T}^{\sigma_n + i T} G^\circ(s) x^s ds\right|\ll (x/x_0)^{\sigma_n} T \to 0.$$
\end{proof}

Let us now apply Lemma \ref{lem:kolobrz2} to our specific function $\Phi_\lambda^\pm$.

\begin{proposition}\label{prop:shiftnonneg}
Let $F(s)$ be a meromorphic function on $R = (-\infty,1] + i [-T,T]$ with
$\overline{F(s)} = F(\overline{s})$. Let $\mathcal{C}$ be an admissible contour contained in $(-\infty,1] + i [0,\frac{T}{4}]$, going from $1$ to $\Re s = -\infty$. 
Assume that, for some $x_0\geq 1$, $F(s) x_0^s$ is bounded on $\partial R \cup \mathcal{C}\cup L$, where $L$ is as in \eqref{eq:defL}.

Let $\lambda\ne 0$.
Let $\Phi_{|\lambda|}^{\pm,\circ}$ and $\Phi_{|\lambda|}^{\pm,\star}$ be as in \eqref{eq:arnoros}. Define $$\Phi_\lambda^\pm(z) =  \Phi_{|\lambda|}^{\pm,\circ}(\sgn(\lambda) z) + 
\sgn(\lambda) \sgn(\Re z) \Phi_{|\lambda|}^{\pm,\star}(\sgn(\lambda) z).$$
Then, for any $x>x_0$,
\begin{equation}\label{eq:garmun}\frac{1}{2\pi i} \int_{1-iT}^{1+i T} \Phi_{\lambda}^\pm\left(\frac{s-1}{i T}\right) F(s) x^s ds 
\end{equation} equals
\begin{equation}\label{eq:trothis}\begin{aligned}
&\sum_{\substack{\text{$\rho$ a pole of $F(s)$}\\ \rho\in R\setminus R_{\mathcal{C}}}} 
\Res\limits_{s=\rho} \; \Phi_{\lambda}^\pm\left(\frac{s-1}{i T}\right) F(s) x^s 
+
\sum_{\substack{\text{$\rho\in R_{\mathcal{C}}$ a pole of $F(s)$}\\
\text{or $\rho = 1 +\frac{\lambda T}{2\pi}$ and $\lambda<0$}}}
\Res\limits_{s=\rho} \;\Phi_{|\lambda|}^{\pm,\circ}\left(\sgn(\lambda) \frac{s-1}{i T}\right) F(s) x^s 
\\ &+ \frac{1}{2\pi} O^*\left(\frac{1}{T}
\sum_{\xi=\pm 1} \int_0^\infty t |F(1 - t + i \xi  T)| x^{1-t} dt +
2\left|\int_{\mathcal{C}} \Phi_{|\lambda|}^{\pm,\star}(\sgn(\lambda) z(s)) F(s) x^s ds\right|\right),
\end{aligned}\end{equation}
where  $R_{\mathcal{C}}$ is the closed subregion of $R$ between 
$\mathcal{C}$ and $\overline{\mathcal{C}}$.
\end{proposition}
\begin{proof}
Let $z(s) =\frac{s-1}{iT}$.
By Lemma \ref{lem:benvenuto2},  
$\Phi_{|\lambda|}^{\pm,\circ}(\sgn(\lambda) z(s))$ is bounded
outside arbitrarily small neighborhoods of the points $s = 1 + \frac{\lambda T}{2 \pi} + i T m$ for $m\in \{-1,0,1\}$,
$\Phi_{|\lambda|}^{\pm,\star}(\sgn(\lambda) z(s)) = O(1+|z(s)|) = O(1+|s|)$ outside arbitrarily small neighborhoods of the points $s = 1 + \frac{\lambda T}{2 \pi} + i T m$ for $m\in \{-1,1\}$,
and $\Phi_{\lambda}^{\pm}(z(s)) = O(1+|s|)$ outside an arbitrarily
small neighborhood of $s = 1 + \frac{\lambda T}{2 \pi}$.
Hence,
$\Phi_{|\lambda|}^{\pm,\circ}(\sgn(\lambda) z(s))$ is bounded
on the part $L_0$ of $L$ corresponding to all $\sigma_n<1+\frac{\lambda T}{2\pi}$, and, for any $\epsilon>0$,
 $\Phi_{|\lambda|}^{\pm,\star}(\sgn(\lambda) z(s)) (1+\epsilon)^s$
is bounded
on $L_0\cup \mathcal{C}$ 
and $\Phi_\lambda^\pm(z(s)) (1+\epsilon)^s$ is bounded on
$\partial R$.

Moreover, to ensure that
$\Phi_{|\lambda|}^{\pm,\circ}(\sgn(\lambda) z(s))$ is bounded on $\mathcal{C}$, it suffices
to lift $\mathcal{C}$ slightly if it goes through $1 + \frac{\lambda T}{2\pi}$, without going through
a pole of $F(s)$; we can then take limits at the very end. We let $\epsilon>0$ be such that
$(1+\epsilon) x_0 < x$, and replace $x_0$ by $(1+\epsilon) x_0$.

At this point we may apply Lemma \ref{lem:kolobrz2} with $G(s) = \Phi_\lambda^\pm(z(s)) F(s)$,
$$G^\circ(s) = \Phi_{|\lambda|}^{\pm,\circ}(\sgn(\lambda) z(s)) F(s),\;\;\;\;\;\;\;
G^\star(s) = \sgn(\lambda) \Phi_{|\lambda|}^{\pm,\star}(\sgn(\lambda) z(s)) F(s),$$ and $L_0$ instead of $L$; the assumption $G^\star(\overline{s}) = 
-\overline{G^\star(s)}$ is satisfied because $F(\overline{s}) = \overline{F(s)}$
and (again by Lemma \ref{lem:benvenuto2})
$\Phi_{|\lambda|}^{\pm,\star}(z(\overline{s}))=-\overline{\Phi_{|\lambda|}^{\pm,\star}(z(s))}$,
and the assumption that $G(s) x_0^s$ is bounded on $\partial R$ and 
$G^\circ(s) x_0^s$ and $G^\star(s) x_0^s$ are bounded on $L_0\cup \mathcal{C}$ holds,
since we have replaced $x_0$ by $(1+\epsilon) x_0$.
Note also that $\sgn(\Re z) = 
 \sgn(\Im s)$, and so $G(s) = G^\circ(s) + \sgn(\Im s) G^\star(s)$
 as required by Lemma \ref{lem:kolobrz2}.

We obtain that the expression in \eqref{eq:garmun} equals
\[\begin{aligned}
&\frac{1}{2\pi i} \int_{\mathcal{C}_\infty} \Phi_{\lambda}^\pm\left(z(s)\right) F(s) x^s ds +
 \frac{\sgn(\lambda)}{\pi} \Im \int_{\mathcal{C}}  
 \Phi_{|\lambda|}^{\pm,\star}\left(\sgn(\lambda) z(s) \right) F(s) x^s ds
 \\
+ &\sum_{\substack{\text{$\rho$ a pole of $G$}\\ \rho\in R\setminus R_{\mathcal{C}}}} 
\Res\limits_{s=\rho} G(s) x^s +
\sum_{\substack{\text{$\rho$ a pole of $G^\circ$}\\ \rho\in R_{\mathcal{C}}}} 
\Res\limits_{s=\rho} G^\circ(s) x^s,
\end{aligned}\]
where $\mathcal{C}_\infty$ consists of
a straight path from $1 - i T$ to $-\infty- i T$, and another from $-\infty+i T$ to $1 + i T$.
By Lemma \ref{lem:benvenuto2}, the poles of $G(s)$ in $R\setminus R_{\mathcal{C}}$ are just the poles of $F(s)$ in $R\setminus R_{\mathcal{C}}$, and the poles of
$G^\circ(s)$ in $R_{\mathcal{C}}$
are just the poles of $F(s)$ in $R_{\mathcal{C}}$, plus a pole at
$1+\frac{\lambda T}{2\pi}$ if $\lambda<0$, coming from the pole of $\Phi_{|\lambda|}^{\pm,\circ}$
there.
Yet again by Lemma \ref{lem:benvenuto2},
$|\Phi_{\lambda}^\pm(1 + i r)| \leq |r|$ and $|\Phi_\lambda^\pm(-1 + i r)| \leq |r|$
for $r$ real.
Therefore, $$\frac{1}{2\pi}\left|\int_{\mathcal{C}_\infty} \Phi_{\lambda}^\pm\left(\frac{s-1}{i T}\right) F(s) x^s ds
\right|\leq \frac{1}{2\pi T} \sum_{\xi=\pm 1} \int_0^\infty t |F(1 - t + i \xi  T)| x^{1-t} dt.$$

    \end{proof}

\section{Proof of the main theorem}\label{sec:proofmain}

We will now prove Theorem \ref{thm:mainthmB}, that is, our general result for non-negative $a_n$.

We will start by applying Prop.~\ref{prop:gennonneg}.
We will want to choose $\varphi_+$ and
$\varphi_-$ so as to make the term $\varphi_\pm(0)= \int_{-\infty}^\infty \widehat{\varphi_\pm}(u) du $ in
\eqref{eq:rotterdam2} as close to $\int_{-\infty}^\infty I_\lambda(u) du$
as possible.
Because $\widehat{\varphi_+}$ is a majorant and $\widehat{\varphi_-}$ is a minorant,
that is the same as minimizing $\|\widehat{\varphi_{\pm}} - I_\lambda\|_1$. Thus, we will
choose $\varphi_\pm = \varphi_{\pm,\lambda}$ as in Prop.~\ref{prop:grahamvaaler}.

A benevolent thief then comes in the night and replaces $\varphi_{\pm,\lambda}(z)$
in the integrand by something equal to it on the segment
of integration, namely, $\Phi_\lambda^\pm(z)$. As we know, $\Phi_\lambda^\pm(z)$ consists
of two holomorphic functions, sewn together along the $x$-axis. We can then shift
the contour to the left, as in Prop.~\ref{prop:shiftnonneg}. If $\sigma=1$, we work with $\sigma\to 1^-$.

\begin{proposition}\label{prop:ewsai}
    Let the assumptions on $A(s)$ and $T$ be as in 
    the statement of Theorem \ref{thm:mainthmB}.
    Let $S_\sigma(x)$ be as in \eqref{eq:sotodef} for
    $\sigma\in \mathbb{R}\setminus \{1\}$. Then, for any
    $x>T$,
    \begin{equation}\label{eq:julcot}\begin{aligned}
    x^\sigma &S_\sigma(x) = \frac{2\pi }{T}
\Im 
\sum_{\substack{\text{$\rho$ a pole of $A(s)$}\\ \rho\in R^+}}
\sgn(\sigma-1)
 \left(\theta(\rho)\cot (\pi \theta(\rho)) - c_{\sigma,T}\right) x^\rho \cdot
\Res\limits_{s=\rho} A(s)\\ &- \frac{\pi}{T} 
\sum_{\substack{\text{$\rho\in \mathbb{R}$ a pole of $A(s)$}\\
\text{or $\rho = \sigma$ and $\lambda<0$}} }
 \sgn(\sigma-1)
  \Res_{s=\rho} \coth \left(\frac{\pi (s-\sigma)}{T}\right) A(s) x^s\\
&+ \frac{2\pi}{T} O^*\left(
\Re
\sum_{\substack{\text{$\rho$ a pole of $A(s)$}\\ \rho\in R^+}} 
 \left(1 - \frac{\rho-1}{i T}\right) x^\rho\cdot
\Res\limits_{s=\rho} A(s) + \frac{1}{2}  
\sum_{\text{$\rho\in \mathbb{R}$ a pole of $A(s)$}} \Res_{s=\rho} A(s) x^s
\right)\\
&+ \frac{1}{T^2} O^*\left(
\sum_{\xi=\pm 1} \int_0^\infty t |F(1 - t + i \xi  T)| x^{1-t} dt +
2 T\left|\int_{\mathcal{C}} \Phi_{2\pi|\sigma-1|/T}^{\pm,\star}(\sgn(\lambda) z(s)) F(s) x^s ds\right|
\right),
    \end{aligned}\end{equation}
    where  $R^+ =(-\infty,1] + i (0,T)$,
    $$\theta(s) = 1 - \frac{s-\sigma}{i T},\;\;
    c_{\sigma,T} = \theta(1+i T) \cot (\pi \theta(1+ i T)),\;\;z(s) = \frac{s-1}{i T},\;\;
    F(s) = A(s) - \frac{\Res_{s=1} A(s)}{s-1},$$
    and $\Phi_\nu^{\pm,\star}$ is as in
    \eqref{eq:arnoros}.
\end{proposition}
\begin{proof}
We can assume without loss of generality that $\Res_{s=1} A(s) = 1$.

Apply Proposition \ref{prop:gennonneg} with $\varphi_\pm=\varphi_{\pm,\lambda}$, where $\varphi_{\pm,\lambda}$ are as in 
Proposition \ref{prop:grahamvaaler} and
$\lambda = \frac{2\pi (\sigma-1)}{T}$. The condition $\widehat{\varphi_\pm}(y) = 
O(1/|y|^2)$ holds because the second distributional derivative
$\varphi_\pm''$ lies in $L^1(\mathbb{R})$. 
We are integrating on a straight line from $1-iT$ to $1+i T$, and there $\Im s = \frac{s-1}{i}$.
Thus
\begin{equation}\label{eq:hofman}\begin{aligned}S_\sigma(x) &\leq  
     \frac{2\pi x^{1-\sigma}}{T} \varphi_+(0) + \frac{x^{-\sigma}}{iT} \int_{1-i T}^{1 +i T} \varphi_+\left(\frac{s-1}{i T}\right)
	 \left(A(s) - \frac{1}{s-1}\right) x^s ds +
     \frac{\mathds{1}_{(-\infty,1)}(\sigma)}{\sigma-1},
     \\S_\sigma(x) &\geq
     \frac{2\pi x^{1-\sigma}}{T} \varphi_-(0) + \frac{x^{-\sigma}}{iT} \int_{1-i T}^{1 +i T} \varphi_-\left(\frac{s-1}{i T}\right)
	 \left(A(s) - \frac{1}{s-1}\right) x^s ds +
     \frac{\mathds{1}_{(-\infty,1)}(\sigma)}{\sigma-1}.
\end{aligned}\end{equation}

By \eqref{eq:arnor1}--\eqref{eq:arnoros},
since  $\Phi_{|\lambda|}^{\pm,\circ}(0) = \frac{1}{2} \coth \frac{|\lambda|}{2} \pm \frac{1}{2}$
and $\Phi_{|\lambda|}^{\pm,\star}(0) = 0$,
$$\varphi_{\pm}(0) = \frac{1}{2} \coth \frac{|\lambda|}{2} \pm \frac{1}{2}.$$
Also by \eqref{eq:arnor1},
\begin{equation}\label{eq:amfia}
\int_{1-i T}^{1 +i T} \varphi_\pm\left(\frac{s-1}{i T}\right)
	 \left(A(s) - \frac{1}{s-1}\right) x^s ds =
\int_{1-i T}^{1 +i T} \Phi_\lambda^\pm\left(\frac{s-1}{i T}\right)
	 \left(A(s) - \frac{1}{s-1}\right) x^s ds
\end{equation}
where $\Phi_\lambda^\pm(z) =  \Phi_{|\lambda|}^{\pm,\circ}(\sgn(\lambda) z) + \sgn(\lambda)
\sgn(\Re z) \Phi_{|\lambda|}^{\pm,\star}(\sgn(\lambda) z)$.
We apply Proposition
\ref{prop:shiftnonneg} with $F(s) = A(s)-1/{(s-1)}$ to estimate
the integral on the right in \eqref{eq:amfia}. Let us simplify what 
we obtain.

As the coefficients $a_n$ are real, the poles of $A(s)$ come
in conjugate pairs $\rho$, $\overline{\rho}$. Hence, by \eqref{eq:Phicong},
$$
\sum_{\substack{\text{$\rho$ a pole of $F(s)$}\\ \rho\in R\setminus R_{\mathcal{C}}}} 
\Res\limits_{s=\rho} \; \Phi_\lambda^\pm(z(s)) F(s) x^s = 
2 \Re
\sum_{\substack{\text{$\rho$ a pole of $F(s)$}\\ \rho\in R^+}} 
\Res\limits_{s=\rho} \; \Phi_\lambda^\pm(z(s)) F(s) x^s.
$$
Here we are using the simplifying assumption that all poles of $A(s)$ in $R^+$ lie above $\mathcal{C}$. 
 Since, by Lemma \ref{lem:benvenuto2},
$\Phi_\lambda^{\pm}(z(s))$ has no poles in $R^+$,
we can replace $F(s) = A(s)-1/(s-1)$ here 
by $A(s)$, and write
$\Res_{s=\rho} \; \Phi_\lambda^\pm(z(s)) F(s) x^s$ as
$\Phi_\lambda^\pm(z(\rho)) x^\rho \Res_{s=\rho} A(s)$.
Thus, for $\rho$ with $\Im \rho > 0$, Lemma \ref{lem:omlet} gives us that
$\Re \Res\limits_{s=\rho} \Phi_\lambda^\pm(z(s)) A(s) x^s$ equals 
$$\begin{aligned}
   &\sgn(\lambda) \Im 
\left(\left( \frac{\theta(\rho)}{2} \cot (\pi \theta(\rho)) -
\frac{\theta(1+i T)}{2} \cot (\pi \theta(1+ i T))
 \right)x^\rho\cdot
 \Res\limits_{s=\rho} A(s) \right)
\\
&\pm\frac{1}{2} \Re \left( \left(1 - z(\rho)\right) x^\rho\cdot
\Res\limits_{s=\rho} A(s) \right).
\end{aligned}$$

Let us now examine the second sum in \eqref{eq:trothis}, and more particularly, the contribution
of $\rho = 1 + \frac{\lambda T}{2\pi} = \sigma$ for $\lambda<0$. By Lemma \ref{lem:benvenuto2},
$\Phi_{|\lambda|}^{\pm,\circ}$ has a simple pole at $\sgn(\lambda) z(1 + \frac{\lambda T}{2\pi}) = 
- \frac{i|\lambda|}{2\pi}$ with residue $\frac{i}{2\pi}$. Hence, by the chain rule for residues, 
\begin{equation}\label{eq:resbal}\Res\limits_{s=\sigma} \;\Phi_{|\lambda|}^{\pm,\circ}\left(\sgn(\lambda) z(s)\right) =
\frac{i}{2\pi} \cdot \frac{1}{\sgn(\lambda) z'(\sigma)} = 
\frac{i}{2\pi} \cdot \frac{1}{\sgn(\lambda)/(i T)} = \frac{T}{2\pi}.
\end{equation}
Therefore, if $\lambda<0$, 
the contribution of the term $-\frac{1}{s-1}$ in $A(s)$ to $S_\sigma(x)$ at $\rho=\sigma$  is
$$\frac{2\pi x^{-\sigma}}{T} \Res\limits_{s=\sigma} \;\Phi_{|\lambda|}^{\pm,\circ}\left(\sgn(\lambda) z(s)\right) \left(-\frac{1}{s-1}\right) x^s 
= \frac{2 \pi x^{-\sigma}}{T} \cdot \frac{T}{2\pi} \left(-\frac{1}{\sigma-1}\right) x^\sigma =  - \frac{1}{\sigma-1}.
$$
This contribution thus cancels out the term $\frac{1}{\sigma-1}$ in \eqref{eq:hofman} for $\lambda<0$, which is
precisely when that term appears in the first place. 

At poles $\rho\in R_{\mathcal{C}}$ other than
$\sigma$, we know $\Phi_{|\lambda|}^{\pm,\circ}\left(\sgn(\lambda) z(s)\right)$ does not have
a pole,  and so the term $\frac{1}{s-1}$ can be removed as before, as it does not affect the residue
 $\Res\limits_{s=\rho} \;\Phi_{|\lambda|}^{\pm,\circ}\left(\sgn(\lambda) z(s)\right) F(s) x^s$.
We recall that $\Phi_{|\lambda|}^{\pm,\circ}\left(\sgn(\lambda) z(s)\right) = 
\frac{1}{2} \left(\coth \frac{w}{2} \pm 1\right)$, where $w =
-\sgn(\lambda) \frac{2\pi (s-\sigma)}{T}$.

We conclude that
\begin{equation}\label{eq:padrone}
S_\sigma(x) = \frac{x^{1-\sigma} }{\sigma-1} \cdot \frac{\lambda}{2} \coth \frac{|\lambda|}{2} +
\frac{2 \pi x^{-\sigma}}{T} \left(\mathcal{T}_{1,\sigma,x} +  O^*(\mathcal{T}_{2,\sigma,x})\right)+ 
\frac{x^{-\sigma}}{T^2} O^*(\mathcal{T}_{3,\sigma,x}),\end{equation}
where
\[\begin{aligned}
\mathcal{T}_{1,\sigma,x} &= 2 \sgn(\lambda) 
\Im 
\sum_{\substack{\text{$\rho$ a pole of $F(s)$}\\ \rho\in R^+}}
 \left(\frac{\theta(\rho)}{2} \cot (\pi \theta(\rho)) - \frac{\theta(1+ i T)}{2}
 \cot (\pi \theta(1+ i T))\right) x^\rho \cdot
\Res\limits_{s=\rho} A(s) \\
&- \frac{\sgn(\lambda)}{2} 
\sum_{\substack{\text{$\rho\in (-\infty,1)$ a pole of $F(s)$}\\
\text{or $\rho = \sigma$ and $\lambda<0$}} }
  \Res_{s=\rho} \coth \frac{\pi (s-\sigma)}{T} A(s) x^s
,\end{aligned}\]
\[\begin{aligned}
\mathcal{T}_{2,\sigma,x}&= \frac{x}{2} + \Re
\sum_{\substack{\text{$\rho$ a pole of $F(s)$}\\ \rho\in R^+}} 
 \left(1 - \frac{\rho-1}{i T}\right) x^\rho\cdot
\Res\limits_{s=\rho} A(s) + \frac{1}{2}  
\sum_{\text{$\rho\in (-\infty,1)$ a pole of $F(s)$}} \Res_{s=\rho} A(s) x^s,
\end{aligned}\]
\[\begin{aligned}
\mathcal{T}_{3,\sigma,x}&=
\sum_{\xi=\pm 1} \int_0^\infty t |F(1 - t + i \xi  T)| x^{1-t} dt +
2 T\left|\int_{\mathcal{C}} \Phi_{|\lambda|}^{\pm,\star}(\sgn(\lambda) z(s)) F(s) x^s ds\right|.\;\;\;\;\;\;\;\;\;\;\;\;\;\;\;\;\;
\end{aligned}\]
We realize that
$$\begin{aligned}
    \Res_{s=1} \coth \frac{\pi (s-\sigma)}{T} A(s) x^s = \coth 
    \frac{\pi (1-\sigma)}{T} \cdot \Res_{s=1} A(s) \cdot  x = 
x \coth \frac{-\lambda}{2} = - x \coth \frac{\lambda}{2},\end{aligned}
$$
and so the contribution of that residue to $S_{\sigma}(x)$, if included in the second sum in
$\mathcal{T}_{1,\sigma,x}$, would be $\frac{2 \pi}{T} x^{1-\sigma} \frac{\sgn(\lambda)}{2} \coth \frac{\lambda}{2}
= \frac{x^{1-\sigma}}{\sigma-1} \frac{\lambda}{2} \coth \frac{|\lambda|}{2}$, matching the
first term in \eqref{eq:padrone}. So, we include $\rho=1$ in that sum.
Similarly, we subsume the term $\frac{x}{2}$ in $\mathcal{T}_{2,\sigma,x}$ by including
$\rho=1$ in the second sum in $\mathcal{T}_{2,\sigma,x}$.
\end{proof}
\begin{proof}[Proof of Theorem \ref{thm:mainthmB}]
Let $\Phi(s)$ be whichever of
$T \Phi_{2\pi|\sigma-1|/T}^{\xi,\star}(\sgn(\lambda) (s-1)/i T)$, $\xi=\pm 1$, would
make the last integral in \eqref{eq:julcot} larger. Then $\Phi(1)=0$ and
$|\Phi'(s)|\leq 1$ for $s\in R_{1/4}$ by Lemma \ref{lem:benvenuto2}.

{\bf Case $\sigma <1$.} We apply Proposition \ref{prop:ewsai} and are done.

{\bf Case $\sigma >1$.} We 
want
to estimate $\sum_{n\leq x} a_n n^{-\sigma} = \sum_n a_n n^{-\sigma} - \sum_{n>x} a_n n^{-\sigma} = A(\sigma) - S_\sigma(x^+)$, so
we apply Prop.~\ref{prop:ewsai} with $x^+$ (that is, a sequence of reals tending to $x$ from above) instead of $x$. Since $\sigma>1$, $A(s)$ does not have a pole at $s=\sigma$,
whereas $\coth\big(\frac{\pi(s-\sigma)}{T}\big)$ has a simple pole at $s=\sigma$ with residue
$\frac{T}{\pi}$, and so $\coth\big(\frac{\pi(s-\sigma)}{T}\big) A(s) x^s$ has a simple pole
with residue $\frac{T}{\pi} A(\sigma) x^\sigma$. We see, then, that that pole would contribute
exactly $-A(\sigma)$ to \eqref{eq:padrone} if included in the sum $\mathcal{T}_{1,\sigma,x}$, and so we include it. The sign in $-S_\sigma(x^+)$ cancels out the minus sign here.

There is a subtlety regarding convergence here: the sums $\sum_\rho$ are, in general,
infinite sums, and we do not know a priori that the limit as $x_n\to x^+$ of such a sum of residues equals the sum of the limits of the residues. 
The solution is the same as in \cite{Moebart}. Recall that $\sum_{\rho \in \mathcal{Z}_A^+(T) }$ here means
$\lim_{m\to \infty} \sum_{\rho \in \mathcal{Z}_A^+(T) : \Re \rho > \sigma_m}$, where $\sigma_m\to -\infty$ (monotonically, it may be assumed). The difference between the sums of this form for two consecutive
values is then
\begin{equation}\label{eq:adorno} \Delta_m = \sum_{\rho \in \mathcal{Z}_A^+(T): \sigma_m < \Re \rho\leq \sigma_{m+1}}
\omega^+_{T,\sigma}(\rho) x^{\rho-1}
\Res_{s=\rho}  A(s).\end{equation}
and the same with $\theta_{T,1}(\rho)$ instead of $\omega^+_{T,\sigma}(\rho)$.
These sums equal $\frac{1}{2\pi}$ times the integral on the closed contours that goes from
$\sigma_m + i T$ to $\sigma_{m+1} + i T$, then goes down vertically until meeting
the contour $\mathcal{C}$, follows $\mathcal{C}$ rightwards until meeting
the line $\Re s = \sigma_m$, and then goes vertically up to $\sigma_m+ i T$.

On that closed contour, $A(s) T^{s-1}$ is uniformly bounded, and our weights 
$\omega^+_{T,\sigma}(s)$,  $\theta_{T,1}(s)$ are at most linear on $s$.
So, by $x_n>x>T$, we see that $\Delta_m$
decays exponentially on $\sigma_m$, uniformly on $n$.
Since 
$\sum_{\rho \in \mathcal{Z}_A^+(T): \Re \rho > \sigma_M}$ is the sum of
the terms \eqref{eq:adorno} for $m<M$, then, by dominated convergence, the limit as
$x_n\to x^+$ of the limit as $M\to \infty$ equals the limit as $M\to \infty$ of the limit as $x_n\to x^+$.
In other words, the limit as $x_n\to x^+$ of each of the sums  $\sum_{\rho \in \mathcal{Z}_A^+(T)}$ in Prop.~\ref{prop:ewsai} is just the
sum  $\sum_{\rho \in \mathcal{Z}_A^+(T)}$ for $x$. The same argument works
for the sums $\sum_{\rho\in \mathcal{Z}_{A,\mathbb{R}}\cup \{\sigma\}}$.

{\bf Case $\sigma = 1$.} Apply our final statement \eqref{eq:quentino} with $\sigma\to 1^-$. The
same issue regarding convergence arises as for $\sigma>1$, and it is dealt with in the same way.
All that is left to show is that one quantity, namely, the limit as $\sigma\to 1^-$ of the sum of the contributions of 
$\rho=\sigma$ and
$\rho=1$ to the first sum in \eqref{eq:quentino}, equals another, namely, the contribution of $\rho=1$ to the first sum in \eqref{eq:quentino} when $\sigma=1$. (We get the former through our limit process, and wish to show that it matches the latter.)

For $\sigma\ne 1$, the residues of $\frac{\pi}{T} \coth \frac{\pi (s-\sigma)}{T} A(s) x^{s-1}$ at 
$\rho=\sigma$ and $\rho=1$ are $A(\sigma) x^{\sigma-1}$ and $\frac{\pi}{T} \coth \frac{\pi (1-\sigma)}{T}$, respectively.
For $\sigma$ near $1$, the sum of these two residues is of the form
\[ 
\left(\frac{1}{\sigma-1} + c_0 + o(1)\right) \left(1 + (\sigma-1) \log x + o(\sigma-1)\right) + \frac{1}{1-\sigma}
+o(1),\]
with $c_0\in\R$, and so it tends to $\log x + c_0$ as $\sigma\to 1^-$. For $\sigma=1$, 
$\frac{\pi}{T} \coth \frac{\pi (s-\sigma)}{T} A(s) x^{s-1}$ is
\begin{equation}\label{eq:kelkap}\left(\frac{1}{s-1} + O(s-1)\right)
\left(\frac{1}{s-1} + c_0 + O(s-1)\right) \left(1 + (s-1) \log x +o(s-1)\right),
\end{equation}
whose residue at $s=1$ is $\log x + c_0$, so all is well.
\end{proof}


\section{The case of $\Lambda(n)$: sums over zeros of $\zeta(s)$}\label{sec:zetres}

We now wish to apply Theorem~\ref{thm:mainthmB} to estimate sums $\sum_n \Lambda(n) n^{-\sigma}$.
That means we will want to estimate the sums in the statement of Thm.~\ref{thm:mainthmB}
in the special case $A(s) = -\zeta'(s)/\zeta(s)$. 


We will be working with $\cot$ and $\coth$, using the Laurent series 
\begin{equation}\label{eq:mireio}
\pi \cot \pi z = \frac{1}{z} - 2 \sum_n
\zeta(2 n) z^{2 n - 1},
\end{equation}
immediate from \cite[(4.19.6)]{zbMATH05765058} and \cite[\S 1.5(2)]{zbMATH03492781}.
We will also use Euler's expansion 
\begin{equation}\label{eq:eulercot}
\cot z = \frac{1}{z} + 2 z \sum_n \frac{1}{z^2 - n^2\pi^2}
\quad\text{and so}\quad \coth z = \frac{1}{z} + 2 z \sum_n \frac{1}{z^2 + n^2\pi^2}.
\end{equation}
In particular, for $y\geq 0$, $\coth y \leq \frac{1}{y} + 2 y \sum_n \frac{1}{n^2\pi^2} 
\leq \frac{1}{y} + \frac{y}{3}$, while also $\coth y = 1 + \frac{2}{e^{2 y} - 1} \leq 1 + \frac{1}{y}$.

\subsection{Trivial zeros}

\begin{lemma}\label{lem:trivialzer}
Let $T>0$,  $x>1$, $\sigma>-2$. Let $A(s) = -\zeta'(s)/\zeta(s)$. Then
\[\left|\frac{\pi}{T} \sum_n \Res_{s=-2 n} \coth \frac{\pi (s-\sigma)}{T} A(s) x^{s-1}\right| +
\left|\frac{\pi}{T} \sum_n \Res_{s=-2 n}A(s) x^{s-1}\right|\leq 
\frac{\frac{1}{2+\sigma} + \frac{2\pi}{T}}{x^3 (1 - x^{-2})}.
\]
\end{lemma}
\begin{proof}
The residue of $A(s)$ at $s=-2, -4,\dotsc$ is $-1$. Thus
\[\left|\sum_n \Res_{s=-2 n}A(s) x^{s-1}\right|\leq  \sum_n x^{-2 n - 1} = 
\frac{1}{x^3 (1 - x^{-2})},\]
\[\left|\sum_n \Res_{s=-2 n} \coth \frac{\pi (s-\sigma)}{T} A(s) x^{s-1}\right| =
\sum_n \coth\left(\pi \frac{2 n + \sigma}{T}\right) x^{- 2 n - 1}
\leq \coth \pi \frac{2+ \sigma}{T} \cdot \sum_n x^{-2 n - 1},
\]
since $\coth y$ is decreasing for $y>0$. By $\coth y \leq 1/y + 1$, we are done.
\end{proof}

\subsection{Non trivial zeros}\label{subs:nontrivz}
\subsubsection{Estimates on weights}

We have two sums over non-trivial poles of $A(s)$ in Theorem \ref{thm:mainthmB}: the sum with
weight $\omega_{T,\sigma}^+(s)$ and the sum with weight $\theta_{T,1}(s)$. 
The fact that we are taking real $\Re$ and imaginary parts $\Im$ in Theorem \ref{thm:mainthmB} means
we need not work with $|\omega_{T,\sigma}^+(\rho)| + |\theta_{T,1}(\rho)|$: we may rather work with
$|\omega_{T,\sigma}^+(\rho) + \xi\theta_{T,1}(\rho)i|$, with $\xi\in [-1,1]$, which is smaller. Of course this is just the same as working with the norm $|\Phi_\lambda^\pm(z)|$ of our original weight function.
Our aim will be to show that, on average, this weight will contribute less than the classical weight $\frac{1}{\pi |\gamma|}$ would.

We will approximate our
weight $\omega_{T,\sigma}^+(\rho) +  \xi\theta_{T,1}(\rho)i$ in two ways:
(a) by a simplified weight depending only on $\frac{\gamma}{T} = \frac{\Im \rho}{T}$, and (b) by
the classical weight $\frac{1}{\rho-\sigma}$. The former
approximation is closer for $\gamma$ large, the latter for $\gamma$ small.

We can now approximate our weight by a relatively simple function on the reals. Define
\begin{align} \label{eq:kullervo}
F(z) = \frac{1}{\pi}- (1-z) \cot \pi (1-z).
\end{align}

\begin{lemma}\label{cor:thonny}
Let $T>0$, $\sigma\in\mathbb{R}$ with $|\sigma-\frac{1}{2}|\leq \frac{T}{2}$.
 Let $\omega_{T,\sigma}^+(s)$ and $\theta_{T,1}(s)$ be as in Theorem \ref{thm:mainthmB}. Let $F$
 be as in \eqref{eq:kullervo}. 
 Then, for $s = \frac{1}{2} + i t$ with $0< t\leq T$ and any $\xi\in [-1,1]$,
\begin{equation}\label{eq:witdim}\omega_{T,\sigma}^+(s) + \xi \theta_{T,1}(s) i = 
 F\left(\frac{t}{T}\right)  +\xi \cdot \left(1 - \frac{t}{T}\right) i + O^*\left(
 \frac{\left|\sigma-\frac{1}{2}\right|}{\pi}  \frac{T}{t^2} + 
\frac{2.78 \left|\sigma-\frac{1}{2}\right| + 1}{T}\right).\end{equation}
Moreover, for $s=\frac{1}{2}+it$ with $0<t\leq \frac{T}{2}$
 and any $\xi\in [-1,1]$,
 \begin{equation}\label{eq:demoscen}
 \omega_{T,\sigma}^+(s) + \xi \theta_{T,1}(s) i = 
 \frac{i T}{(s-\sigma) \pi}+ O^*\left( 1 + 
\frac{
2.78 \left|\sigma-\frac{1}{2}\right| + 1}{T}\right).
 \end{equation}
 \end{lemma}
 The condition $\left|\sigma-\frac{1}{2}\right|\leq \frac{T}{2}$ is of course very loose; it is all that is needed to apply Lemma 
 \ref{lem:sibelius}.
\begin{proof}
Let $c_{T,\sigma}$ be as in Theorem
\ref{thm:mainthmB}.
Then, by the definition of $\omega_{T,\sigma}^+$(s),
\[\omega_{T,\sigma}^+(s)  = 
\frac{1}{\pi} - \theta_{T,\sigma}(s) \cot \big(\pi \theta_{T,\sigma}(s)\big) 
+ O^*\left(\left|c_{T,\sigma} - \frac{1}{\pi}\right|\right).\]
Since
$\theta_{T,\sigma}(1+i T) = \frac{1}{i T} (\sigma-1)$,
\[c_{T,\sigma}-\frac{1}{\pi} = \frac{\sigma-1}{i T} \cot \frac{\pi (\sigma-1)}{i T} 
- \frac{1}{\pi} = 
\frac{1}{\pi}\left(\frac{\pi (\sigma-1)}{T} \coth \frac{\pi (\sigma-1)}{T} - 1\right),
\]
which has absolute value $\leq \frac{|\sigma-1|}{T}$ by Lemma \ref{lem:cothder}. Thus, since
$F\left(\frac{s-\sigma}{i T}\right) = \frac{1}{\pi} - \theta_{T,\sigma}(s) \cot \big(\pi \theta_{T,\sigma}(s)\big)$, 
\[\omega_{T,\sigma}^+(s) = F\left(\frac{s-\sigma}{i T}\right) + O^*\left(\frac{|\sigma-1|}{T}\right).
\]

Assume $0<t\leq T$. Then, by \eqref{eq:solucky},
\begin{align*}
\omega_{T,\sigma}^+(s) + \xi \theta_{T,1}(s) i & = F\left(\frac{t}{T} + \frac{\sigma-\frac{1}{2}}{T} i\right) +\xi \theta_{T,1}(s) i + O^*\left(\frac{|\sigma-1|}{T}\right) \\
& 
= F\left(\frac{t}{T}\right) + \xi \left(1-\frac{t}{T}\right) i + O^*\left(
\frac{\left|\sigma-\frac{1}{2}\right|}{\pi t^2/T} + 
\frac{2.78 \left|\sigma-\frac{1}{2}\right| + 1}{T}\right)
\end{align*}
since $1.78 \left|\sigma-\frac{1}{2}\right| + |\sigma-1| + \frac{1}{2}\leq
2.78 \left|\sigma-\frac{1}{2}\right| + 1$.

Let us now consider $0< t\leq \frac{T}{2}$. For $A(z)$ as in Lemma \ref{lem:sibelius}, since $s= \frac{1}{2}+ i t$, \eqref{eq:garmenio} gives us that
\[F\left(\frac{s-\sigma}{i T}\right) = \frac{i T}{(s-\sigma) \pi} + A\left(\frac{s-\sigma}{i T}\right) = \frac{i T}{(s-\sigma) \pi} + \frac{\pi}{3} O^*\left(\frac{t}{T}\right) +
1.78\cdot O^*\left(\frac{\left|\sigma-\frac{1}{2}\right|}{T}\right)
,\]
where the term $O^*(\frac{t}{T})$ is real-valued.
Since $\theta_{T,1}(s) = 
\left(1-\frac{t}{T}\right) - \frac{i}{2 T}$, it follows that
\[\omega_{T,\sigma}^+(s) + \xi \theta_{T,1}(s) i 
= \frac{i T}{(s-\sigma) \pi} + O^*\left(\alpha\left(\frac{t}{T}\right) +
\frac{1.78 \left|\sigma-\frac{1}{2}\right| + |\sigma-1|+\frac{1}{2}}{T}\right),\]
where $\alpha(r) = \sqrt{(1-r)^2 + (\frac{\pi}{3})^2 r^2} = 
\sqrt{1- (2 - (1+(\frac{\pi}{3})^2) r) r} \leq 1$ for
$0\leq r\leq \frac{1}{2}$.
\end{proof}


Now we can compare the integral of the norm of a simplified version of our weight with the integral of the norm of the classical weight.
\begin{proposition}\label{prop:tritura}
Let $F$ be as in \eqref{eq:kullervo}. 
Let $T\geq t_0\geq 2\pi$. Then
\begin{equation}\label{eq:simplon}\frac{2 \pi}{T} \int_{t_0}^T \sqrt{F\left(\frac{t}{T}\right)^2 + \left(1- \frac{t}{T}\right)^2}
\log \frac{t}{2\pi} dt \leq  \log^2 \frac{T}{2\pi} - \log^2 \frac{t_0}{2\pi}.\end{equation}
Furthermore, if $t_0\geq 2\pi e$ and $T\geq 3 t_0$,
\begin{equation}\label{eq:adago}\frac{2 \pi}{T} \int_{t_0}^T \sqrt{F\left(\frac{t}{T}\right)^2 + \left(1- \frac{t}{T}\right)^2}
\log \frac{t}{2\pi} dt \leq  \log^2 \frac{T}{2\pi} - \log^2 \frac{t_0}{2\pi} - 
2 C_1 \log \frac{e T}{2\pi} + 2 C_2,
\end{equation}
where $C_k = \sum_n \zeta(2 n) \left(\frac{1}{(2 n)^k} - \frac{2}{(2 n+1)^k} + \frac{1}{(2 n + 2)^k}\right)$.
\end{proposition}
The classical weight $\frac{1}{t}$ would contribute precisely $2\int_{t_0}^T \frac{1}{t} \log \frac{t}{2\pi} dt = \log^2 \frac{T}{2\pi} - \log^2 \frac{t_0}{2\pi}$.
\begin{proof}
Since $x\to \sqrt{x}$ is concave, $\sqrt{a+b}\leq \sqrt{a} + \frac{b}{2\sqrt{a}}$ for any $a>0$, $b>-a$, and so
\begin{equation} \label{1_55pm}
\sqrt{F(u)^2 + (1-u)^2} \leq \frac{1}{\pi u} + \frac{F(u)^2 + (1-u)^2 - \frac{1}{(\pi u)^2}}{2/\pi u}
\end{equation}
for $u>0$ arbitrary. Write $w(u) = \log \frac{u T}{2\pi}$. Then, for $\varepsilon = \frac{t_0}{T}$,
\begin{equation}\label{eq:synthetica}\begin{aligned}\int_{t_0}^T &\sqrt{F\left(\frac{t}{T}\right)^2 + \left(1- \frac{t}{T}\right)^2}
\log \frac{t}{2\pi} dt = T \int_{\varepsilon}^1 \sqrt{F(u)^2 + (1-u)^2 } w(u) du\\
&\leq
T \left(\int_{\varepsilon}^1 \frac{1}{\pi u} \log \frac{u T}{2\pi} du + 
\frac{\pi}{2} \int_{\varepsilon}^1 \left(F(u)^2 + (1-u)^2 - \frac{1}{(\pi u)^2}\right) u w(u) du\right).\end{aligned}\end{equation}
Here $\int_{\varepsilon}^1 \frac{1}{\pi u} \log \frac{u T}{2\pi} du = 
\left. \frac{1}{2\pi} \log^2 \frac{u T}{2\pi}\right|_{u=\varepsilon}^1$.
It is easy to verify that $G(u) = \frac{F(u) (u-1)}{\pi}$
is an antiderivative of $F(u)^2+(1-u)^2$. Since 
\begin{equation}\label{eq:bossa}\frac{1}{\pi^2 u} + G(u) = \frac{1}{\pi^2 u} + \frac{u-1}{\pi} \left( 
 \frac{1}{\pi} + (1-u) \cot \pi u\right) = \frac{1}{\pi^2} + 
 \frac{(1-u)^2}{\pi}\left(\frac{1}{\pi u} - \cot \pi u\right),\end{equation}
 we see that $\frac{(1-u)^2}{\pi}\left(\frac{1}{\pi u} - \cot \pi u\right)$ is
 an antiderivative of $F(u)^2+(1-u)^2 - \frac{1}{(\pi u)^2}$. 
Hence, by integration by parts and $(u w(u))' = w(u) + u w'(u) = w(u)+1$, the last integral
in \eqref{eq:synthetica} equals
\begin{equation}\label{eq:nolights}\begin{aligned}
\left. \frac{(1-u)^2}{\pi}\left(\frac{1}{\pi u} - \cot \pi u\right) u w(u) \right|_{\varepsilon}^1
- \int_{\varepsilon}^1 \frac{(1-u)^2}{\pi}\left(\frac{1}{\pi u} - \cot \pi u\right) (w(u)+1)\, du.
\end{aligned}\end{equation}
Recall that $\frac{1}{\pi u} - \cot \pi u>0$
for all $0<u<1$. We will estimate the quantity in \eqref{eq:nolights}. We already see that it
is negative: because $1-u =0$ for $u=1$ and
$w(u)\geq 0$ for $u\geq \varepsilon$ by $t_0\geq 2\pi$, the term on the left is $\leq 0$,
and the integral is $\geq 0$.  Thus, \eqref{eq:simplon} holds. Let us now
aim for \eqref{eq:adago}.

We can write the integral in \eqref{eq:nolights} as $(I_1(1)-I_1(\varepsilon)) \log \frac{e T}{2\pi} - (I_0(1)-I_0(\varepsilon))$, where
\[I_1(t) = \int_0^t  \frac{(1-u)^2}{\pi}\left(\frac{1}{\pi u} - \cot \pi u\right) du,\quad
I_0(t) = \int_0^t  \frac{(1-u)^2}{\pi}\left(\frac{1}{\pi u} - \cot \pi u\right) (-\log u) du.\]
By \eqref{eq:mireio},
  $ \frac{1}{\pi u} - \cot \pi u  =  \frac{2}{\pi} \sum_n \zeta(2 n)\,u^{2 n-1}$, and so
$$I_{1}(t) = \frac{2}{\pi^2} \sum_n \zeta(2 n) \int_0^1 (1-u)^2 u^{2 n - 1}\, du = \frac{2}{\pi^2} \sum_n \zeta(2 n) \left(\frac{t^{2 n}}{2 n} - \frac{2 t^{2 n+1}}{2 n+1} + \frac{t^{2 n +2}}{2 n + 2}\right),$$
whereas
$I_{0}(t) = \frac{2}{\pi^2} \sum_n \zeta(2 n) \int_0^t (1-u)^2 u^{2 n - 1} (-\log u)\, du$ equals
\[\frac{2}{\pi^2} \sum_n \zeta(2 n) \left(\left(\frac{t^{2 n}}{(2 n)^2} - \frac{2 t^{2 n+1}}{(2 n+1)^2}
+ \frac{t^{2 n +2}}{(2 n + 2)^2}\right) - \left(
\frac{t^{2 n}}{2 n} - \frac{2 t^{2 n+1}}{2 n+1} + \frac{t^{2 n +2}}{2 n + 2}
\right)\log t\right).\]
Thus, $I_1(\varepsilon) \log \frac{e T}{2\pi} - I_0(\varepsilon)$, i.e., $
I_1(\varepsilon) \log \frac{\varepsilon T}{2\pi e} + I_1(\varepsilon) (2-\log \varepsilon) - I_0(\varepsilon)$, equals
\[\begin{aligned}
&\frac{2}{\pi^2} \sum_n \zeta(2 n) 
\left(\frac{\varepsilon^{2 n}}{2 n} - \frac{2 \varepsilon^{2 n+1}}{2 n+1} + \frac{\varepsilon^{2 n +2}}{2 n + 2}\right)
\log \frac{\varepsilon T}{2\pi e} \\ + &\frac{2}{\pi^2} \sum_n \zeta(2 n) 
\left(2 \left(\frac{\varepsilon^{2 n}}{2 n} - \frac{2 \varepsilon^{2 n+1}}{2 n+1} + 
\frac{\varepsilon^{2 n +2}}{2 n + 2}\right)
- \left(\frac{\varepsilon^{2 n}}{(2 n)^2} - \frac{2 \varepsilon^{2 n+1}}{(2 n+1)^2}
+ \frac{\varepsilon^{2 n +2}}{(2 n + 2)^2}\right)\right)
\end{aligned}\]
In comparison, the first term in \eqref{eq:nolights} is
 $-\frac{(1-\varepsilon)^2}{\pi} \left(\frac{1}{\pi \varepsilon} - \cot \pi \varepsilon\right) \varepsilon w(\varepsilon)$, which can be written as
\[-\frac{2 (1-\varepsilon)^2}{\pi^2} \sum_n \zeta(2 n) \varepsilon^{2 n} \left(\log \frac{\varepsilon T}{2\pi e} + 1\right).\]
The sum $S$ of that and $I_1(\varepsilon) \log \frac{e T}{2\pi} - I_0(\varepsilon)$ thus equals
$-\frac{2}{\pi^2}\sum_n \zeta(2 n) (a_n(\varepsilon) \log \frac{\varepsilon T}{2\pi e} + b_n(\varepsilon))  \varepsilon^{2 n}$, where $a_n$ and $b_n$ are as in Lemma \ref{lem:tremic},
which assures us that $a_n(\varepsilon), b_n(\varepsilon)\geq 0$, since $\varepsilon\leq \frac{1}{3}$.
By $t_0\geq 2\pi e$, $\log \frac{\varepsilon T}{2\pi e}\geq 0$, and so $S\leq 0$. We
conclude that the last integral in
 \eqref{eq:synthetica} is $\leq I_1(1) \log \frac{e T}{2\pi} - I_0(1)$, and so
 \eqref{eq:adago} follows.
\end{proof}

Here is an easy variant of Proposition \ref{prop:tritura}, needed for an error term.
\begin{lemma}\label{lem:nimes}
Let $F$ be as in \eqref{eq:kullervo}.
Let $T\geq t_0\geq 2\pi$. Then
\[\frac{2 \pi}{T} \int_{t_0}^T \sqrt{F\left(\frac{t}{T}\right)^2 + \left(1- \frac{t}{T}\right)^2}
\frac{dt}{t} \leq \frac{2}{t_0} - \frac{2}{T}.
\]
\end{lemma}
\begin{proof} 
Just as in the proof of Prop.~\ref{prop:tritura}, for $\varepsilon = \frac{t_0}{T}$,
\[\int_{t_0}^T \sqrt{F\left(\frac{t}{T}\right)^2 + \left(1- \frac{t}{T}\right)^2}
\frac{dt}{t}\leq \int_{\varepsilon}^1 \frac{1}{\pi u} \frac{du}{u} +
\frac{\pi}{2} \int_{\varepsilon}^1 
\left(F(u)^2 + (1-u)^2 - \frac{1}{(\pi u)^2}\right)u \cdot \frac{du}{u}.\]
Since, as we saw in that proof,  $\frac{(1-u)^2}{\pi}\left(\frac{1}{\pi u} - \cot \pi u\right)$ is
 an antiderivative of $F(u)^2+(1-u)^2 - \frac{1}{(\pi u)^2}$,  and
 $\frac{1}{\pi \varepsilon} > \cot \pi \varepsilon$, we get
 $\int_{\varepsilon}^1 
\left(F(u)^2 + (1-u)^2 - \frac{1}{(\pi u)^2}\right) du \leq 0$. Finally, $\int_\varepsilon^1 \frac{du}{\pi u^2} = \frac{1}{\pi \varepsilon} - \frac{1}{\pi} $.
\end{proof}

\subsubsection{Sums over non-trivial zeros}

We finally come to our estimates on the contribution of non-trivial zeros $\rho$ of $\zeta(s)$.
We will use the following for $\Im \rho$ not too small. In all of the following sums over $\rho$, we consider 
$\rho$ with multiplicity. (Of course all zeros of $\zeta(s)$ are believed to be simple, but we do not know that, and neither do we need to assume it.)

\begin{lemma}\label{lem:adar} Let $t_0\geq 2\pi e$, $T\geq 3 t_0$, $\sigma\in \mathbb{R}$ with
$\left|\sigma-\frac{1}{2}\right|\leq \frac{T}{2}$.
 Let $\omega_{T,\sigma}^+(s)$ and $\theta_{T,1}(s)$ be as in Thm.~\ref{thm:mainthmB}.
Assume RH holds up to height $T$. 
Then, for any $\xi\in [-1,1]$,
\begin{equation}\label{eq:oumya}\frac{2\pi}{T} \sum_{\substack{\rho\\ t_0<\Im \rho \leq T}} \left|\omega_{T,\sigma}^+(\rho) + \xi 
\theta_{T,1}(\rho) i \right|\end{equation}
is at most
\begin{equation}\label{eq:amyar}
\frac{1}{2\pi }
 \left(\log^2 \frac{T}{2\pi} -\log^2 \frac{t_0}{2\pi} - P_1\left(
\log \frac{T}{2\pi}\right)\right) + \frac{\err_{1,\sigma}(t_0)}{t_0} + 
\frac{\err_{2,\sigma}(t_0,T)}{T},
\end{equation}
where $P_1(y) = 2 C_1 y + 2 (C_1-C_2)$ with
 $C_1$, $C_2$ as in Lemma \ref{lem:konstanz}, and \[\begin{aligned}
    \err_{1,\sigma}(t_0) &= 2 \left(\frac{2}{5}\log t_0 + \frac{21}{5}\right)
+ \left(\frac{1}{\pi} \log \frac{e t_0}{2\pi}   + \frac{4}{5 t_0} \left(\log t_0 +
\frac{41}{4}\right)\right) \left|\sigma-\frac{1}{2}\right|,\\
\err_{2,\sigma}(t_0,T) &= \left(2.78 \left|\sigma - \frac{1}{2}\right| + 1\right) \log \frac{T}{2\pi} - \frac{2}{5} +
\frac{\pi^2 t_0}{T} \left(\frac{2}{5}\log t_0 + 4\right). \end{aligned}\]
\end{lemma}
\begin{proof}
    Define $\phi(t) = \left|F\left(\frac{t}{T}\right) + \xi \left(1 -\frac{t}{T}\right) i\right|$ for $F$ as in \eqref{eq:kullervo}
    Then, by \eqref{eq:witdim} in Lemma~\ref{cor:thonny},
\begin{equation}\label{eq:tricme}\begin{aligned}
\sum_{\substack{\rho\\ t_0<\Im \rho \leq T}} \left|\omega_{T,\sigma}^+(\rho) 
+ \xi\theta_{T,1}(\rho) i\right| 
&\leq \sum_{\substack{\rho\\ t_0<\Im \rho \leq T}}
\left(\phi(\Im \rho) + \frac{\left|\sigma-\frac{1}{2}\right|T}{\pi (\Im \rho)^2} + \frac{ 2.78 \left|\sigma-\frac{1}{2}\right| + 1}{T}\right)
\\ & \leq   \sum_{t_0<\gamma\leq T} \phi(\gamma) + 
\frac{T}{\pi} \sum_{t_0<\gamma\leq T} \frac{\left|\sigma-\frac{1}{2}\right|}{\gamma^2}
+ 
c_\sigma \frac{N(T)}{T},\end{aligned}\end{equation}
where $c_\sigma = 2.78 \left|\sigma-\frac{1}{2}\right| + 1$ and $N(T)$ is the number of zeros of $\zeta(s)$ with imaginary part $0<\gamma\leq T$. We bound $N(T)$ by Cor.~\ref{cor:brut}
and the last sum in \eqref{eq:tricme} by Lemma \ref{lem:cathinv}. Their contribution to
\eqref{eq:amyar} is thus at most
\[\left|\sigma-\frac{1}{2}\right|
\left(\frac{\log \frac{e t_0}{2\pi}}{\pi t_0} + \frac{4}{5 t_0^2}
\left(\log t_0 + \frac{41}{4}\right)\right)
+ \frac{c_\sigma}{T} \log \frac{T}{2\pi}.\]

Since $F(\frac{t}{T})$ is decreasing on $(0,T]$, so is $\phi(t)$. By Lemma \ref{lem:Lehmanmodern}, 
\begin{equation}\label{eq:espresso}\sum_{t_0<\gamma\leq T} \phi(\gamma) = 
\dfrac{1}{2\pi}\int_{t_0}^{T}\phi(t)\log \frac{t}{2\pi}\, dt +
 \phi\left(t_0\right) \left(\frac{2}{5} \log t_0 +4\right) + \frac{1}{5} \int_{t_0}^T \phi(t) \frac{dt}{t}.\end{equation}
Proposition \ref{prop:tritura} tells us that 
 \[\frac{1}{2\pi} \int_{t_0}^{T}\phi(t)\log \frac{t}{2\pi}\, dt = 
\frac{T}{(2\pi)^2} \left(\log^2 \frac{T}{2\pi} -\log^2 \frac{t_0}{2\pi} - P_1\left(
\log \frac{T}{2\pi}\right)\right),\]
for 
$P_1(y) = 2 C_1 y + 2 (C_1-C_2)$, where 
 $C_1$, $C_2$ are as in Lemma \ref{lem:konstanz}.



Using part (a) of Lemma \ref{lem:sibelius} in \eqref{1_55pm}, we obtain $|F(u) + \xi(1-u)i| \leq \frac{1}{\pi u} + \frac{\pi(1-u)^2u}{2}<\frac{1}{\pi u} + \frac{\pi u}{2}$. Hence, $\phi(t_0)\leq \frac{T}{\pi t_0} + \frac{\pi t_0}{2 T}.$ 
On the other hand, by Lemma \ref{lem:nimes}, $\int_{t_0}^T \phi(t) \frac{dt}{t} < \frac{T}{\pi t_0} - \frac{1}{\pi}$.
Thus \[
 \phi\left(t_0\right) \left(\frac{2}{5} \log t_0 +4\right) + \frac{1}{5} \int_{t_0}^T \phi(t) \frac{dt}{t} \leq
  \left(\frac{T}{\pi t_0} + \frac{\pi t_0}{2 T}\right) \left(\frac{2}{5} \log t_0 + 
  4\right) + \frac{1}{5} \frac{T}{\pi t_0} - \frac{1}{5\pi}
 .\]

\end{proof}

\begin{lemma}\label{lem:salmon}
Let $T\geq 4 \pi$, $2\pi\leq t_0\leq \frac{T}{2}$, $\sigma\in \mathbb{R}$ with $|\sigma-\frac{1}{2}|\leq \frac{T}{2}$.
 Let $\omega_{T,\sigma}^+(s)$ and $\theta_{T,1}(s)$ be as in Theorem \ref{thm:mainthmB}. 
 Assume RH holds up to height $T$. Then, for any $\xi \in [-1,1]$,
 \[\frac{2\pi}{T} \sum_{\substack{\rho\\ 0<\Im \rho \leq t_0}} \left|\omega_{T,\sigma}^+(\rho) + \xi 
\theta_{T,1}(\rho) i \right| \leq  
 2  \sum_{\substack{\rho\\ 0<\Im \rho \leq t_0}} \frac{1}{|\rho-\sigma|} +
 \frac{c\,t_0}{T } \log \frac{t_0}{2\pi}, 
 \]
 where $c = 1 + \frac{1}{T} \left(2.78 \left|\sigma-\frac{1}{2}\right| + 1\right)$.
 \end{lemma}
 \begin{proof}
Write $\gamma = \Im \rho$.
By \eqref{eq:demoscen}, for $0<\gamma\leq \frac{T}{2}$,
 $\left|\omega_{T,\sigma}^+(s) + \xi \theta_{T,1}(s) i\right|
 \leq
 \left|\frac{T}{(s-\sigma) \pi}\right| + c$.
We apply Cor.~\ref{cor:brut} to bound
$\sum_{\rho: 0<\Im \rho\leq t_0} c = c N(t_0)$.
\end{proof}

Here is our main result on the contribution of non-trivial zeros. We do better with our weight than we would have done with the classical weight $\frac{1}{\rho}$, obtaining the
same main term $\frac{1}{2\pi} \log^2 \frac{T}{2\pi}$. 
\begin{proposition}\label{prop:vihuela} 
Let $T\geq 10^7$. Let $\omega_{T,\sigma}^+(s)$ and $\theta_{T,1}(s)$ be as in Thm.~\ref{thm:mainthmB}. Assume RH holds up to height $T$. 
For $\sigma\in \mathbb{R}$ with $\left|\sigma-\frac{1}{2}\right|\leq 100$, and any $\xi \in [-1,1]$,
\[\frac{2\pi}{T}
\sum_{\substack{\rho\\ 0<\Im \rho \leq T}} \left|\omega_{T,\sigma}^+(\rho) 
+ \xi  \theta_{T,1}(\rho) i\right|\leq 
\frac{1}{2\pi} \log^2 \frac{T}{2\pi} - \frac{1.01}{6\pi} \log \frac{T}{2\pi}.\]
\end{proposition}
\begin{proof}
Let $t_0 = 2\cdot 10^4$. Then, for $\err_{1,\sigma}(t_0)$ and $\err_{2,\sigma}(t_0,T)$ as in Lemma \ref{lem:adar},
\[\frac{\err_{1,\sigma}(t_0)}{t_0} \leq 8.162\cdot 10^{-4} + 1.444\cdot 10^{-4} \cdot \left|\sigma-\frac{1}{2}\right|,\]
\[\frac{\err_{2,\sigma}(t_0,T)}{T} \leq 1.404\cdot 10^{-6} + 3.97\cdot 10^{-6}\cdot \left|\sigma-\frac{1}{2}\right|.\]
Since $T\geq 10^7$, we see that
$\frac{t_0}{T} \log \frac{t_0}{2\pi}\leq 0.016132$ and
 \[\frac{1}{T} \left(2.78\left|\sigma-\frac{1}{2}\right| + 1\right) \frac{t_0}{T} \log \frac{t_0}{2\pi}\leq
 1.62\cdot 10^{-9} + 4.5\cdot 10^{-9} \cdot \left|\sigma-\frac{1}{2}\right|.\] 
 
By a brief computation using the location of all $\rho$ with $\gamma = \Im \rho \leq t_0$ (furnished by D.~Platt),
 \[2 \sum_{\substack{\gamma\\0<\gamma\leq t_0}} \frac{1}{\gamma} = 10.319317\dotsc  
 = \frac{1}{2\pi} \log^2 \frac{t_0}{2\pi} - 0.03435\dotsc.\]
We add the bounds from Lemmas \ref{lem:adar} and \ref{lem:salmon}, and conclude that
\begin{equation}\label{eq:amiaro}\begin{aligned}\frac{2\pi}{T}
\sum_{\substack{\rho\\ 0<\Im \rho \leq T}} \left|\omega_{T,\sigma}^+(\rho) 
+ \xi  \theta_{T,1}(\rho) i\right|&\leq
\frac{1}{2\pi }
 \left(\log^2 \frac{T}{2\pi} -\log^2 \frac{t_0}{2\pi} - P_1\left(
\log \frac{T}{2\pi}\right)\right) + \frac{1}{2\pi} \log^2 \frac{t_0}{2\pi}\\
&- 0.03435\dotsc + 0.0169\dotsc + 1.49\cdot 10^{-4} \cdot \left|\sigma-\frac{1}{2}\right|
,\end{aligned}\end{equation}
where $P_1(y)\geq 0.337876 y + 0.0095 \geq \frac{1.01 y}{3} + 0.0095$. By $|\sigma-\frac{1}{2}|\leq 100$, the negative term in the second line of \eqref{eq:amiaro} dominates the positive terms.
\end{proof}

\section{The case of $\Lambda(n)$: Integrals}\label{sec:horint}
\subsection{The integral on the real line} 
By a change of variables $\sigma = 1-t$,
\begin{equation*}
\int_{0}^\infty t|F(1-t+iT)|x^{-t}dt=\int_{-\infty}^1 \left|\frac{\zeta'}{\zeta}(\sigma+ i T)+\frac{1}{\sigma+i T-1}\right| (1-\sigma) x^{-(1-\sigma)} d\sigma
\end{equation*}
We first separate the logarithmic derivative $-\zeta'(s)/\zeta(s)$ of $\zeta(s)$ from the term $\frac{1}{s-1}$. Next, we split the integral of $-\zeta'(s)/\zeta(s)$ into  ranges $(-\infty, -\frac{1}{2}]$ and $[-\frac{1}{2}, 1]$. The estimate over the former follows directly from Lemma \ref{lem:adioso}. The latter requires more work, and we proceed differently. Our starting point is an explicit representation of $\zeta(s)$ in terms of its zeros -- an explicit version (Prop.~\ref{prop:titch96A}) of
\cite[Theorem 9.6(A)]{MR882550}. That expression will help us establish the following result.

\begin{lemma}
\label{lem:saghar}
 Let $x\geq e^7$ and $t\geq {1000}$. Write $L=\log x$. Let $a\in \big(0,\frac{1}{\sqrt{2}}\big]$. {If all zeros of $\zeta(s)$ with imaginary parts in $[t-a,t+a]$ have the form $\rho=\frac{1}{2}+i\gamma$},
\begin{equation}\label{eq:kemper}\begin{aligned}&\int_{-\frac{1}{2}}^1 \left|\frac{\zeta'}{\zeta}(\sigma+ i t)\right| (1-\sigma) x^{-(1-\sigma)} d\sigma \leq \frac{1}{\sqrt{x}}
\sum_{{|\gamma-t|\leq a}} \left(\frac{1}{2} \log \frac{1}{|t-\gamma|}+ \frac{1 + \frac{2}{L}}{2 L |t-\gamma|}\right)
\\
&+ \frac{1}{L^2} \left((c_{1,1} \log t + c_{1,0}) \left(\frac{1}{2} + \frac{2}{L}\right) + c_{0,1} \log t + c_{0,0}\right)
+ \sum_{{|\gamma-t|\leq a}}  \left(\frac{2}{L^2} + \frac{8}{L^3} + \frac{320}{3 L^4}\right),
\end{aligned}\end{equation}
where $c_{1,1} = \frac{1}{\pi a} + \frac{4}{5 a^2}$, $c_{1,0} = \frac{1}{\pi a} \log \frac{1}{2\pi} + \frac{8}{a^2}$,
$c_{0,1} = \frac{a}{\pi} - \frac{2}{5}$ and $c_{0,0} = \frac{a}{\pi} \log \frac{1}{2\pi} + 1.508 - 4$.
\end{lemma}
\begin{proof}
Let $\sigma_+ = \frac{3}{2}$. By Proposition~\ref{prop:titch96A}, for $s=\sigma+it$ with $\sigma\in [-\frac{1}{2},1]$,
\[\begin{aligned}
\frac{\zeta'}{\zeta}(s) & = \sum_{\rho:|\Im{\rho}-t|\leq a} \frac{1}{s-\rho} + O^*\left((c_{1,1}\log t + c_{1,0})(\tfrac{3}{2} - \sigma)
+ c_{0,1} \log t + c_{0,0}\right),
\end{aligned}\]
where we note that $\frac{\sigma_+-\sigma}{a^2}\geq \frac{1/2}{a^2} \geq 1 = \frac{1}{\sigma^+-\frac{1}{2}}$ and $\frac{\zeta'}{\zeta}\left(\frac{3}{2}\right) =  -1.50523\dotsc$. Now,
\[
\int_{-\frac{1}{2}}^1 (1-\sigma)x^{-(1-\sigma)}  d\sigma <\frac{1}{\log^2 x},\quad \int_{-\frac{1}{2}}^1 (1-\sigma)^2 x^{-(1-\sigma)} d\sigma < \frac{2}{\log^3 x}
\]
by integration by parts. Thus, {since $(\frac{3}{2}-\sigma)(1-\sigma)=\frac{1}{2}(1-\sigma)+(1-\sigma)^2$}, the contribution of  the
terms involving $c_{i,j}$ to the integral is at most
\[\frac{1}{\log^2 x} \left((c_{1,1} \log t + c_{1,0}) \left(\frac{1}{2} + \frac{2}{\log x}\right) + c_{0,1} \log t + c_{0,0}\right).\]
{By \eqref{eq:arguc2} in Lemma \ref{lem:rameau}}, the contribution of each term $\frac{1}{s-\rho}$ is
\[\int_{-\frac{1}{2}}^1 \frac{(1-\sigma) x^{-(1-\sigma)}}{\sqrt{|t-\gamma|^2 + \left(\sigma - \frac{1}{2}\right)^2}} d\sigma \leq
\frac{2}{L^2} + \frac{8}{L^3} + \frac{320}{3 L^4} +  
\frac{1}{\sqrt{x}} \left(\frac{1}{2} \log \frac{1}{|t-\gamma|}+ \frac{1 + \frac{2}{L}}{2 L |t-\gamma|}\right).
\]
\end{proof}

\begin{proposition}\label{prop:adoro}
Let $x\geq e^7$, $\alpha>0$, $0\leq a\leq \frac{1}{\sqrt{2}}$, $a+2\alpha\leq e$, $T_\circ> 1000+\alpha$. 
Assume all zeros of $\zeta(s)$ with imaginary part in $[T_\circ-\alpha-a,T_\circ+\alpha+a]$ have real part
$\frac{1}{2}$. Write $L = \log x$. Then there is $t\in [T_\circ-\alpha,T_\circ+\alpha]$ such that 
 \begin{equation}\label{eq:kimsch}
 \begin{aligned}
 &\int_{-\frac{1}{2}}^{1} \left|\dfrac{\zeta'}{\zeta}(\sigma+i t)\right|  (1-\sigma) x^{-(1-\sigma)} d\sigma\leq
 \left(\log \frac{e}{\alpha} +
\frac{\left(1+\frac{2}{L}\right) }{\alpha L} \log\left(2 (N_0+1)\right)\right) \frac{N_+}{\sqrt{x}}
\\
  + &\left(\frac{2}{L^2} + \frac{8}{L^3} + \frac{320}{3 L^4}\right) N_+ +
 (c_{1,1} \log (T_\circ+\alpha) + c_{1,0}) \left(\frac{1}{2 L^2} + \frac{2}{L^3}\right) + \frac{c_{0,1} \log (T_\circ+\alpha) + c_{0,0}}{L^2},
\end{aligned}\end{equation}
where $N_0 = N(T_\circ+\alpha)-N((T_\circ-\alpha)^-)$, $N_+ = N(T_\circ+\alpha+a) - N((T_\circ-\alpha-a)^-)$ 
and $c_{0,0}$, $c_{0,1}$, $c_{1,0}$ and $c_{1,1}$ are as in Lemma~\ref{lem:saghar}.
\end{proposition}
\begin{proof}
We will find $t$ such that the integral bounded in \eqref{eq:kemper} is small. (This means mainly controlling  terms
proportional to $\frac{1}{|t-\gamma|}$.)
Let us set up a pigeonhole argument of the continuous kind: we will show that a function has integral at most
$J$ on a set $S$ of measure $|S|$, and conclude that it must attain value $\leq J/|S|$ somewhere. Our set $S$ will be a subset of $[T_\circ-\alpha,T_\circ+\alpha]$; we will define it as the complement of the union of certain subsets of $[T_\circ-\alpha,T_\circ+\alpha]$, which we call ``forbidden''. Our integrand is the sum of the two sums on the right side of \eqref{eq:kemper}, or rather those two sums extended to all $\gamma\in [T_\circ-\alpha-a,T_\circ+\alpha+a]$, so that which $\gamma$ are in the sums does not depend on $t$.
(Since $a+2\alpha\leq e$, the contribution of each such $\gamma$ to the sums in \eqref{eq:kemper} is non-negative, simply
because the expression in \eqref{eq:arguc2} has to be non-negative, or else Lemma \ref{lem:rameau} would not hold.)

Let each $\gamma$ in $[T_\circ-\alpha-a,T_\circ+\alpha+a]$ forbid an interval $(\gamma-\Delta,\gamma+\Delta)
\cap [T_\circ-\alpha,T_\circ+\alpha]$. Each $\gamma$ counted by $N_0$
forbids an interval of width at most $2 \Delta$, and thus such $\gamma$ forbid a subset of measure $\leq 2\Delta N_0$. The intervals forbidden by all other $\gamma$ are contained in $[T_\circ-\alpha,T_\circ-\alpha+\Delta]$ and $[T_\circ+\alpha-\Delta,T_\circ+\alpha]$, and so their union has area at most $2 \Delta$. We let $\Delta = \alpha/(2 (N_0+1))$. Then the non-forbidden zone $S$ has area at least $2\alpha - 2\Delta (N_0 + 1) = \alpha$.

For given $\gamma\in [T_\circ-\alpha+\Delta,T_\circ+\alpha-\Delta]$, the integral of 
$\frac{1}{|t-\gamma|}$ for $t$ ranging on $S$ is
\[\leq \int_{[T_\circ-\alpha,T_\circ+\alpha]\setminus [\gamma-\Delta,\gamma+\Delta]} \frac{dt}{|t-\gamma|} \leq
\int_{[T_\circ-\alpha,T_\circ+\alpha]\setminus [T_\circ-\Delta,T_\circ+\Delta]} \frac{dt}{|t-T_\circ|} = 2 \log \frac{\alpha}{\Delta}.\]
(The second inequality holds by a simple argument we will use time and again: because our integrand is increasing on $\gamma$ for $\gamma<0$ and decreasing for $\gamma>0$, when we bound our integral by that for the central value $\gamma=T_0$,
the tail we lose is smaller than the tail we gain.)
Each of the other $\gamma\in [T_\circ-\alpha-a,T_\circ+\alpha+a]$ contributes at most $\int_{\Delta}^{2\alpha}  \frac{dt}{t} = \log \frac{2\alpha}{\Delta}$. Since $\Delta\leq \frac{\alpha}{2}$, we know that
$\log \frac{2\alpha}{\Delta} \leq 2 \log \frac{\alpha}{\Delta}$. Thus, every $\gamma$ contributes at most
$2 \log \frac{\alpha}{\Delta}$.

As for $\int_S \log \frac{1}{|t-\gamma|} dt$, we just bound it by
$\int_{T_\circ-\alpha}^{T_\circ+\alpha} \log \frac{1}{|t-\gamma|} dt$, which is at most $2 \int_0^\alpha \log \frac{1}{y} dy = 
2 \alpha \log \frac{e}{\alpha}$ for $\gamma$ arbitrary. 
Hence, the integral of the two sums in \eqref{eq:kemper}
for $t$ in $S$ is at most 
\[ \left(\left(\alpha \log \frac{e}{\alpha} +
\frac{1+\frac{2}{L}}{L} \log \frac{\alpha}{\Delta}\right) \frac{1}{\sqrt{x}} + |S| \left(\frac{2}{L^2} + \frac{8}{L^3} + \frac{320}{3 L^4}\right)\right)\cdot N^+.\]
By the argument we explained at the beginning, the statement follows.
\end{proof}

\begin{corollary}\label{cor:adiaro} Let $T\geq 10^6$. Assume all zeros of $\zeta(s)$ with imaginary part in $[T-\frac{3}{4},T]$ have real part $\frac{1}{2}$.
Let $x\geq 10^6$. Then there is $t\in [T-\frac{1}{2}, T-\frac{1}{4}]$ such that 
\[\int_{-\frac{1}{2}}^{1} \left|\dfrac{\zeta'}{\zeta}(\sigma+i t)\right|  (1-\sigma) x^{-(1-\sigma)} d\sigma\leq 
\frac{\frac{39}{5} \log T + 68}{\log^2 x} + \frac{37 \log T + 311}{\log^3 x}
+ \frac{\kappa(\log T,\log x)}{\sqrt{x}},\]
where $\kappa(R,L) = (R+7.3) \left(5\frac{\log R}{L}+2\right)$.
\end{corollary}
\begin{proof}
We apply Proposition \ref{prop:adoro} with $T_\circ = T-a-\alpha$ and
$\alpha,a>0$  to be chosen soon.
By Corollary~\ref{cor:zeroinbox}, $N_0\leq \frac{2}{5}\log T_\circ + \frac{\alpha}{\pi} \log \frac{T_\circ}{2\pi} + 4$ and
$N_+\leq \frac{2}{5} \log T_\circ + \frac{\alpha+a}{\pi} \log \frac{T_\circ}{2\pi} + 4$.
Since $T_\circ\geq 10^6- 2$, $N_0+1\leq \frac{2}{5} \log T_\circ + 5 + \frac{\alpha}{\pi}\log \frac{T_\circ}{2\pi}\leq
c_\alpha \log T_\circ$, where $c_\alpha = \frac{2}{5} + \frac{\alpha}{\pi} + \frac{5 - \frac{\alpha}{\pi}\log 2\pi}{\log(10^6-2)}$,
\[\log \frac{e}{\alpha} +
\frac{1+\frac{2}{L}}{\alpha L} \log(2 (N_0+1))\leq \frac{\kappa_1 \log \log T_\circ + \kappa_0}{\alpha}\]
for $\kappa_1 = (1+\frac{2}{L})\frac{1}{L}$, 
$\kappa_0= \kappa_1  \log 2 c_\alpha + \alpha \log \frac{e}{\alpha}$.
Write $\beta_1 = \frac{2}{5} +\frac{\alpha+a}{\pi}$,
$\beta_0 = 4 - \frac{\alpha+a}{\pi} \log 2\pi$, so that $N_+\leq \beta_1 \log T_\circ + \beta_0$. Write $R = \log (T_\circ+\alpha)$,
$L=\log x$. Then the expression in \eqref{eq:kimsch} is at most
$$\begin{aligned}
\frac{1/\alpha}{\sqrt{x}} (\beta_1 R+ \beta_0)
\left(\kappa_1 \log R + \kappa_0\right) + \frac{k_{2,1} R + k_{2,0}}{L^2} + \frac{k_{3,1} R + k_{3,0}}{L^3},
\end{aligned}
$$
where $k_{2,1} = 2\beta_1 +\frac{c_{1,1}}{2} + c_{0,1}$, $k_{2,0} =  2\beta_0 + \frac{c_{1,0}}{2}+c_{0,0}$, 
$k_{3,1} = \left(8 + \frac{320}{3 L}\right) \beta_1 + 2 c_{1,1}$,
$k_{3,0} = \left(8 + \frac{320}{3 L}\right) \beta_0 + 2 c_{1,0}$. We may write
$\frac{1}{\alpha} (\beta_1 R + \beta_0) (\kappa_1 \log R + \kappa_0)$ as 
$\left(R + \frac{\beta_0}{\beta_1}\right) \left(\frac{\kappa_1 \beta_1}{\alpha} \log R + \frac{\kappa_0 \beta_1}{\alpha}\right)$.

We choose $a=\frac{1}{4}$, $\alpha=\frac{1}{8}$. Then $\frac{\beta_0}{\beta_1} =  7.279\dotsc$, $\frac{\kappa_1 \beta_1}{\alpha} \leq \frac{4.756\dotsc}{L}$,
$\frac{\kappa_0 \beta_1}{\alpha} \leq 1.76$,
$k_{2,1} = 7.754\dotsc$, $k_{2,0} = 67.752\dotsc$, $k_{3,1} = 36.311\dotsc$ and
$k_{3,0} = 310.75\dotsc$.

\end{proof}

\begin{remark}
We could prove a version of Prop.~\ref{prop:adoro} and Cor.~\ref{cor:adiaro} without the condition that the zeros of $\zeta(s)$ with imaginary part in
$[T-1,T]$ obey RH. We would then obtain a bound proportional to $\frac{\log T \log \log T}{\log^2 x}$, which would be acceptable to us.
However, we assume RH elsewhere up to height $T$ anyhow, and using the RH assumption barely takes more work (Lemmas \ref{lem:bettnist}--\ref{lem:rameau}).
\end{remark}

\subsection{The integral over $\mathcal{C}$}
\label{subs:tripa}
We will now estimate the contribution of the contour $\mathcal{C}$ in Thm.~\ref{thm:mainthmB} for $A(s) = - \zeta'(s)/\zeta(s)$. As one can tell
from the statement of Thm.~\ref{thm:mainthmB}, we can choose $\mathcal{C}$ rather
freely; it is just about any contour from $1$ to $\Re s = -\infty$ close enough to the $x$-axis to go under the non-trivial zeros of $\zeta(s)$ yet above the trivial
zeros of $\zeta(s)$. 

We choose the initial segment of $\mathcal{C}$ to go from $1$ to $-1$.
Then we split the integrand $F(s) = A(s)-1/(s-1)$ into two parts, one being
$\frac{\pi}{2} \cot \frac{\pi s}{2}$, and the other one being
the remainder $F(s) - \frac{\pi}{2} \cot \frac{\pi s}{2}$, which, by the
functional equation, is holomorphic on the left half of the plane. 
We shift the contour for the remainder integral
to the straight line from $-1$ to $-\infty$. The contour $\mathcal{C}_<$ for the first part
has to stay away from the $x$-axis; we will take it to be a straight segment from $-1$ to $-2 + i$, followed by a half-line.

\begin{figure}[ht]
  \centering
\begin{tikzpicture}[scale=1]
  \draw[gray,->] (-7,0) -- (1.5,0) node[below right] {$\Re s$};
  \draw[gray,->] (0,-0.5) -- (0,1.5) node[left] {$\Im s$};

  \foreach \x in {-6,-4,-2} {
    \draw[thick] (\x-0.1,-0.1) -- (\x+0.1,0.1);
    \draw[thick] (\x-0.1,0.1)  -- (\x+0.1,-0.1);
  }

  \draw[blue,thick,postaction={decorate},
        decoration={markings, mark=at position 0.5 with {\arrow{>}}}]
        (1,0) -- (-1,0);
  \draw[blue,thick,postaction={decorate},
        decoration={markings, mark=at position 0.5 with {\arrow{>}}}]
        (-1,0) -- (-2,1);
  \draw[blue,thick,->,postaction={decorate},
        decoration={markings, mark=at position 0.5 with {\arrow{>}}}]
        (-2,1) -- (-7,1) node[above right] {$\mathcal{C}_<$};

  \draw[red,thick,->,postaction={decorate},
    decoration={markings, mark=at position 0.4 with {\arrow{>}},
      mark=at position 0.65 with {\arrow{>}}}]
        (-1,0) -- (-7,0);

  \fill (-1,0) circle (.1) node[above right] {$-1$};
  \fill (1,0) circle (.05) node[above right] {$1$};
  \fill (-2,1) circle (.05) node[above] {$-2+i$};

\end{tikzpicture}
\end{figure}

The choice of contours is motivated by convenience.  We integrate $F(s) - \frac{\pi}{2} \cot \frac{\pi s}{2}$ horizontally because it seems simplest; the only complication is the
fact that the integrand is unbounded, but that would have been the case at any rate, as
the digamma function
$\digamma(s)$ is unbounded as $\Re s\to \infty$. As for the contribution of
$\frac{\pi}{2} \cot \frac{\pi s}{2}$: we want the angle
to be acute for $x^s$ to decay, and 
$45^\circ$ is the smallest angle for which a bound (Lemma \ref{lem:adamant}) holds. Once we
are far enough from the pole at $-2$, we continue horizontally.

\subsubsection{The initial segment of $\mathcal{C}$}
The integral on the segment from $1$ to $-1$ is easy; following a pattern that will repeat,
we will 
use a technical lemma (Lem.~\ref{lem:convexistan}), and use it to prove the integral estimate we need in the next lemma (Lem.~\ref{lem:arles}). We could
refine Lemma \ref{lem:arles} to give as many ``correct'' terms $(n+1)! \frac{a_n x}{\log^{n+2} x}$ as requested, but what we give will do nicely.

\begin{lemma}\label{lem:arles}
Let $\tilde{A}(s) = -\frac{\zeta'(s)}{\zeta(s)} - \frac{1}{s-1}$. Let $x>1$. Then
 $\tilde{A}(s)<0$ for all $-2<s\leq 1$, and
\[- \int_{-1}^1 \tilde{A}(s) (1-s) x^s ds\leq \frac{\gamma x}{\log^2 x} +
\frac{c-\gamma}{2} \frac{x}{\log^3 x} - \frac{c+\gamma}{x \log x} - \frac{c}{x \log^2 x} - \frac{c-\gamma}{2 x \log^3 x},
\]
where $\gamma=0.577215\dotsc$ is Euler's constant and
$c =  \frac{\zeta'}{\zeta}(-1) - 2 \left(\frac{\zeta'}{\zeta}\right)'(-1) = 
3.86102\dotsc$. 
\end{lemma}
\begin{proof}
  Write $-\int_{-1}^1 \tilde{A}(s) (1-s) x^s ds = - \int_0^2 \tilde{A}(1-t) t x^{1-t} dt$.
  By Lemma \ref{lem:kalmynin}, $\tilde{A}(s)$ is of the form $\sum_{n=0}^\infty (-1)^{n+1} a_n (s-1)^n$, $a_n> 0$. Hence, $- \tilde{A}(1-t) = \sum_{n=0}^\infty a_n t^n$, and so, for all $t\geq 0$ within the radius of convergence, i.e., 
  $0\leq t<3$, all derivatives of
  $- \tilde{A}(1-t) t$ are increasing; moreover,
  $\tilde{A}(1-t)<0$.
  Apply Lemma \ref{lem:convexistan} with $G(t) = - t \tilde{A}(1-t)$, $a=2$.
    Note that
    $G'(0) = - \tilde{A}(1) = a_0 = \gamma$ and 
    $G'(2) = \frac{\zeta'}{\zeta}(-1) - 2 \left(\frac{\zeta'}{\zeta}\right)'(-1)$.
\end{proof}

\subsubsection{The horizontal contour from $-1$ to $-\infty$}
We take $-1$ as the point where our contour forks because a term coming from the functional equation then vanishes, due to $\cos\left(-\frac{\pi}{2}\right) = 0$.

\begin{lemma}\label{lem:moruno}
Let $\tilde{A}(s) = -\frac{\zeta'}{\zeta}(s) - \frac{1}{s-1}$. Let $x\geq 2$. Then,
for any $\Phi:(-\infty,-1]\to \mathbb{R}$ such that $-(1-s)\leq \Phi(s)\leq 0$ for all $s\leq -1$,
\begin{equation}\label{eq:avignon1}
\left|\int_{-\infty}^{-1} \left(\tilde{A}(s) + \frac{\pi}{2} \cot \frac{\pi s}{2}\right) \Phi(s) x^s ds
\right|\leq \frac{-2  \tilde{A}(-1)}{x \log x},
\end{equation}
\end{lemma}
\begin{proof}
By the functional equation (Lemma \ref{lem:derivbound}),
\begin{equation}\label{eq:guruno}\tilde{A}(s) + \frac{\pi}{2} \cot \frac{\pi s}{2} = 
\frac{\zeta'}{\zeta}(1-s) +
\digamma(1-s) + \frac{1}{1-s} - \log 2\pi. 
\end{equation}
Let $f(t)$ be as in Lemma \ref{lem:badabook}.
Then $\tilde{A}(s) + \frac{\pi}{2} \cot \frac{\pi s}{2} = f(1-s) + \frac{1}{1-s}$. 
In particular, for $s=-1$, $\tilde{A}(-1) = f(2)+1/2$. 
Write $g(t)= f(t)+1/t$ and  $\tilde{\Phi}(t) = - \Phi(1-t)$. Our task is to bound 
\begin{equation}\label{eq:omelette}
I = \int_{-\infty}^{-1} 
\left(\tilde{A}(s) + \frac{\pi}{2} \cot \frac{\pi s}{2}\right) (-\Phi(s)) x^s ds
= \int_2^\infty g(t) \tilde{\Phi}(t) x^{1-t} dt\end{equation}
from below and from above.
By Lemma \ref{lem:badabook} again,  
$t g(t) = t (f(t)+1/t)\geq 2 f(2) + 1 = 2 \tilde{A}(-1)$
for any $t\geq 2$. (Note that
$\tilde{A}(-1) = -1.48505\dotsc<0$.)
Hence, by 
$\tilde{\Phi}(t) = - \Phi(1-t)\leq t$ for $t\geq 2$,
$$\int_2^\infty \min\left(g(t),0\right) \tilde{\Phi}(t)x^{1-t} dt \geq
\int_2^\infty \min(g(t),0) t x^{1-t} dt \geq
2 \tilde{A}(-1) \int_2^\infty x^{1-t} dt = \frac{2 \tilde{A}(-1)}{x \log x},$$
and so, by $\tilde{\Phi}(t)\geq 0$ for $t\geq 2$, 
$I\geq \frac{2 \tilde{A}(-1)}{x \log x}$.

Let us now prove an upper bound. By the concavity proved in Lemma \ref{lem:badabook},
$g(t) \leq g(a) + m (t-a)$ for all $t>1$, where $a$ is any real $>1$ 
and $m=g'(a)$. We will choose $a> 2$ such that $g(a)<0$; then
$g(t)<0$ for all $2\leq t\leq a$. Since $\tilde{\Phi}(t)\geq 0$ for all $t\geq 2$,
$\int_2^a g(t) \tilde{\Phi}(t) x^{1-t} dt \leq 0$.
By $0\leq \tilde{\Phi}(t)\leq t$,
$$\int_a^\infty g(t) \tilde{\Phi}(t) x^{1-t} dt \leq\int_a^\infty m (t-a) \tilde{\Phi}(t) x^{1-t} dt \leq
m \int_a^\infty (t-a) t x^{1-t} dt = \frac{m a + \frac{2 m}{\log x}}{x^{a-1} \log^2 x} .
$$
We can take $a=5$, and so $m = 0.203\dotsc$. Then, for
$x\geq 2$, $\frac{m a + \frac{2 m}{\log x}}{x^{a-1} \log^2 x}<
\frac{1.61}{x^4 \log^2 x} < -\frac{2 \tilde{A}(-1)}{x \log x}$.
\end{proof}

\subsubsection{The integral over $\mathcal{C}_<$} 
\begin{lemma}\label{lem:adamant}
Let $g(t) = \tan\left(e^{\frac{3\pi i}{4}} t\right)$. Then $g(0)=0$ and
$|g'(t)|\leq 1$ for all real $t$, and so $|g(t)|\leq |t|$ for all
real $t$.
\end{lemma}
\begin{proof}
Since $\tan' z = \sec^2 z$ and
$\left|\cos z\right|^2 = \frac{1}{2} (\cos(2 x) + \cosh(2 y))$ for $z= x+i y$,
the statement follows from
$\cos(-2 y) + \cosh 2 y= 2+ 2 \sum_n (2 y)^{4 n}/(4 n)!\geq 2$.
\end{proof}

\begin{lemma}\label{lem:ranadi}
Let $\mathcal{C}_{<}$ be the contour going on straight lines from
$-1$ to $-2 + i $ and from there to 
$- \infty + i$. Let $\Phi(s)$ be 
a holomorphic function in a neighborhood of $\mathcal{C}_{<}$,
satisfying $|\Phi(s)|\leq |s-1|$ and $|\Phi'(s)|\leq 1$
on $\mathcal{C}_{<}$.
Let $x>1$. Then
\begin{equation}\label{eq:amira}\left|\int_{\mathcal{C}_{<}} \frac{\pi}{2} \cot \frac{\pi s}{2} \cdot \Phi(s) x^s ds\right|\leq \frac{\pi^2}{4 x} \left(\frac{2 \sqrt{2}}{\log^2 x} +
\frac{2 + \sqrt{2}}{\log^3 x} + \frac{1/\sqrt{2}}{\log^4 x}\right).
\end{equation}
\end{lemma}
In fact, it would be enough for $\Phi$ to be defined on $\mathcal{C}_{<}$
as a $C^1$ function.
On another matter: \eqref{eq:amira} could be improved by a factor of about $\sqrt{2}$ if we
made assumptions on $\Phi''$.
\begin{proof}
Denote by $\mathcal{C}_1$ the first segment of $\mathcal{C}_{<}$,
going from $-1$ to $-2 + i$. By $\cot(z-\pi/2) = 
- \tan z$ and integration by parts,
\begin{equation}\label{eq:habeni}\begin{aligned} I_1 &:= \int_{\mathcal{C}_1} \frac{\pi}{2} \cot \frac{\pi s}{2} \cdot \Phi(s) x^s ds = 
\frac{\pi}{2} \int_0^{-1+i} -\tan \frac{\pi s}{2}\cdot 
\Phi(s-1) x^{s-1} ds\\
&= \frac{\pi}{2} \cot \frac{\pi (-2+i)}{2}\cdot 
\Phi(-2+i) \frac{x^{-2+i}}{\log x} +
\frac{\pi}{2} 
\int_0^{-1 + i} \left(\tan \frac{\pi s}{2}\cdot 
\Phi(s-1)\right)' \frac{x^{s-1}}{\log x} ds
.\end{aligned}\end{equation}
For $s$ of the form $e^{\frac{3\pi}{4} i} t$, $t\geq 0$,
 Lemma \ref{lem:adamant}, gives us $\left|\left(\tan \frac{\pi s}{2}\right)'\right| \leq \frac{\pi}{2}$
and $\left|\tan \frac{\pi s}{2}\right| \leq 
\frac{\pi |s|}{2}$. Thus,
$\left|\left(\tan \frac{\pi s}{2}\cdot 
\Phi(s-1)\right)'\right|\leq \frac{\pi}{2} (|\Phi(s-1)| +
|s| |\Phi'(s-1)|) \leq \frac{\pi}{2} (|s-2|+|s|)$.
Clearly
$$\left|2- e^{\frac{3\pi i}{4}} t\right| = \sqrt{\left(2+\frac{t}{\sqrt{2}}\right)^2 + \left(\frac{t}{\sqrt{2}}\right)^2}
=\sqrt{4 + \sqrt{8} t + t^2} \leq 2  + \frac{t}{\sqrt{2}} + \frac{t^2}{8},$$
and so, by a change of variables followed by repeated
integration by parts,
\begin{equation}\label{eq:benidorm}\begin{aligned}
&\left|\int_0^{-1 + i} \left(\tan \frac{\pi s}{2}\cdot 
\Phi(s-1)\right)' \frac{x^{s-1}}{\log x} ds\right|\leq
\frac{\pi}{2} \int_0^{\sqrt{2}}
\left(2 + \left(1 + \frac{1}{\sqrt{2}}\right) t + \frac{t^2}{8}\right)
\frac{x^{-1-\frac{t}{\sqrt{2}}}}{\log x} dt\\
&< \frac{\pi}{2 x} \left(\frac{2 \sqrt{2}}{\log^2 x} +
\frac{2 + \sqrt{2}}{\log^3 x} + \frac{1/\sqrt{2}}{\log^4 x}\right)- \frac{\pi}{2 x^2} \left(\frac{2\sqrt{2} + 2}{\log^2 x} +\frac{2}{\log^3 x}\right).
\end{aligned}\end{equation}

Let $\mathcal{C}_2$ be the second part of $\mathcal{C}_<$, that
is, a segment from $-2+i$ to $-\infty + i$,
\begin{equation}\label{eq:hadorm}\int_{\mathcal{C}_2} \frac{\pi}{2}
 \cot \frac{\pi s}{2} \cdot \Phi(s) x^s ds =
 -\frac{\pi}{2} \cot \frac{\pi (-2+i)}{2}\cdot
\Phi(-2+i) \frac{x^{-2+i}}{\log x} - \frac{\pi/2}{ \log x}
\int_{\mathcal{C}_2} 
 \left(\cot \frac{\pi s}{2}\cdot \Phi(s)\right)' x^s ds 
\end{equation}

In general, for $x$, $y$ real with $y>0$, $|\cot(x+i y)|\leq \coth y$
and $|\cot'(x+i y)| = |-\csc^2(x+i y)|\leq \csch^2 y$. Hence, for $y=1$,
$\left|\left(\cot \frac{\pi s}{2}\cdot 
\Phi(s)\right)'\right|\leq 
\frac{\pi}{2}|\Phi(s)| \csch^2 \frac{\pi}{2} +
|\Phi'(s)| \coth \frac{\pi}{2} \leq \frac{\pi}{2} |s-1| \csch^2 \frac{\pi}{2} +
\coth \frac{\pi}{2}$.
By the triangle inequality,  
$|s-1|\leq |t| + |-2 + i -1| = -t + \sqrt{10}$
for $s = t -2 + i$ with $t\leq 0$. Therefore, $\int_{\mathcal{C}_2} \left| 
 \left(\cot \frac{\pi s}{2}\cdot \Phi(s)\right)' x^s ds 
\right|$ is bounded by
$$\begin{aligned}
 \frac{\pi}{2} \csch^2 \frac{\pi}{2}
\int_{-\infty}^0 (-t+\sqrt{10}) x^{t-2} dt + \coth \frac{\pi}{2}
\int_{-\infty}^0 x^{t-2} dt
&= \frac{\frac{\pi}{2} \csch^2 \frac{\pi}{2} \cdot \left(\sqrt{10}+\frac{1}{\log x}\right)
+\coth \frac{\pi}{2}}{x^2 \log x} 
.\end{aligned}$$
Since $\frac{\pi}{2} \cdot \left(2\sqrt{2} + 2\right)
> \frac{\pi}{2} \csch^2 \frac{\pi}{2} \cdot \sqrt{10} + \coth \frac{\pi}{2}$ and
$\frac{\pi}{2} \cdot 2 >
\frac{\pi}{2}\csch^2 \frac{\pi}{2}$, this last contribution is dominated by the negative terms from \eqref{eq:benidorm}.
The first term in the last line of \eqref{eq:habeni} -- that
is, the tail term from $\mathcal{C}_1$ -- cancels out the first
term from \eqref{eq:hadorm}, that is, the ``head'' term from
$\mathcal{C}_2$.
\end{proof}

\subsubsection{Summing up: the integral over $\mathcal{C}$}

\begin{proposition}\label{prop:coronidis}
  Let $\tilde{A}(s) = -\frac{\zeta'(s)}{\zeta(s)} - \frac{1}{s-1}$. Let $x\geq 2$.
  Let $\mathcal{C}$ run along straight lines from $1$ to $-1$, from
  $-1$ to $-2+i$ and from $-2+i$ to $-\infty+ i$.
 Let $\Phi$ be holomorphic in a neighborhood of
  $(-\infty,1]+ i [0,1]$ and satisfying 
  $\Phi(1) = 0$ and $|\Phi'(s)|\leq 1$, with
  the restriction $\Phi|_{(-\infty,1)}$ being real and of constant sign. Then, for any $x\geq 15$,
\begin{equation}\label{eq:adosto}
\left|\int_{\mathcal{C}} \tilde{A}(s) \Phi(s) x^s ds\right|\leq
\frac{\gamma x}{\log^2 x} + \frac{\frac{5}{3} x}{\log^3 x},
\end{equation}
where $\gamma$ is Euler's constant.
\end{proposition}
The region of holomorphicity here is of course larger than needed.
\begin{proof}
Since $|\Phi'(s)|\leq 1$ for all $s$ and $\Phi(1) = 0$, clearly $|\Phi(s)|\leq |s-1|$. Then,
by Lemma \ref{lem:arles},
\begin{equation}\label{eq:adaid}\int_{-1}^1 |\tilde{A}(s) \Phi(s)| x^s ds 
\leq \frac{\gamma x}{\log^2 x} +
\frac{c-\gamma}{2} \frac{x}{\log^3 x} - \frac{c+\gamma}{x \log x} - \frac{c}{x \log^2 x} - \frac{c-\gamma}{2 x \log^3 x},
\end{equation}
where $c$ is as in Lemma \ref{lem:arles}.
Let $\mathcal{C}_<$ run along straight lines from $-1$ to $-2 + i$ and
from $-2+i$ to $-\infty + i$. We separate a regular term
 and shift its contour to the $x$-axis:
$$\int_{\mathcal{C}_<} \tilde{A}(s) \Phi(s) x^s ds = 
-\int_{\mathcal{C}_<} \frac{\pi}{2} \cot \frac{\pi s}{2}\cdot \Phi(s) x^s ds
+ \int_{-1}^{-\infty} \left(\tilde{A}(s) + 
\frac{\pi}{2} \cot \frac{\pi s}{2}\right) \Phi(s) x^s ds,
$$
where we know that $\tilde{A}(s) + \frac{\pi}{2} \cot \frac{\pi s}{2}$ has no poles with $\Re s < 0$ by the functional equation, as in \eqref{eq:guruno}.
By Lemma \ref{lem:ranadi},
$$\left|\int_{\mathcal{C}_{<}} \frac{\pi}{2} \cot \frac{\pi s}{2} \cdot \Phi(s) x^s ds\right|\leq \frac{\pi^2}{4 x} \left(\frac{2 \sqrt{2}}{\log^2 x} +
\frac{2 + \sqrt{2}}{\log^3 x}+ \frac{1/\sqrt{2}}{\log^4 x}\right).$$
We know $\Phi|_{(-\infty,1)}$ is of constant sign; we
can assume that sign to be $-1$. Then, by Lemma \ref{lem:moruno},
\begin{equation}\label{eq:adoid} \left|\int_{-\infty}^{-1} \left(\tilde{A}(s) + \frac{\pi}{2} \cot \frac{\pi s}{2}\right) \Phi(s) x^s ds\right|\leq \frac{-2\cdot  \tilde{A}(-1)}{x \log x}.
\end{equation}

We take totals. The term proportional to $\frac{1}{x\log x}$ will be negative:
its coefficient is $-(c+\gamma)-2\cdot \tilde{A}(-1) = -1.468\dotsc$. As for the other terms,
$-c + \frac{\pi^2}{4} 2\sqrt{2} < 3.118$, and
$- \frac{c-\gamma}{2} + \frac{\pi^2}{4} (2+\sqrt{2}) < 6.79$,
$\frac{\pi^2}{4\sqrt{2}} < 1.75$. Note that $\frac{c-\gamma}{2} = 1.6419\dotsc$.
For $x\geq 15$,
$$\left(\frac{c-\gamma}{2} - \frac{5}{3}\right) \frac{x^2}{\log^2 x} 
-1.468 + \frac{3.118}{\log x} + \frac{6.79}{\log^2 x} + \frac{1.75}{\log^3 x} < 0,$$
as the inequality holds for $x=15$, and the left side increases for $x\geq e$.
\end{proof}

\subsection{Bounding the total $I_{+,\mathcal{C}}$}

\begin{lemma}\label{lem:hardin}
    Let $T\geq 10^6$. 
     Assume all zeros of $\zeta(s)$ with imaginary part in 
     $[T-\frac{3}{4},T]$ have real part $\frac{1}{2}$.
Let $x\geq T$. Let $I_{+,\mathcal{C}}$ be as in Theorem~\ref{thm:mainthmB}.
Then there is $t\in [T-\frac{1}{2}, T]$ such that
\[I_{+,\mathcal{C}}(t)\leq \left(13 + \frac{60}{\log x}\right) \frac{\log T}{\log^2 x} 
+ \frac{5 \log T}{\sqrt{x}}.\]
\end{lemma}
\begin{proof}
By Corollary~\ref{cor:adiaro}, there is a $t\in [T-\frac{1}{2},T]$ such that
\[\int_{-\frac{1}{2}}^{1} \left|\dfrac{\zeta'}{\zeta}(\sigma+i  t)\right|  (1-\sigma) x^{-(1-\sigma)} d\sigma\leq 
\frac{\frac{39}{5} \log T + 68}{\log^2 x} + \frac{37 \log T + 311}{\log^3 x} + \frac{\kappa(\log T,\log x)}{\sqrt{x}},
\]
where $\kappa(R,L) = (R+7.3) \left(5\frac{\log R}{L}+2\right)$. Here we may bound $\frac{39}{5} + \frac{68}{\log T}<
\frac{51}{4}$, $37 + \frac{311}{\log T}<60$
and $\kappa(\log T,\log x)< \frac{32}{7} \log T$.
Write $L = \log x$. By Lemma \ref{lem:adioso},
\begin{align*}
\int_{-\infty}^{-\frac{1}{2}}\left|\dfrac{\zeta'}{\zeta}(\sigma+i t)\right| (1-\sigma)x^{-(1-\sigma)} d\sigma  &\leq 
 \frac{4}{3} \left(\log t + c_{\frac{1}{2}} + \frac{3}{2 t^2}\right) \frac{1}{x \log x}
 \leq \frac{2.02}{x},
\end{align*}
where
$c_{1/2} = \left|\frac{\zeta'}{\zeta}(3/2)\right|+4+\frac{\pi}{2} = 7.076\dotsc$.
By integration by parts,
\[\int_{-\infty}^{1} \left|\frac{1}{1-(\sigma+i t)}\right|  (1-\sigma) x^{-(1-\sigma)} d\sigma \leq
\frac{1}{t} \int_{-\infty}^{1} (1-\sigma) x^{-(1-\sigma)} d\sigma = \frac{1}{t \log^2 x}.\]
We conclude that, for $F(s) = -\frac{\zeta'(s)}{\zeta(s)} - \frac{1}{s-1}$,
\[\int_{-\infty}^{1} \left|F(\sigma+i t)\right|  (1-\sigma) x^{-(1-\sigma)} d\sigma\leq 
\left(\frac{51}{4} + \frac{60}{\log x}\right) \frac{\log T}{\log^2 x} + \frac{1}{(T-\frac{1}{2})\log^2 x} + \frac{\frac{14}{3} \log T}{\sqrt{x}} + \frac{2.02}{x}.\]
By Proposition \ref{prop:coronidis}, whose assumptions on $\Phi(s)$ are fulfilled by Thm.~\ref{thm:mainthmB},
\[\left|\int_{\mathcal{C}} F(s) \Phi(s) x^{s-1} ds\right|\leq
\frac{\gamma}{\log^2 x} + \frac{5/3}{\log^3 x}.\]

\end{proof}

\section{The case of $\Lambda(n)$: conclusion}\label{sec:lambconcl}

We come to our main bounds for the sums $\sum_{n\leq x} \Lambda(n) n^{-\sigma}$.

\begin{proposition}\label{prop:sagaro}
 Assume the Riemann hypothesis holds up to height $T\geq 10^7$. Then, for any $x> \max(T,10^9)$, and any 
 $-1.999\leq\sigma\leq 100$,
   \[\left|\frac{1}{{x^{1-\sigma}}}\,{\sum_{n\leq x} \Lambda(n) n^{-\sigma}}  - \mathrm{Main}(x,\sigma) \right|\leq
   \frac{\pi}{T-1} + \left(\frac{1}{2\pi} \log^2 \frac{T}{2\pi} - \frac{1}{6\pi} \log \frac{T}{2\pi}\right) \frac{1}{\sqrt{x}}
\]
where
\[\mathrm{Main}(x,\sigma) =
\begin{cases} \frac{\pi}{T} \coth \frac{\pi (1-\sigma)}{T} - \frac{\zeta'(\sigma)}{\zeta(\sigma)}x^{\sigma-1}&\text{if $\sigma\ne 1$,}\\
\log x - \gamma &\text{if $\sigma=1$.}\end{cases}\]
If $\sigma=0$, the term $ - \frac{\zeta'(\sigma)}{\zeta(\sigma)} x^{\sigma-1}$ can be omitted.
\end{proposition}
\begin{proof}
Let $T'=t$, where $t\in [T-\frac{1}{2},T]$ is as in Lemma~\ref{lem:hardin}.
Apply Thm.~\ref{thm:mainthmB} with $A(s) = - {\zeta'(s)}/{\zeta(s)}$ and $T'$ instead of $T$.
The poles of $A(s)$ are the zeros of $\zeta(s)$ and the pole of $\zeta(s)$ at $s=1$.
The residue of $A(s)$ at a zero of $\zeta(s)$ is $-1$ times the zero's multiplicity, and its residue at $s=1$ is $1$.

The real poles $\mathcal{Z}_{A,\mathbb{R}}$ of $A(s)$ lie at the trivial zeros
$\rho = - 2 n$ of $\zeta(s)$, and of course at the pole at $s=1$. 
If $\sigma\ne 1$, then the contribution of the pole at $s=1$ to the first and third 
sums in \eqref{eq:quentino} is
\[\frac{\pi}{T'} \left(\coth \frac{\pi (1-\sigma)}{T'} + O^*(1)\right),\]
whereas the contribution to the first sum of the pole at $s=\sigma$ coming from the weight
$\coth \frac{\pi(s-\sigma)}{T'}$ is simply $A(\sigma) x^{\sigma-1}$. If $\sigma=1$,
then $\coth \frac{\pi (s-\sigma)}{T'} A(s) x^{s-1}$ has a double pole at $s=1$; as we already saw in \eqref{eq:kelkap},
its Laurent series at $s=1$ is
\[\begin{aligned}\left(\frac{T'}{\pi (s-1)} + O(s-1) \right) &\left(\frac{1}{s-1} - \gamma + \dotsc\right) \left(1 + (s-1) \log x + \dotsc\right)\\
&= \frac{T'}{\pi} \left(\frac{1}{(s-1)^2} + \frac{\log x - \gamma}{s-1}  + \dotsc\right),\end{aligned}\] and so the contribution of $\rho=1$ to the first and third sums in 
\eqref{eq:quentino} is $\log x - \gamma + O^*\left(\frac{\pi}{T'}\right)$.

By Lemma \ref{lem:trivialzer}, the rest of the first and third sums is at most 
$\frac{\frac{1}{2+\sigma} + \frac{2\pi}{T'}}{x^3 (1 - x^{-2})}$.
By Proposition~\ref{prop:vihuela}, the second and the fourth sums add up to at most $\frac{1}{2\pi} \log^2 \frac{T'}{2\pi} - \frac{1.01}{6\pi} \log \frac{T'}{2\pi}$.  
Finally, we bound $I_{+,\mathcal{C}}$ as in Lemma \ref{lem:hardin}.
We conclude that
\[\begin{aligned}
    \frac{\sum_{n\leq x} \Lambda(n) n^{-\sigma}}{x^{1-\sigma}} &=
\begin{cases} \frac{\pi}{T'} \coth \frac{\pi (1-\sigma)}{T'} - \frac{\zeta'}{\zeta}(\sigma) x^{\sigma-1}&\text{if $\sigma\ne 1$,}\\
\log x - \gamma &\text{if $\sigma=1$}\end{cases}\\
&+  O^*\left(\frac{\pi}{T'}\right) 
+ O^*\left(26 + \frac{120}{\log x}\right) \frac{\log T'}{(T')^2 \log ^2 x}\\
&+  O^*\left(\left(\frac{1}{2\pi} \log^2 \frac{T'}{2\pi} - \frac{1.01}{6\pi} \log \frac{T'}{2\pi} + \frac{10 \log T'}{(T')^2} + \frac{\frac{1}{2+\sigma} + \frac{2\pi}{T'}}{x^{\frac{5}{2}} (1 - x^{-2})}\right) \frac{1}{\sqrt{x}}\right).
\end{aligned}\]
Both $\frac{10 \log T'}{(T')^2}$ and $\frac{\frac{1}{2+\sigma} + \frac{2\pi}{T'}}{x^{5/2} (1 - x^{-2})}$ 
are less than $\frac{0.0001}{6\pi} \log \frac{T'}{2\pi}$.
Note $\frac{\zeta'}{\zeta}(0) = \log 2\pi < \frac{0.0098}{6\pi} \log \frac{T'}{2\pi} \cdot \sqrt{x}$.

Clearly, $\frac{\pi}{T'}\leq \frac{\pi}{T-\frac{1}{2}}$  and 
$\frac{1}{2\pi} \log^2 \frac{T'}{2\pi} - \frac{1}{6\pi} \log \frac{T'}{2\pi}\leq
\frac{1}{2\pi} \log^2 \frac{T}{2\pi} - \frac{1}{6\pi} \log \frac{T}{2\pi}$. 
By $x\geq T\geq 10^7$, $x\geq 10^9$ and $T'\in [T-\frac{1}{2},T]$, 
$\left(26 + \frac{120}{\log x}\right) \frac{\log T'}{(T')^2 \log ^2 x}$ is at most $\frac{1.535}{T^2}$. Since $|(y\coth y)'|\leq|y|$ for all real $y$ (Lemma \ref{lem:cothder}),
we know that, for $\Delta\geq 0$, $(y+\Delta) \coth (y+\Delta) = y \coth y + O^*(\Delta (y+\Delta))$.
If $\sigma<1$, we let $\Delta = \pi (1-\sigma) \left(\frac{1}{T'}-\frac{1}{T}\right) \leq   \frac{\frac{\pi}{2} (1-\sigma)}{T (T-\frac{1}{2})}$, and so, since $y\coth y$ is increasing for $y>0$,
\[
0 \leq \frac{\pi}{T'} \coth \frac{\pi (1-\sigma)}{T'} - \frac{\pi}{T} \coth \frac{\pi (1-\sigma)}{T} \leq
\frac{1}{1-\sigma}\cdot \Delta \frac{\pi (1-\sigma)}{T'}\leq \frac{\frac{\pi^2}{2} (1-\sigma)}{T (T-\frac{1}{2})^2}.
\]
If $\sigma>1$, we let $\Delta = \pi (\sigma-1) \left(\frac{1}{T'}-\frac{1}{T}\right) = 
 \frac{\frac{\pi}{2} (\sigma-1)}{T (T-\frac{1}{2})}$, and proceed likewise, with $y = \pi (\sigma-1)/T$:
 \[
0 \leq \frac{\pi}{T'} \coth \frac{\pi (\sigma-1)}{T'} - \frac{\pi}{T} \coth \frac{\pi (\sigma-1)}{T} \leq \frac{\frac{\pi^2}{2} (\sigma-1)}{T (T-\frac{1}{2})^2}.
\]
In any event, by $|\sigma-1|\leq 100$, the difference is bounded by $\frac{0.001}{T^2}$, say. Since $1.2+0.001<\frac{\pi}{2}$ and  $\frac{{\pi}/{2}}{T^2} + \frac{\pi}{T-\frac{1}{2}}< \frac{\pi}{T-1}$, we are done.

\end{proof}

\begin{proof}[Proof of Corollary~\ref{cor:psi}]
This is just cases $\sigma=0$ and $\sigma=1$ of Proposition~\ref{prop:sagaro}.
\end{proof}

\begin{lemma}\label{lem:pernic}
For $1\leq x\leq 10^{13}$, 
\[-\sqrt{2} <\frac{\psi(x)-x}{\sqrt{x}}\leq 0.79059275\dotsc,\]
with the extrema being reached at $x=2^-$ and $x=110102617$, respectively. On the range
$10^4\leq x\leq 10^{13}$, the minimum is 
$-0.7509024438\dotsc$, which is reached at 
$36917099^-$.

For $1\leq x\leq 10^{12}+3$,
\begin{equation}\label{eq:malina}
-0.7585825520\dotsc \leq \sum_{n\leq x} \frac{\Lambda(n)}{n} - (\log x - \gamma) \leq 0.787\dotsc,
\end{equation}
with the extrema being reached at $1423^-$ and $110102617$, respectively.
\end{lemma}
We could use \cite{zbMATH06864192}, but choose to keep matters self-contained.
\begin{proof}
These are medium-small brute-force computations, of the kind carried out over a weekend on a laptop.
We used {\em primesieve} for sieving, and CRLibm and Dave Platt's header file \texttt{int\_double14.2.h} for interval arithmetic. 

Some care is needed for \eqref{eq:malina} -- we must
avoid a catastrophic loss of accuracy due to cancellation.
For $x\leq 10^{10}$, we compute $\left(\sum_{n\leq x} \Lambda(n)/n - (\log x - \gamma)\right)\cdot \sqrt{x}$ directly. For larger $x$, we keep track of $2\gamma + \sum_{n\leq x} \frac{\Lambda(n)-1}{n}$ times $\sqrt{x}$ instead:
since $\sum_{n\leq x} 1/n = \log x + \gamma + O^*(1/x)$,  
\[\sum_{n\leq x} \frac{\Lambda(n)}{n} - (\log x - \gamma) = 2\gamma + \sum_{n\leq x} \frac{\Lambda(n)-1}{n}  + O^*\left(\frac{1}{x}\right),
\] In any event, the extrema in the range $[1,10^{12}+3]$ are reached within $[1,10^{10}]$: the minimum in $[10^{10},3\cdot 10^{12}+3]$ (at $x= 110486344211^-$) is higher, and the maximum in that range (at $x = 330957852107$) is lower.
\end{proof}
\begin{proof}[Proof of Corollary~\ref{cor:psiexpl}]
We can assume $x\geq 3\cdot 10^{12}+3$, as otherwise the stronger bounds in Lemma~\ref{lem:pernic} hold.

To bound $\psi(x) - x$, 
we can apply Corollary~\ref{cor:psi} with $T = 3\cdot 10^{12} + 1 + \frac{\pi}{3}$, since RH holds up to that height
\cite{zbMATH07381909}. We obtain
   \begin{equation}\label{eq:zartor}\left|\psi(x)-x\right|\leq \left|1 - \frac{\pi}{T} \coth \frac{\pi}{T}\right| \cdot x + \frac{\pi x}{T-1} + \left(\frac{1}{2\pi} \log^2 \frac{T}{2\pi} - \frac{1}{6\pi} \log \frac{T}{2\pi}\right) \sqrt{x}.\end{equation}
   Since $\frac{1}{y} \leq \coth y \leq \frac{1}{y} + \frac{y}{3}$ for $y>0$, we have
   $1\leq \frac{\pi}{T} \coth\frac{\pi}{T} \leq 1 + \frac{\pi^2}{3 T^2}$. Clearly $\frac{\pi^2}{3 T^2} + \frac{\pi}{T-1} \leq \frac{\pi}{T-1 - \frac{\pi}{3}}$. 

   To bound $\sum_{n\leq x} \Lambda(n)/n - (\log x - \gamma)$, simply apply  Corollary~\ref{cor:psi} with $T = 3\cdot 10^{12} + 1$.


\end{proof}

Let us finish by discussing what is the best value to take for $T$ moderate. It is easy to show that
\[\frac{\pi}{T} + \left(\frac{1}{2\pi} \log^2 \frac{T}{2\pi} - \frac{1}{6\pi} \log \frac{T}{2\pi}\right) \frac{1}{\sqrt{x}}\]
reaches its minimum when $T = 2 \pi e^{1/6} e^{W_0\left(\frac{\pi \sqrt{x}}{2 e^{1/6}}\right)}$, where $W_0$ is the principal branch of the Lambert function. To simplify, we may work with an approximation
$\frac{1}{\log \sqrt{x}} \frac{\pi \sqrt{x}}{2 e^{1/6}}$ to $e^{W_0\left(\frac{\pi \sqrt{x}}{2 e^{1/6}}\right)}$.

\begin{corollary}
Let $x>1$, $T= 2\pi^2 \frac{\sqrt{x}}{\log x}\geq 10^7$.
Assume the Riemann hypothesis holds up to height $T$. Then
\begin{equation}\label{eq:schoen}|\psi(x)-x|\leq \frac{\sqrt{x}}{8\pi} \log^2 x.\end{equation}
In fact,
\begin{equation}\label{eq:leenosal}|\psi(x)-x|\leq \frac{\sqrt{x}}{8\pi} \left(\left(\log x - 2 \log \frac{\log x}{\pi e}\right)^2 - 4\right).\end{equation}
\end{corollary}
Here \eqref{eq:schoen} is a bound proved in \cite{zbMATH03510407} assuming full RH, and in \cite{zbMATH06596195} assuming RH up to height 
$4.92 \frac{\sqrt{x}}{\sqrt{\log x}}$, which is of course a stronger assumption for $x\geq 10^7$. The bound \eqref{eq:leenosal} is somewhat 
stronger than the bound proved in \cite[Thm.~1.1]{leenosal2023} under full RH.

We have $T<10^7$ if and only if $x< 2.8427\dotsc \cdot 10^{14}$, which is still within brute-force territory, and of course
   well within what is covered by \cite{zbMATH06864192}.
\begin{proof}
Assume first that $T\geq 10^7$. Then we can apply Corollary~\ref{cor:psi}. We start as in \eqref{eq:zartor}:
   \[\begin{aligned}\left|\psi(x)-x\right|&\leq \frac{\pi x}{T}+ \left(\log \frac{T^2}{(2\pi)^2} - \frac{2}{3}\right) \log \frac{T^2}{(2\pi)^2}\cdot \frac{\sqrt{x}}{8\pi}+ \frac{\pi x}{T (T-1)} + \frac{\pi^2 x}{3 T^2} ,\\
   &\leq 4\log x\cdot  \frac{\sqrt{x}}{8\pi} + \left(\log x - 2\log \frac{\log x}{\pi} - \frac{2}{3}\right)
   \left(\log x - 2\log \frac{\log x}{\pi}\right)\cdot \frac{\sqrt{x}}{8\pi}+ c \log^2 x\end{aligned}\]
   for $c =  \frac{\frac{\pi}{3} + \frac{1}{1-10^{-7}}}{4 \pi^3}$. We can simplify:  we write $y = \log x$, $w = \log(y/\pi)$, and note that
\[4 y + \left(y - 2 w -\frac{2}{3}\right) (y-2 w) = (y - 2 w + 2)^2 -  \frac{2}{3} (y-14 w) - 4.
\]
Now, $y- 14 \log(y/\pi)$ is increasing for $y>14$, and so, for $y>33$, it is greater than its value at $y=33$, namely,
$0.075\dotsc$. We can assume $x>e^{33}= 2.14\dotsc \cdot 10^{14}$ and so $y>33$, as otherwise $T\geq 10^7$ does not hold.  By
$x>10^{12}$, we have $\frac{2}{3} \cdot 0.075 \frac{\sqrt{x}}{8\pi} > c \log^2 x$, since that holds for $x=10^{12}$. Hence \eqref{eq:leenosal} follows, and so does the weaker statement
\eqref{eq:schoen}.
\end{proof}

  \section{Final remarks}\label{sec:finremark}

The estimates and methods here affect all sorts of bounds in explicit analytic number theory; for
instance, most of the first chapter of \cite{TMEEMT} will have to be updated, in view
of both the results here and those in the companion paper \cite{Moebart}. Since what we do is provably optimal in a rather precise sense (\S \ref{sec:asideopt}), there is little point in waiting for further improvements (though see \S \ref{subs:companal}). 

What should make sense is to base the explicit theory on (a) Theorem~\ref{thm:mainthmB} and Cor.~\ref{cor:psi}, with the parameter $T$ being updated as the Riemann hypothesis is checked to further heights, (b) estimates relying on zero-free regions such as the bound in \cite{zbMATH07723301} (see \eqref{eq:dodo1}). Here the dependency on (b) should be encoded in as flexible a way as possible, since those results, while proved with great care, have not been proved to
use information on $\zeta(s)$ optimally.

We are at a point where the formalization of explicit results is becoming a realistic possibility and a priority. We have kept this paper relatively self-contained in part for the sake of reliability even in the absence of formalization, but also with a future formalization in mind.

  \subsection{Prior work on $\psi(x)$}\label{subs:compmeth}
There is a series of bounds of the form $|\psi(x)-x|\leq \epsilon x$ in the literature:
see Tables \ref{tab:psi-bounds} and \ref{tab:psi-bounds2}. All work there, except for Büthe's \cite{zbMATH06596195}
and ours, uses not only verifications of RH up to a given height $T$ but also 
zero-free regions of the form $\sigma > 1 - c/\log t$. Both \cite{zbMATH06417200} and its update \cite{Dusart2018} used zero-density results as well.

The reason why several bounds are significantly better in Table \ref{tab:psi-bounds2} than in Table \ref{tab:psi-bounds} is second-order terms. In fact, for the very high $T$ in the second half of Table \ref{tab:psi-bounds}, 
the terms coming from zeros on $\Re s = 1/2$ can overwhelm the other terms: for instance, for us, for
$x = e^{60}$, the terms coming from such zeros contribute $1.06367\dotsc \cdot 10^{-11}\cdot x$, whereas the rest
-- what we think of as the leading term -- contributes only $1.04719\dotsc \cdot 10^{-12}\cdot x$.


Thus, it should not be a surprise that the improvement our results represent is much more marked in Table \ref{tab:psi-bounds2} than in Table \ref{tab:psi-bounds}: for us, the contribution of zeros on 
$\Re s = \frac{1}{2}$ is somewhat smaller than for other authors (see the negative term in Cor.~\ref{cor:psi}, which is new), but the only way
to reduce that contribution more substantially would be cancellation (\S \ref{subs:companal}).

\begin{table}[htbp]          
  \centering                 
  \small                     
  \caption{For all $x\geq e^{60}$, $|\psi(x)-x|\leq \epsilon x$, assuming
RH holds up to $T$}  
  \label{tab:psi-bounds}     
  \begin{tabular}{@{}lrrr@{}}
    \toprule
    $\epsilon$ & reference & $T$ & For that $T$, our $\epsilon$ would be\dots\\
    \midrule
     $0.0101$    & \cite{zbMATH03039120} & $1{\,}468$ & $0.0021431$\\
     $1.3740\cdot 10^{-3}$     & \cite{RosserSchoenfeld}  & $21{\,}808$ & $1.44071\cdot 10^{-4}$\\
     $1.7583\cdot 10^{-5}$   & \cite{RosserSchoenfeldSharper} & $1{\,}894{\,}439$ & $1.65833\cdot 10^{-6}$\\
     $9.04993\cdot 10^{-8}$      & \cite{Dusart1998}  & $545{\,}439{\,}824$ & $5.76463\cdot 10^{-9}$\\ 
     $3.1732\cdot 10^{-11}$   &\cite{zbMATH06417200}, \cite{MR3766393} & $2{\,}445{\,}999{\,}556{\,}030$ & $1.17592 \cdot 10^{-11}$\\
     $2.978\cdot 10^{-11}$    & \cite{Dusart2018}  & $2{\,}445{\,}999{\,}556{\,}030$ & $1.17592 \cdot 10^{-11}$\\
     $1.23991\cdot 10^{-11}$   & \cite{zbMATH06596195}, \cite{Bhattacharjee}  & $2{\,}445{\,}000{\,}000{\,}000$ & $1.17592 \cdot 10^{-11}$\\
     $1.16840\cdot 10^{-11}$ & our bound & $3{\,}000{\,}000{\,}000{\,}003$ & $1.16840\cdot 10^{-11}$\\
    \bottomrule
  \end{tabular}
\end{table}

\begin{table}[htbp]          
  \centering                 
  \small                     
  \caption{For all $x\geq e^{100}$, $|\psi(x)-x|\leq \epsilon x$, assuming
RH holds up to $T$}  
  \label{tab:psi-bounds2}     
  \begin{tabular}{@{}lrrr@{}}
    \toprule
    $\epsilon$ & reference & $T$ & For that $T$, our $\epsilon$ would be\dots\\
    \midrule
     $0.00932$    & \cite{zbMATH03039120} & $1{\,}468$ & $0.00214304$\\
     $9.9653\cdot 10^{-4}$     & \cite{RosserSchoenfeld}  & $21{\,}808$ & $1.44071\cdot 10^{-4}$\\
     $1.6993\cdot 10^{-5}$   & \cite{RosserSchoenfeldSharper} & $1{\,}894{\,}439$ & $1.65833\cdot 10^{-6}$\\
     $8.84263\cdot 10^{-8}$      & \cite{Dusart1998}   & $545{\,}439{\,}824$ & $5.75975\cdot 10^{-9}$\\ 
     $2.4178\cdot 10^{-11}$   &\cite{zbMATH06417200}, \cite{MR3766393} & $2{\,}445{\,}999{\,}556{\,}030$ & $1.28438 \cdot 10^{-12}$\\
     $1.815\cdot 10^{-11}$    & \cite{Dusart2018}  & $2{\,}445{\,}999{\,}556{\,}030$ & $1.28438 \cdot 10^{-12}$\\
     $2.63677\cdot 10^{-12}$   & \cite{zbMATH06596195}, \cite{Bhattacharjee}  & $2{\,}445{\,}000{\,}000{\,}000$ & $1.28438 \cdot 10^{-12}$\\
$1.04720\cdot 10^{-12}$ & our bound & $3{\,}000{\,}000{\,}000{\,}003$ & $1.04720\cdot 10^{-12}$\\
    \bottomrule
  \end{tabular}
\end{table}

Here \cite{Dusart1998} simply implements the method in \cite{RosserSchoenfeldSharper} (non-rigorously: an incomplete modified Bessel function is dealt with by non-verified numerical integration) and runs it with a higher
value of $T$, viz., $T = 545{\,}439{\,}824$. Incidentally, most bounds in both tables are given by formulas far
more complicated than those in Corollary~\ref{cor:psi}, often involving half-page expressions and/or special functions for which a rigorous implementation is not trivial to find.

All work before \cite{zbMATH06417200} followed \cite{zbMATH03039120} in using repeated integration, which amounts to multiplying $\Lambda(n)$ by a fixed polynomial weight $P(n)|_{I}$, where $I$ is a compact interval. In contrast, \cite{zbMATH06417200} tried
to find an optimal weight, but the way that the optimization problem was set up hampered the approach. The Mellin transform was bounded in terms of a derivative $g^{(k)}$ at the very beginning (and so the task became to minimize that {\em bound}), not to mention that an application of Cauchy-Schwarz biased the optimization \cite[\S 3.1]{zbMATH06417200}; as a result, the ``optimal'' weight
found was again a truncated polynomial $P(n)|_I$, though a better one than before.

Büthe (\cite{zbMATH06596195}, corrected in \cite{Bhattacharjee}) defined his weight on $n$
to be of the form $(1_I\ast\,\widehat{\ell_{c,\epsilon}})(\log n)$, where $\ell_{c,\epsilon}$ is the
Logan function (\cite{logan1971optimal}, \cite{zbMATH04032305}). The Logan function is a clever choice, but it is, again, the
solution to an optimization problem with the wrong constraint, viz., compact support 
{\em in physical space} (that is, on $n$); that may have made sense in \cite{zbMATH04008495}, which 
was computational, but, as our work serves to show, what should have been sought in the problem at hand is compact support in Fourier space. Now, in Fourier space, 
Büthe's tails are light enough (thanks to the Logan function) that he does not need
a zero-free region, but the contribution of zeros with $|\Im s|>T$ is still there, and so his 
approach works only for $\psi(x)$, not for $M(x)$.

The best bound known to date based on the classical zero-free region is the one in \cite{zbMATH07723301}, viz.,
\begin{equation}\label{eq:dodo1} 9.22106 x (\log x)^{\frac{3}{2}} \exp\left(-0.8476836 \sqrt{\log x}\right).\end{equation}
It is better than our bound \begin{equation}\label{eq:dodo2}\frac{\pi}{3\cdot 10^{12}} x + 113.67 \sqrt{x}\end{equation} only for $x\geq e^{2394.19\dotsc}$.

We can obviously put \eqref{eq:dodo1}, \eqref{eq:dodo2} and a brute-force verification for $x\leq 10^{13}$ together, and obtain
\begin{equation}\label{eq:dada}|\psi(x)-x|\leq \frac{34.409 x}{\log^4 x}\quad \text{for all $x\geq 69991$},\end{equation}
for instance, but that is just wrapping paper. It is a rather different situation from that of $M(x)$, where we do not
have a bound like \eqref{eq:dodo1}, and bounds of type $|M(x)|\leq x/\log^k x$ are often crucial -- indeed the only way known to obtain bounds of the form $|M(x)| = o(x)$ is to start from bounds $|M(x)|\leq \epsilon x$ and bounds on $\psi(x)$ to obtain bounds of the form $|M(x)|\leq x/(\log x)^\alpha$.


{\em Note.} The reader may object that we have assumed $T\geq 10^7$ in Cor.~\ref{cor:psi} and Cor.~\ref{cor:psiexpl}; why are the first comparisons in Tables \ref{tab:psi-bounds} and \ref{tab:psi-bounds2} valid? That was just a simplifying assumption, absent from Thm.~\ref{thm:mainthmB}; for low $T$, we can just literally compute the sum bounded in Prop.~\ref{prop:vihuela}, instead of 
using Prop.~\ref{prop:vihuela} to bound it. We have actually computed these sums, and found
the bound in Prop.~\ref{prop:vihuela} still holds, with some room to spare; e.g., the sum for $T=1468$
is $4.281\dotsc$ rather than $4.441\dotsc$, while the sum for $T = 1894439$ is $24.318\dotsc$ rather than $24.657\dotsc$.
Bounding the integrand in Cor.~\ref{cor:adiaro} computationally is also a light task, compared to the task of finding zeros up to $T$. 

\subsection{Generalizations}
Theorem \ref{thm:mainthmB} is already quite general; it should be applicable, say, to functions
in Selberg $S$-class, or even more broadly, since we do not assume the existence of a functional equation, though we find it helpful for $A(s) = \zeta'(s)/\zeta(s)$ (\S \ref{subs:tripa}). 

\subsubsection{Dirichlet $L$-functions}
Let $\chi$ be a Dirichlet character.
We can apply Thm.~\ref{thm:mainthmB} to 
$a_n = (1 + \Re(\chi(n))) \Lambda(n) = (1 + (\chi(n) + \overline{\chi}(n))/2) \Lambda(n)$ and $a_n = (1 + \Im(\chi(n))) \Lambda(n) = (1 + (\chi(n) - \overline{\chi}(n))/2i) \Lambda(n)$ (since Thm.~\ref{thm:mainthmB} works
for non-negative real coefficients). Then Platt's
verification of $GRH$ for $\chi \bmod q$ for $q\leq 400\,000$ up to height $H_q=10^8/q$ \cite{zbMATH06605790} will just need to be supplemented with a smaller
computation, viz.,  on $\max_{\sigma\in [-1/64,1]} 1/\zeta(\sigma + i H_q)$
and $\max_{\sigma\in [-1/64,1]} \zeta'(\sigma + i H_q)/\zeta(\sigma + i H_q)$ for each $\chi$.

If one aims at sums of $\Lambda(n)$ on arithmetic progressions $a + q \mathbb{Z}$, one should define
$A(s)$ to be $- \frac{1}{\phi(q)} \sum_\chi \overline{\chi(a)}\frac{L'}{L}(s,\chi)$, and then apply Thm.~\ref{thm:mainthmB} to 
$A(s)$, rather than obtain bounds for $A(s) = L(s,\chi)$ and combine them. The residue of 
$A(s) = -\frac{1}{\phi(q)} \sum_\chi \overline{\chi(a)} \,\frac{L'}{L}(s,\chi)$ at $s=1$ is $1/\phi(q)$, and so 
the main error term from Thm.~\ref{thm:mainthmB} will be proportional to 
$O^*\left(\frac{\pi}{\phi(q) H_q}\right)$; thus, we
will have bounds  far stronger than those in \cite{zbMATH07036793} (for $x\leq \exp(10\,000)$, say)
or those in the older source \cite[Table 1]{Ramrum}.

\subsubsection{Poles of higher order at $s=1$.} One sense in which Theorem \ref{thm:mainthmB} is not
stated as generally as it could be is that it has the condition that the pole at $s=1$ be simple.
Dropping that condition seems to entail no great complications; we have kept the condition for
the sake of simplicity. A relevant test case is that of the sum $\sum_{n\leq x} \Lambda(n) \log(x/n)$.
Of course one can estimate it by integrating $\psi(y)$, but one will do much better by 
proving Thm.~\ref{thm:mainthmB} in the case of a multiple pole at $s=1$. The approximation result to use then is
\cite[Thm.~2.6]{zbMATH05639041}.

  

  \subsection{Computational-analytic bounds}\label{subs:companal}


The constant $C = \frac{1}{2\pi} \log^2 \frac{T}{2\pi} - \frac{1}{6\pi} \log \frac{T}{2\pi}$ in Cor.~\ref{cor:psi} is both a little
bothersome and really there: for some very rare, extremely large $x$ 
the arguments of zeros $x^{\frac{1}{2}+i\gamma}$ for $|\gamma|\leq T$ will line up, and give
us a sum of size $C \sqrt{x}$. We do not, however, expect this
to happen for moderate $x$ (say, $x\leq 10^{30}$), beyond which point the leading term dominates.

How do we find cancellation in the sums over non-trivial zeros in
Thm.~\ref{thm:mainthmB} in a range $x_0\leq x\leq x_1$, then? Here $x_0$ would be the end of the brute-force range (currently $x_0 = 10^{16}$).

The basic strategy is known (\cite[\S 4.4]{odlyzko1020}; see also the implementation in \cite{zbMATH06864192}): we can see the finite sum $\sum_\gamma x^{i \gamma}$ as the Fourier
transform of a linear combination of point measures
$\delta_{\frac{\gamma}{2\pi}}$, evaluated at $\log x$.
To bound that transform throughout the range
$[\log x_0,\log x_1]$, it is enough, thanks to a
Fourier interpolation formula (Shannon-Whittaker\footnote{\cite{odlyzko1020} recommends \cite{zbMATH03895606} for a historical overview.}), to evaluate
it at equally spaced points. Actually, we first split the range of $\gamma$ into segments of
length $L$; then we need evaluate the transform only at integer multiples of $2\pi/L$. That one does by
applying a Fast Fourier Transform.

We propose what seems to be an innovation: do not split
the range brutally into segments; rather,
express the constant function as a sum of triangular
functions $n\to \tri(\frac{t}{L}+n)$, where $\tri =  1_{[-\frac{1}{2},\frac{1}{2}]}\ast 1_{[-\frac{1}{2},\frac{1}{2}]}$. Since 
$\widehat{\tri}(x) = \frac{\sin^2 \pi x}{(\pi x)^2}$,
it is not hard to obtain, in effect, an interpolation
formula with weights $\widehat{\tri}(\frac{x-n}{k})$, which
are non-negative and have fast decay.\footnote{D. Radchenko suggests partitioning the constant function using $1_{[-\frac{1}{2},\frac{1}{2}]}^{\ast 2 m}$ instead, as then, for $m>1$, decay is even faster than
for $m=1$.}
Work becomes cleaner and faster as a result.
(It may be best to consider low-lying zeros separately.)

Since there are $O(L \log T)$ zeros in a segment
of length $L$ at height up to $T$, precomputing
a fast Fourier transform should take time
$O(L \log T \log(L \log T))$ per segment; then
we will evaluate it at $O((\log x_1-\log x_0)/L)$ points, at constant cost each.
Total computation time should then be $O(T)\cdot\max(\log T \log L, \log(x_1/x_0))$ or thereabouts.

  \appendix

\section{Explicit estimates on $\zeta(s)$}\label{sec:estzeta}

Here we prove some basic quantitative results on the Riemann zeta function. 
Let us first derive clean bounds for the $\Gamma$ function and for the digamma function\footnote{We hope the digamma function will stop being denoted by
$\psi$, particularly in papers in number theory. In these modern times, it makes far more sense to use the archaic Greek letter digamma, which is available in \LaTeX: $\digamma$.} $\digamma(z) = \frac{\Gamma'(z)}{\Gamma(z)}$.
\begin{lemma}\label{lem:digam}
For $z\in \mathbb{C}$ with $\Re z\geq \frac{1}{2}$,
\[\Re \digamma(z) \leq \log |z|.\]
\end{lemma}
The same inequality (for $\Re z\geq \frac{1}{4}$) is also in \cite[\S 5]{zbMATH05638163}.
\begin{proof}
 By, e.g., \cite[(5.9.13)]{zbMATH05765058}, $\digamma(z) = \log z - \frac{1}{2 z} - \int_0^\infty \left(\frac{1}{2} - \frac{1}{t} + \frac{1}{e^{t}-1}\right) e^{-t z} dt$, and so, 
  since $\Re\left(\frac{1}{2z}\right) = \frac{x}{2|z|^2}$, all we have to show is that $\Re I \leq \frac{x}{2|z|^2}$ for
$I = - \int_0^\infty g(t)\,e^{-t z} dt$ and $g(t) = \frac{1}{2} - \frac{1}{t} + \frac{1}{e^{t}-1}$.
Write $z = x+ i y$.
By integration by parts applied twice,
$$I = -\frac{1}{z} \int_0^\infty g'(t) e^{- t z} dt = - \frac{1}{12 z^2} - \frac{1}{z^2} \int_0^\infty g''(t) e^{-t z} dt =
- \frac{1}{12 z^2} - \frac{1}{z^2} \int_0^\infty g''(t) e^{-t x} e^{-i t y} dt 
.$$
It is enough to prove that $|g''(t)|\leq \frac{1}{12}$ for all $t\geq 0$, since then $\Re I \leq \frac{1}{12 |z|^2} + \frac{1}{12|z|^2 x} \leq \frac{1}{4 |z|^2}\leq \frac{x}{2 |z|^2}$ by the assumption $x\geq 1/2$. Since $g(t) = \sum_{n} \frac{B_{2 n} t^{2 n - 1}}{(2 n)!}$,
where $\sum_n$ means $\sum_{n\geq 1}$ as always,
$$g''(t) = \sum_{n} \frac{B_{2 n+ 2} t^{2 n - 1}}{(2 n-1)! (2 n+ 2)} = \sum_{n} \frac{(-1)^n 2 \cdot 2 n (2 n+1)}{(2\pi)^{2 n+2}}\zeta(2 n + 2) t^{2 n - 1},$$ with leading coefficient $\frac{B_4}{4} = -\frac{1}{120}$. We see that this is an alternating sequence with decreasing terms for $t\leq \pi$ (since then $\frac{(2 n+2) (2 n+3)}{2 n (2 n+1)} \cdot \frac{t^2}{(2 \pi)^2} \leq \frac{10/3}{4}<1$) and so, for $t\leq \pi$, $|g''(t)|\leq \frac{t}{120} \leq \frac{\pi}{120}$. We note that $g''(t) = - \frac{2}{t^3} + \frac{1}{e^t-1} + \frac{3}{(e^t-1)^2} + \frac{2}{(e^t-1)^3}$, and so, for $t\geq \pi$,
$$|g''(t)|\leq \max\left(\frac{2}{\pi^3}, \frac{1}{e^\pi-1} + \frac{3}{(e^\pi-1)^2} + \frac{2}{(e^\pi-1)^3}\right) = \frac{2}{\pi^3}.$$ Since $\frac{\pi}{120}$ and $\frac{2}{\pi^3}$ are both smaller than $\frac{1}{12}$, we are done.
\end{proof}

\begin{lemma}\label{lem:distur} Let $z=x+i y$, $x\geq \frac{1}{2}$, $y\in \mathbb{R}$. Then
$$|\Gamma(z)|\leq \sqrt{\frac{\pi}{\cosh \pi y}} \cdot |z|^{x-\frac{1}{2}} \leq 
\sqrt{2\pi} |z|^{x-\frac{1}{2}} e^{-\frac{\pi}{2} |y|}.$$
\end{lemma}
The versions we have found in the literature (e.g., \cite[5.6.9]{zbMATH05765058}) have a 
fudge factor that is unnecessary for $x\geq \frac{1}{2}$.
\begin{proof}
For $x=\frac{1}{2}$, the inequality is true, as $|\Gamma(\frac{1}{2}+i y)|= \sqrt{\frac{\pi}{\cosh \pi y}}$ (as in, e.g., \cite[(5.4.4)]{zbMATH05765058}). For $x>\frac{1}{2}$, by Lemma \ref{lem:digam},
\[\log \frac{|\Gamma(z)|}{|\Gamma(\frac{1}{2}+i y)|} = \int_{\frac{1}{2}}^x \Re \digamma(u+ i y) du \leq \int_{\frac{1}{2}}^x \log |u+ i y| du \leq
\left(x-\frac{1}{2}\right)\cdot \log |z|.\]
\end{proof}

Let us make explicit two well-known applications of the functional equation.
\begin{lemma} \label{lem:zetbound}
For $s\in \mathbb{C}$ with $\Re s\leq \frac{1}{2}$,     
     	\begin{equation}\label{eq:fromfunco} 
     |\zeta(s)|\leq   
     \sqrt{\frac{1}{1 - \frac{1}{2}\operatorname{sech}^2 \frac{\pi \Im s}{2}}} \cdot \left(\frac{|1-s|}{2\pi}\right)^{\frac{1}{2}-\Re s} |\zeta(1-s)|.\end{equation}
\end{lemma}
Note that $\sech^2 y\leq 1$ for all $y\in \mathbb{R}$, with $\sech^2 y \to 0$ rapidly as $|y|\to \infty$.
\begin{proof}
    Recall the functional equation
	\begin{equation}\label{eq:funceq}
    \zeta(s) = (2\pi)^{s-1}\cdot 2 \sin \left(\frac{\pi s}{2}\right) \Gamma(1-s) \zeta(1-s).\end{equation}
    By Lemma \ref{lem:distur} applied to $z=1-s$, together with $\cosh(-\pi y) = \cosh(\pi y)$,
$$|\Gamma(1-s)|\leq \sqrt{\frac{2\pi}{2\cosh (\pi \cdot\Im s)}} \cdot |1-s|^{\frac{1}{2} - \Re s} ,$$
and so
\[ |\zeta(s)|\leq \left(\frac{|1-s|}{2\pi}\right)^{\frac{1}{2}-\Re s} \cdot \frac{2 \left|\sin \frac{\pi s}{2}\right|}{\sqrt{2\cosh(\pi\cdot \Im s)}} |\zeta(1-s)|.\]
Since $\left|\sin \frac{\pi s}{2}\right|\leq \cosh \frac{\pi\cdot \Im s}{2}$,
and $\cosh  y = 2\cosh^2(y/2)-1$ for all $y$,
\eqref{eq:fromfunco} holds.
\end{proof}

\begin{lemma} \label{lem:derivbound}
	For $s\in \mathbb{C}$,
	\begin{equation}\label{eq:fromfunceqB} 
		\frac{\zeta'}{\zeta}(s) = -  \frac{\zeta'}{\zeta}(1-s) - \digamma(1-s) + 
        \frac{\pi}{2} \cot \frac{\pi s}{2} + \log 2\pi,
	\end{equation}
    where $\digamma(s) = \Gamma'(s)/\Gamma(s)$.  If $\Re s \leq 0$ and $|\Im s|\geq 1$, then
\begin{equation}\label{eq:fromfunceqB2}
		\frac{\zeta'}{\zeta}(s) = - \frac{\zeta'}{\zeta}(1-s) - \log(1-s) + O^*(4)
.\end{equation}
\end{lemma}
One can obviously give a more precise constant than $4$, but there is no
reason to bother.
\begin{proof}
Equation \eqref{eq:fromfunceqB} follows immediately from the functional equation \eqref{eq:funceq}
once we take logarithmic derivatives. To obtain \eqref{eq:fromfunceqB2}, we first note that, for
$s = \sigma + it$,
$$\left|\cot \frac{\pi s}{2}\right|  = 
\left|\frac{e^{\frac{\pi i s}{2}} +
    e^{\frac{-\pi i s}{2}}}{
   e^{\frac{\pi i s}{2}} -
    e^{\frac{-\pi i s}{2}}}\right|\leq
    \left|\frac{e^{\frac{\pi}{2} t} +
    e^{\frac{-\pi}{2} t}}{
   e^{\frac{\pi}{2} t} -
    e^{\frac{-\pi}{2} t}}\right| \leq 
\coth \frac{\pi}{2}
    = 1.09033\dotsc$$ if $|t|\geq 1$.
By \cite[(5.11.2) and \S 5.11(ii)]{zbMATH05765058} (a bound proved by a trick due to
Stieltjes \cite{zbMATH02691938}),
\begin{equation} \label{eq:digamma}
\digamma(s) = \log s - \frac{1}{2 s} + O^*\left(\frac{1}{3\sqrt{2} |s|^2}\right)
\end{equation}
when $\Re s\geq 0$. Now, the image of the region
$R = \{z: \Re z\geq 1, \Im z\geq 1\}$ under the map $z\to 1/z$ is the 
intersection (``lens'') of two disks with diameter $1$, one centered at $\frac{1}{2}$ and one centered at $-\frac{i}{2}$; hence, $\max_{z\in R} |\Re z| = \max_{z\in R} |\Im z|=\frac{1}{2}$.
Thus, for $\Re s\leq 0$ and $\Im s\geq 1$,
$$\left|\log 2\pi + \frac{1}{2 (1-s)}\right|\leq \left|\log 2\pi + \frac{1-i}{4}\right|\leq 2.10279\dotsc
$$
and so
$$\begin{aligned}\left|
\frac{\zeta'}{\zeta}(s) + \frac{\zeta'}{\zeta}(1-s) + \log(1-s)\right|&\leq \left|\frac{\pi}{2} \cot \frac{\pi s}{2}\right| + 
\left|\log 2\pi + \frac{1}{2 (1-s)}\right| + \frac{1}{3\sqrt{2} |1-s|^2}
\\&\leq \frac{\pi}{2}\cdot 1.091 +
 2.103 + \frac{1}{3\cdot 2^{3/2}} \leq 3.935 < 4.
\end{aligned}$$
\end{proof}

\begin{lemma}\label{lem:adioso} For $\sigma_0>0$, $T\geq 1$ and $x\geq e^3$,
\begin{align*}
\int_{-\infty}^{-\sigma_0}\left|\dfrac{\zeta'}{\zeta}(\sigma+iT)\right| (1-\sigma) x^{-(1-\sigma)} d\sigma 
&\leq \big( \left(\log T + c_{\sigma_0}\right) \left(1 + \frac{1}{L} +\sigma_0\right)  + \dfrac{3}{2T^2} (1+\sigma_0)^3\big) \frac{x^{-1-\sigma_0}}{L}\\
 &\leq \frac{4}{3} \left(\log T + c_{\sigma_0} + \frac{3}{2 T^2}\right) \frac{1}{x L}
\end{align*}
for $L = \log x$ and 
$c_{\sigma_0} = \left|\frac{\zeta'}{\zeta}(1+\sigma_0)\right|+4+\frac{\pi}{2}$.
\end{lemma}
\begin{proof}
By \eqref{eq:fromfunceqB2} we see that for $\sigma\leq -\sigma_0<0$ (note that $1-\sigma>1$),
\begin{equation*}
		\left|\frac{\zeta'}{\zeta}(\sigma+iT) \right|\leq \left|\frac{\zeta'}{\zeta}(1-\sigma-iT)\right|+   |\log(1-(\sigma+iT))|+4 .
\end{equation*}
Clearly $\Big|\frac{\zeta'}{\zeta}(1-\sigma-iT)\Big|\leq \Big|\frac{\zeta'}{\zeta}(1-\sigma)\Big|\leq \Big|\frac{\zeta'}{\zeta}(1+\sigma_0)\Big|$. Moreover, it is not hard to see that $|\log(1-(\sigma+iT))|\leq \log T + \frac{(1-\sigma)^2}{2T^2} + \frac{\pi}{2}$. We obtain our first bound from
\[
\int_{-\infty}^{-\sigma_0}(1-\sigma)x^{\sigma-1} d\sigma  = \left(\frac{1+\sigma_0}{\log x}+ \frac{1}{\log^2 x}\right)x^{-1-\sigma_0},\]
\[
\int_{-\infty}^{-\sigma_0}(1-\sigma)^3x^{\sigma-1} d\sigma  = \left(\frac{(1+\sigma_0)^3}{\log x}+ \frac{3(1+\sigma_0)^2}{\log^2 x} + \frac{6(1+\sigma_0)}{\log^3 x} + \frac{6}{\log^4 x}\right)x^{-1-\sigma_0}
\]
and $1 + \frac{3}{3}+\frac{6}{3^2}+\frac{6}{3^3} < 3$. Use $(t x^{-t})' = (1- t\log x) x^{-t}<0$ and
$(t^3 x^{-t})' = (3 t^2 - t^3 \log x) x^{-t} \leq 0$ for $t\geq 1$ to obtain the second bound.

\end{proof}

\begin{lemma}\label{lem:badabook}
Let $f(t) = \frac{\zeta'}{\zeta}(t) + \digamma(t) - \log 2\pi$. Then 
$t f(t)$ and $f(t)+\frac{1}{t}$ are increasing functions on $[2,\infty)$. Moreover, $f(t)+\frac{1}{t}$ is concave for $t>1$.
\end{lemma}
\begin{proof}
We know $\digamma(t)+\frac{1}{t}$ is increasing for
$t>0$ because $\digamma'(t)+ (\frac{1}{t})' = 
\sum_n \frac{1}{(n+t)^2} > 0$.
Since $-\zeta'(t)/\zeta(t) = \sum_n \Lambda(n) n^{-t}$ is decreasing for $t>1$,
$\zeta'(t)/\zeta(t)$ is increasing for $t>1$.
Hence, $f(t)+\frac{1}{t}$ is increasing on $(1,\infty)$, and thus so is $f(t)$.
We see that $\digamma''(t) + (\frac{1}{t})'' = -\sum_n \frac{2}{(n+t)^3}<0$, and so $\digamma(t)+\frac{1}{t}$ is concave.
Since
$n^{-t}$ is convex, $\zeta'(t)/\zeta(t)$ is also concave. Therefore, $f(t)+\frac{1}{t}$ is concave.

By 
$(t f)'(t) = f(t) + t f'(t)$, we see that, if $f(t_0)\geq 0$ for some
$t_0>1$, then $f(t)\geq f(t_0)\geq 0$ for all $t\geq t_0$, and hence 
$(t f(t))' \geq 0$ for all $t\geq t_0$. Here $t_0 = 7$ will do.
On the interval $[2,7]$, we prove $(t f)'(t)>0$ by the bisection method, implemented by means of interval arithmetic.
\end{proof}

The following Lemma ought to be standard, but seems hard to find in the literature.

\begin{lemma}\label{lem:kalmynin}
    The Laurent expansion of $-\zeta'(s)/\zeta(s)$ at $s=1$
    has alternating signs: $$-\frac{\zeta'}{\zeta}(s) = \frac{1}{s-1} + \sum_{n=0}^\infty (-1)^{n+1} a_n (s-1)^n,\quad\quad
    a_n>0.$$ 
\end{lemma}
The argument below was provided by A. Kalmynin \cite{501199}.
\begin{proof}
We know $\zeta(s)$ has no non-trivial zeros with $|\Im s|\leq 4$. Hence,
$$G(s) = -\frac{\zeta'}{\zeta}(s) - \frac{1}{s-1} + \frac{1}{s+2}$$
is regular on an open neighborhood of the region $V = [-3, 5] + i [-4,4]$. We check
that $|G(s)|\leq 1.396$ on $\partial V$ by the bisection method, run using
interval arithmetic (FLINT/Arb); hence, $|G(s)|\leq 1.396$ on $V$, by
the maximum modulus principle. Write $G(s) = \sum_{n=0}^\infty b_n (s-1)^n$. Then
$$b_n = \frac{1}{2\pi i} \int_{|s-1|=4} \frac{G(s)}{(s-1)^{n+1}} ds$$
and so $|b_n|\leq 4\cdot 1.396/ 4^{n+1} < 1.4/4^n$. The coefficient of $(s-1)^n$ in the Laurent expansion of $-\zeta'(s)/\zeta(s)$ is $b_n - (-1)^n 3^{-(n+1)}$. 
Since $1.4\cdot 3^6 < 4^5$,
we see that the second term dominates for $n\geq 5$.
For $n\leq 4$, we calculate the coefficients of $(s-1)^n$ by a simple
symbolic computation starting from the Laurent expansion of $\zeta(s)$, followed by a computation in FLINT/Arb.
\end{proof}

\section{Explicit estimates related to the zeros of $\zeta(s)$}\label{sec:zeroszeta}
\subsection{Basic estimates on zeros}
Write $N(T)$ for the number of zeros $\rho$ of $\zeta(s)$ with $0<\Im(\rho)\leq T$, counting multiplicity.
Define
\begin{align} \label{eq:remaind}
Q(t)=N(t)-\dfrac{t}{2\pi}\log \dfrac{t}{2\pi e} - \dfrac{7}{8}.
\end{align}


We recall that the first non-trivial zero $\rho$ of $\zeta(s)$ has
$\Im \rho = 14.13472514\dotsc$.
\begin{lemma}\label{lem:oldie}
For $0<t\leq 280$, $|Q(t)|<1$. 
For $t\geq 1$, $|Q(t)|\leq \frac{1}{5} \log t +2$. 
\end{lemma}
One can prove better bounds nowadays. We use an older result here to minimize dependencies.
\begin{proof}
The first bound is as in \cite[Thm.~17--18]{zbMATH03039120}.
By \cite[Thm.~19]{zbMATH03039120}, $|Q(t)|\leq 0.137 \log t + 0.443 \log\log t + 1.588$ for $t\geq 2$. 
We know that
$0.137\log t + 0.443 \log\log t + 1.588<\frac{1}{5}\log t + 2$ for $t\geq 5400$ because
$\left(\frac{1}{5} - 0.137\right) y > 0.443 \log y - 0.412$ for $y\geq \log 5400$.
For $\gamma_0\leq t\leq 5400$, where $\gamma_0 = 14.13472514\dotsc$ is the ordinate of the first non-trivial zero, we do a direct computational check: for each ordinate $\gamma$ of a
non-trivial zero up to $5400$,  we check that $Q(t)\leq \frac{1}{5} \log t + 2$ for $t = \gamma^-$ and
$t=\gamma$; this is enough because $\left(\frac{t}{2\pi} \log \frac{t}{2\pi e}\right)'>
\left(\frac{1}{5} \log t\right)'$ for $t\geq \gamma_0$. 
For $1\leq t<\gamma_0$, $N(t)=0$; since $\frac{1}{5} \log t - \frac{t}{2\pi} \log \frac{t}{2\pi e}$ is concave,
it is enough to check that it is positive for $t=1$, as we already know it is positive for
$t=\gamma_0$.
\end{proof}

\begin{corollary}\label{cor:brut}
For $T\geq 2\pi$, $N(T) \leq \frac{T}{2\pi} \log \frac{T}{2\pi}$.
\end{corollary}
\begin{proof}
By Lemma \ref{lem:oldie}, $Q(t)+\frac{7}{8}\leq \frac{t}{2\pi}$ for all $t\geq {14}$, and so the statement follows for $t\geq {14}$ by \eqref{eq:remaind}. We also know that
there are no zeros with $\gamma\leq 14$, and so $N(t)=0$ for $t\leq 14$. 
\end{proof}

\begin{corollary}\label{cor:zeroinbox}
Let $T\geq 1$, $0<a<T$. Then
\begin{equation}\label{eq:meredi}\begin{aligned}N({T}+a)-N(({T}-a)^-) & \leq  Q({T}+a)-Q(({T}-a)^-) + \frac{a}{\pi} \log \frac{{T}}{2\pi}\end{aligned}\end{equation}
Moreover, $Q({T}+a)-Q(({T}-a)^-)\leq \frac{2}{5}\log T + 4$.
\end{corollary}
\begin{proof}
By \eqref{eq:remaind},
\[\begin{aligned}&N(T+a)-N((T-a)^{{-}})= Q(T+a)-Q((T-a)^{{-}}) +
\int_{T-a}^{T+a} \frac{1}{2\pi} \log \frac{t}{2\pi} dt.
\end{aligned}\]
By the concavity of $\log$ we get \eqref{eq:meredi}. Moreover, by Lemma \ref{lem:oldie} and the concavity of $\log$ again, $Q(T+a)-Q((T-a)^{{-}})\leq \frac{2}{5}\log T + 4$. (If $T+a\leq 280$, use 
$|Q((T-a)^-)|, |Q(T+a)|\leq 1$; if $T-a<1$ and $T+a>280$, use $|Q((T-a)^-)|\leq 1$
and $\frac{1}{5} \log(T+a) + 3 < \frac{2}{5} \log T + 4$.)
\end{proof}

\begin{lemma}\label{lem:Lehmanmodern} 
Let $\phi:[t_0,t_1]\to \mathbb{C}$ be continuous and of bounded
variation, where $0<t_0\leq t_1$. Let $Q(t)$ be as in \eqref{eq:remaind}, and let 
$\gamma$ denote the imaginary parts of zeros of $\zeta(s)$. Then
\begin{equation}\label{eq:initor}\sum_{t_0<\gamma\leq t_1} \phi(\gamma) = \dfrac{1}{2\pi}\int_{t_0}^{t_1}\phi(t)\log \dfrac{t}{2\pi} dt + 
\phi(t) Q(t)\Big|_{t_0}^{t_1} - \int_{t_0}^{t_1} Q(t) d\phi(t).
\end{equation}
In particular, if $\phi(t)$ is real-valued, non-negative and decreasing, and $t_0\geq 14$,
\begin{equation}\label{eq:seguitor}\sum_{t_0<\gamma\leq t_1} \phi(\gamma) \leq \dfrac{1}{2\pi}\int_{t_0}^{t_1}\phi(t)\log \dfrac{t}{2\pi} dt + 
\phi(t_0) \left(\frac{1}{5} \log t_0 + 2 - Q(t_0)\right) + \frac{1}{5}\int_{t_0}^{t_1} \phi(t) \frac{dt}{t},
\end{equation}
and if $\phi(t)$ is real-valued, non-negative and increasing, and $t_0\geq 14$,
\begin{equation}\label{eq:sagasto}\begin{aligned}
    \sum_{t_0<\gamma\leq t_1} \phi(\gamma) \leq \dfrac{1}{2\pi}\int_{t_0}^{t_1}\phi(t)\log \dfrac{t}{2\pi} dt 
&+ \phi(t_1) \left(\frac{1}{5} \log t_1 + 2 + Q(t_1)\right)
- \frac{1}{5} \int_{t_0}^{t_1} \phi(t) \frac{dt}{t}
\\
&- \phi(t_0) \left(
 \frac{1}{5} \log t_0 + 2 +Q(t_0)\right).
\end{aligned}\end{equation}
\end{lemma}
This is basically \cite[Lemma 1]{zbMATH03242284}.
Besides the obvious fact that we are using better bounds on $Q(t)$ (available at the time of \cite{zbMATH03242284}), there are a few small natural improvements that may be useful in the future: we allow a complex-valued $\phi$ in \eqref{eq:initor}, and leave some of the terms $Q(t)$ as they are, instead of bounding them, so that we can obtain cancellation later.

\begin{proof}
By the definition \eqref{eq:remaind} of $Q(t)$ and
$\left(\frac{t}{2\pi}\log \frac{t}{2\pi e}\right)' = \frac{1}{2\pi} \log
\frac{t}{2\pi}$,
$$\sum_{t_0<\gamma\leq t_1} \phi(\gamma) =
\int_{t_0^+}^{t_1^+} \phi(t) dN(t) = \frac{1}{2\pi} \int_{t_0}^{t_1} \phi(t) \log \frac{t}{2\pi} dt
+ \int_{t_0^+}^{t_1^+} \phi(t) dQ(t).
$$
We integrate by parts, recalling the continuity of $\phi$ and the right-continuity of $Q$:
$$\int_{t_0^+}^{t_1^+} \phi(t) dQ(t)
= \phi(t) Q(t)\Big|_{t_0}^{t_1} - \int_{t_0}^{t_1} Q(t) d\phi(t). 
$$
Thus \eqref{eq:initor} holds. By Lemma \ref{lem:oldie} and integration by parts, \eqref{eq:seguitor} and \eqref{eq:sagasto} 
follow. In particular, for \eqref{eq:sagasto},
the last two terms of \eqref{eq:initor} contribute at most
\[\begin{aligned}\phi(t) Q(t)|_{t_0}^{t_1} + \int_{t_0}^{t_1} \left(\frac{1}{5} \log t + 2\right) d\phi(t) =
\phi(t) Q(t)\Big|_{t_0}^{t_1} + \left. \left(\frac{1}{5} \log t + 2\right) \phi(t) \right|_{t_0}^{t_1} - \frac{1}{5}
\int_{t_0}^{t_1} \phi(t) \frac{dt}{t}.\end{aligned}\]
\end{proof}

\begin{lemma}\label{lem:cathinv}
For $t_0\geq 14$,
\begin{equation}\label{eq:darmorio}\sum_{\gamma>t_0} \frac{1}{\gamma^2}\leq \frac{1}{2\pi t_0} \log \frac{e t_0}{2\pi} + 
\frac{1}{t_0^2} \left(\frac{2}{5} \log t_0 + \frac{41}{10}\right).
\end{equation}
\end{lemma}
\begin{proof}
By Lemma \ref{lem:Lehmanmodern} with $\phi(t) = 1/t^2$, for $t_1\geq t_0$, and Lemma \ref{lem:oldie},
\[\begin{aligned}\sum_{t_0<\gamma\leq t_1} \frac{1}{\gamma^2} &\leq
\dfrac{1}{2\pi}\int_{t_0}^{t_1}\dfrac{1}{t^2}{\log \frac{t}{2\pi}} dt + 2\cdot
\frac{\frac{1}{5} \log t_0 +2}{t_0^2} + \frac{1}{5} \int_{t_0}^{t_1} 
 \frac{dt}{t^3}\\
 &= \left.  -\frac{\log \frac{e t}{2\pi}}{2\pi t} \right|_{t_0}^{t_1} +  
 \frac{\frac{2}{5} \log t_0 +4}{t_0^2} + \frac{1}{10} \left(\frac{1}{t_0^2} - 
 \frac{1}{t_1^2}\right).\end{aligned}\]
Letting $t_1\to \infty$ we are done.
\end{proof}

\subsection{Expressing values of $\zeta(s)$ in terms of zeros of $\zeta(s)$}\label{subs:blomaps} 
Our aim here is to give an explicit version (Prop.~\ref{prop:titch96A}) of a well-known expression for $\zeta'(s)/\zeta(s)$ in terms of nearby zeros of $\zeta(s)$.

We start with an auxiliary estimate.
\begin{lemma}\label{lem:daremo}
Let $a>0$, $y_0\geq 2\pi e$, $t\geq a+y_0$. Then
\begin{equation}\label{eq:intogru}
\int_{-\infty}^{-y_0} \frac{\log \frac{-y}{2\pi}}{(y-t)^2} dy +
\left(\int_{y_0}^{t-a} + \int_{t+a}^\infty\right) \frac{\log \frac{y}{2\pi}}{(y-t)^2} dy \leq  \frac{2}{a} \log \frac{t}{2\pi} +
\frac{c_2}{t^2} + \frac{c_4}{t^4},
\end{equation}
where $c_2 = a - 2 y_0 \log \frac{y_0}{2\pi e}$,
$c_4 = (2\log 2 - 1) a^3 - \frac{4}{3} y_0^3$.
\end{lemma}
\begin{proof}
The three integrals on the left of \eqref{eq:intogru} equal
\[\frac{\log \frac{y_0}{2\pi}}{t + y_0} + \frac{1}{t} \log\left(1+\frac{t}{y_0}\right),\quad
  \frac{\log \frac{t-a}{2\pi}}{a} - \frac{\log \frac{y_0}{2\pi}}{t-y_0} + \frac{1}{t} \log \frac{y_0 a}{(t-y_0) (t-a)},\quad
 \frac{\log \frac{t+a}{2\pi}}{a} + \frac{1}{t} \log\left(1 +\frac{t}{a}\right),\]
 respectively. Thus, their total is
\[\frac{2}{a} \log \frac{t}{2\pi} + 
 \frac{1}{a} \log \left(1 - \frac{a^2}{t^2}\right) - \frac{2 y_0 \log \frac{y_0}{2 \pi}}{t^2-y_0^2}+
\frac{1}{t}\left(\log \frac{1 +\frac{a}{t}}{1 - \frac{a}{t}} +
\log \frac{1 + \frac{y_0}{t}}{1 - \frac{y_0}{t}}\right) .\]
For $r\in (0,1)$, $\log(1-r^2) + r \log \frac{1+r}{1-r}$ has
the series expansion \[ \sum_n \left(\frac{2}{2 n-1}
- \frac{1}{n} \right) r^{2 n} < r^2 + 2 r^4 \sum_{n\geq 2} \left(\frac{1}{2 n - 1} - \frac{1}{2 n}\right) 
= r^2 + (2\log 2 - 1) r^4.\]
Again by a series expansion, $-\frac{2\rho}{1-\rho^2} + \log \frac{1+\rho}{1-\rho}< -\frac{4}{3} \rho^3$ for $\rho\in (0,1)$. We set 
$r = \frac{a}{t}$, $\rho = \frac{y_0}{t}$.
\end{proof}

\begin{lemma}\label{lem:janframe}
Let $t\geq 1000$, $a\in (0,4]$. For $\gamma$ going over 
ordinates of non-trivial zeros of $\zeta(s)$,
\[\sum_{|\gamma-t|>a}
 \frac{1}{|t-\gamma|^2} \leq \frac{1}{\pi a}\log \frac{t}{2\pi} + 
 \frac{\frac{2}{5} \log t + 4 + Q((t-a)^-)-Q(t+a)}{a^2}.
 \]
\end{lemma}
\begin{proof}
    By {\eqref{eq:seguitor} and \eqref{eq:sagasto}} with (a) $t_0 = t+a$, $t_1\to \infty$,
$\phi(y) = \frac{1}{(y-t)^2}$,
 (b) $t_0=y_0^-$, $t_1 = (t-a)^-$, $\phi(y) = \frac{1}{(t-y)^2}$,
where $y_0= 2\pi e$, and
(c) $t_0 = y_0^-$, $t_1\to \infty$, $\phi(y) = \frac{1}{(y+t)^2}$,
\[\begin{aligned}
&\sum_{\substack{\gamma: |\gamma|\geq y_0\\|\gamma-t|>a}}
 \frac{1}{|t-\gamma|^2} \leq 
\sum_{\gamma>t+a} \frac{1}{|t-\gamma|^2} + 
\sum_{y_0\leq \gamma<t - a} \frac{1}{|t-\gamma|^2} +
\sum_{y_0 {\leq} \gamma} \frac{1}{|t+\gamma|^2}\\
&\leq\frac{1}{2\pi}\left(
\left(\int_{t+a}^\infty + \int_{y_0}^{t-a}\right) \frac{\log \frac{y}{2\pi}}{(y-t)^2} dy + 
\int_{y_0}^{\infty} \frac{\log \frac{y}{2\pi}}{(y+t)^2} dy 
\right)\\
&+ 
\frac{\frac{1}{5}\log(t+a) + 2 - Q(t+a)}{a^2} 
+ \frac{\frac{1}{5} \log(t-a) + 2 +Q((t-a)^-)}{a^2} + \frac{\frac{1}{5} \log y_0 + 2 - Q(y_0^-)}{(y_0+t)^2}\\
&+ \frac{1}{5} \left(\int_{t+a}^\infty
\frac{dy}{y (y-t)^2} - 
\int_{y_0}^{t-a} \frac{dy}{y (t-y)^2} +
\int_{y_0}^\infty \frac{dy}{y (y+t)^2}\right)
- \frac{\frac{1}{5} \log y_0 + 2 + Q(y_0^-)}{(y_0-t)^2}
\end{aligned}\]
By Lemma \ref{lem:daremo} the total of the first three integrals here is at most $\frac{2}{a} \log \frac{t}{2\pi} +
\frac{a}{t^2} + \frac{c_4}{t^4}$, where $c_4 = (2\log 2 - 1) a^3 - \frac{4}{3} y_0^3$.
Clearly $\log(t+a) + \log(t-a) = 2\log t + \log\left(1+\frac{a}{t}\right) + \log\left(1-\frac{a}{t}\right) \leq 2\log t - \left(\frac{a}{t}\right)^2$; here $\left(\frac{a}{t}\right)^2$ 
contributes $\frac{1}{5t^2}$ in the end.
We use $\log \frac{1+\epsilon}{1-\epsilon} \leq \frac{2\epsilon}{1-\epsilon^2}$ and
$-\log(1-\epsilon) \leq \frac{\epsilon}{1-\epsilon}$ to bound the contribution of the
 last three integrals by $\frac{1}{5}$ times
\[-\frac{2\log \frac{t}{a}}{t^2} + 
\frac{\log \frac{t+y_0}{t-y_0} - \log \left(1 -
\frac{a^2}{t^2}\right)}{t^2} + \frac{2 y_0}{t (t^2-y_0^2)}\leq
-\frac{2\log \frac{t}{a}}{t^2} +
\frac{a^2}{t^2 (t^2-a^2)} + \frac{4 y_0}{t (t^2-y_0^2)}
\]


Since $y_0= 2\pi e$ and $N(y_0^-) = 1$, we know that $Q(y_0^-) = \frac{1}{8}>0$. By convexity,
$\frac{1}{(t+y_0)^2}+\frac{1}{(t-y_0)^2}\geq \frac{2}{t^2}$.
We must still account for
$\frac{1}{2} \pm i \gamma_0$ with 
$\gamma_0 = 14.134725\dotsc < 2\pi e$.
Since $f(\epsilon) = \frac{1}{\epsilon^2} 
\left(\frac{1}{(1-\epsilon)^2} + \frac{1}{(1+\epsilon)^2} - 2\right)=\frac{2(3-\epsilon^2)}{(1-\epsilon^2)^2}$ is increasing on $0<\epsilon<1$,
$\frac{1}{(t-\gamma_0)^2} +
\frac{1}{(t+\gamma_0)^2}= \frac{2}{t^2} +
\frac{\epsilon^2}{t^2} f(\epsilon)\leq
\frac{2}{t^2} + \frac{\gamma_0^2}{t^4} f\left(\frac{\gamma_0}{10^3}\right)$
for $\epsilon = \frac{\gamma_0}{t}$. Hence
\[\begin{aligned}\sum_{|\gamma-t|>a}
 \frac{1}{|t-\gamma|^2} &\leq
 \frac{1}{\pi a}\log \frac{t}{2\pi} + 
 \frac{\frac{2}{5} \log t + 4 + Q((t-a)^-)-Q(t+a)}{a^2}\\
 &+ \frac{2 + \frac{a}{2\pi} - \frac{2}{8} - \frac{1}{5} -\frac{2}{5} \log \frac{t}{a}}{t^2} + \frac{4 y_0}{5 t^3 (1-(y_0/t)^2)} + \frac{\kappa}{t^4}
 \end{aligned}\]
with $\kappa = 
\gamma_0^2 f\left(\frac{\gamma_0}{10^3}\right) + \frac{c_4}{2\pi} + \frac{a^2}{5(1-(a/t)^2)} < 150$.
 The sum of $\frac{7}{4} + \frac{a}{2\pi} - \frac{1}{5} - \frac{2}{5}\log \frac{2\cdot 10^3}{a} = -1.2131 + \frac{a}{2\pi} + \frac{2}{5} \log a <  -0.0219$ and $\frac{4\cdot 2\pi e}{5\cdot 10^3 (1-(2 \pi e/10^3)^2)} + \frac{150}{(10^3)^2}\leq 0.0139$
is negative.
\end{proof}

\begin{proposition}\label{prop:titch96A} Let $a>0$, $\sigma_+>1$. For $s=\sigma+it$ with 
$-2\leq \sigma\leq \sigma_+\leq 2$, $t\geq \max(1000,a)$,
\begin{align} \label{eq:kombucha}
\dfrac{\zeta'}{\zeta}(s) & = \sum_{{\rho: |\Im \rho - t|\leq a }} \dfrac{1}{s-\rho} + O^*\left(\kappa_1 \log \frac{t}{2\pi} + \kappa_2\cdot \left(\frac{2}{5}\log t + 4\right) + \epsilon
+ \left|\frac{\zeta'}{\zeta}(\sigma_+)\right|\right),
\end{align}
where $\kappa_1 = 
\frac{1}{\pi} \left( \frac{\sigma_+-\sigma}{a} + \frac{a}{\sigma_+-1}\right)$, $\kappa_2 = 
 \frac{\sigma_+-\sigma}{a^2} + \left|\frac{\sigma_+-\sigma}{a^2}-\frac{1}{\sigma_+-1}\right|$, and $\epsilon =2.02\cdot 10^{-3}$. If every $\rho$ with $t-a\leq \Im \rho\leq t+a$ satisfies $\Re \rho =\frac{1}{2}$, then both instances of 
$\frac{1}{\sigma_+-1}$ can be replaced by $\frac{1}{\sigma_+-\frac{1}{2}}$.

\end{proposition}
Our aim here is to give a clean version of \cite[Theorem 9.6 (A)]{MR882550}, not to optimize every constant. We have arranged matters carefully so that we do obtain good constants, just by proceeding logically: for instance, the terms $Q(t_0)$, $Q(t_1)$ from \eqref{eq:seguitor} and \eqref{eq:sagasto} will in part cancel -- an excess in zeros just outside the interval $[t-a,t+a]$ necessitates a deficit inside the interval.

One could do better by applying Lemma \ref{lem:Lehmanmodern} with $\phi$ complex-valued to bound
the left side of \eqref{eq:11_01pm}, and perhaps also inside the proof of Lemma \ref{lem:janframe}. Another possible
improvement would be not to replace \eqref{eq:slante} by a bound proportional to $\frac{1}{|t-\gamma|}$ (rather than $\frac{1}{|t-\gamma|^2}$) for $|t-\gamma|<1$, say.
\begin{proof} By \cite[Corollary 10.14]{MR2378655}, for all $s\in\C$:
\begin{align} \label{10_15pm}
\dfrac{\zeta'}{\zeta}(s) = \sum_{\rho}\left(\dfrac{1}{s-\rho}+\dfrac{1}{\rho}\right) - \dfrac{1}{2}\,\digamma\left(\frac{s}{2}+1\right) - \dfrac{1}{s-1} + C,
\end{align}
where $C$ is a constant. Evaluate \eqref{10_15pm} at $s=\sigma+it$ and $s_+=\sigma_++it$, and take the difference:
\[
\dfrac{\zeta'}{\zeta}(s) = \sum_{\rho}\left(\dfrac{1}{s-\rho}-\dfrac{1}{s_+-\rho}\right) - \dfrac{1}{2}\left(\digamma\left(\frac{s}{2}+1\right)-\digamma\left(\frac{s_+}{2}+1\right)\right)+\dfrac{\zeta'}{\zeta}(s_+) + O^*\left(\dfrac{\sigma_+-\sigma}{t^2}\right).
\]
Clearly, \begin{equation}\label{eq:slante}\sum_{\rho: |\Im \rho - t|>a}
\left|\frac{1}{s-\rho}-\frac{1}{s_+-\rho}\right|
\leq \sum_{|{\gamma} - t|> a}
\frac{\sigma_+-\sigma}{|t-\gamma|^2},\end{equation} and, by Lemma~\ref{lem:janframe},
\begin{equation}\label{eq:horgoz}\begin{aligned} \sum_{\rho: |\Im \rho - t|> a}
\frac{1}{|t-\gamma|^2}\leq \frac{1}{\pi a}\log \frac{t}{2\pi} + 
 \frac{\frac{2}{5} \log t + 4 + Q((t-a)^-)-Q(t+a)}{a^2}.
\end{aligned}\end{equation}
By $\Re{\rho}\leq 1$ and Corollary \ref{cor:zeroinbox},
\begin{align} \label{eq:11_01pm}
\sum_{\rho: |\Im \rho - t|\leq a }\frac{1}{\left|s_+-\rho\right|}\leq \frac{N(t+a)-N((t-a)^-)}{\sigma_+-1} \leq 
\frac{Q(t+a)-Q((t-a)^-) +\frac{a}{\pi}\log \frac{t}{2\pi}}{\sigma_+-1}.
\end{align}
The terms $Q(t+a)$, $Q((t-a)^-)$ here cancel out partly with those in \eqref{eq:horgoz}.

By \eqref{eq:digamma},
\[\frac{1}{2}\left(\digamma\left(\frac{s}{2}+1\right)-\digamma\left(\frac{s_+}{2}+1\right)\right)
= \frac{1}{2} \log \frac{s+2}{s_++2} - \frac{1}{2} \left(
\frac{1}{s+2} - \frac{1}{s_++2}\right) + O^*\left(
\frac{1}{3\sqrt{2}\cdot t^2}\right).\]
For $z\in \mathbb{C}$ with $|z|\leq \frac{2}{3}$, $\log(1+z) = z + O^*(|z|^2)$. Letting $z = \frac{\sigma-\sigma_+}{s_++2}$, we obtain
\[\frac{1}{2}\left|\digamma\left(\frac{s}{2}+1\right)-\digamma\left(\frac{s_+}{2}+1\right)\right|\leq
 \frac{\sigma_+-\sigma}{2 t} + \frac{1}{t^2}
\left(\frac{(\sigma_+-\sigma)^2}{2} + 
\frac{\sigma_+-\sigma}{2} +\frac{1}{3\sqrt{2}}\right).\]
Lastly, $\left|\frac{\zeta'}{\zeta}(s_+)\right| = \left|\sum_n \Lambda(n) n^{-s_+}\right| \leq \sum_n \Lambda(n) n^{-\sigma_+} =
\left|\frac{\zeta'}{\zeta}(\sigma_+)\right|$. We obtain the statement with
\[\epsilon= \frac{\sigma_+-\sigma}{2 t} + \frac{1}{t^2}
\left(\frac{(\sigma_+-\sigma)^2}{2} + 
\frac{3}{2} (\sigma_+-\sigma) +\frac{1}{3\sqrt{2}}\right)\leq 0.00202.\]

\end{proof}

\section{Series, functions and comparisons}\label{sec:darf}
\subsection{Estimates}\label{subs:darf1}
\begin{lemma}\label{lem:sibelius}
Let $F(z) = \frac{1}{\pi}- (1-z) \cot \pi (1-z)$. Then \begin{enumerate}[(a)]
\item $F(x)$ is decreasing on $(0,1]$, and $0< F(x)<\frac{1}{\pi x}$\, for $x\in (0,1)$.
\item For any $x\in \left(0,1\right]$, $y\in \left[-\frac{1}{2},\frac{1}{2}\right]$, 
\begin{equation}\label{eq:solucky}|F(x+iy)-F(x)| \leq \frac{|y|}{\pi x |x+i y|} + 1.78 |y| ,\end{equation}
\item Let $A(z) = F(z) -\frac{1}{\pi z}$. Then, 
 for $x\in \left(0,\frac{1}{2}\right]$, $y\in \left[-\frac{1}{2},\frac{1}{2}\right]$,
\begin{equation}\label{eq:garmenio}
|A(x)|\leq \frac{\pi}{3} x,\quad   \mbox{and} \,\,\,\,\,\,\, |A(x+i y) - A(x)|\leq 1.78 |y|.\end{equation}
\end{enumerate}
\end{lemma}
\begin{proof}
By \eqref{eq:mireio}, $F(z)=\frac{2}{\pi}\sum_{n}\zeta(2n)(1-z)^{2n}$ for $|z-1|<1$. In particular, $F(x)\geq 0$ for $0<x\leq 1$, and $F$ is decreasing on $(0,1]$. Moreover, since for $0<x<1$ we have $\frac{1}{\pi x}-\cot \pi x>0$ and $\cot \pi (1-x) = -\cot \pi x$, it follows that $F(x) <\frac{1}{\pi}+(1-x)\frac{1}{\pi x} = \frac{1}{\pi x}$.

\noindent
{\bf Case $\frac{1}{2}\leq x\leq 1$.} 
We can assume $y\geq 0$. First, $F(x+i y) - F(x) = i \int_0^y F'(x+i r) dr$. Since
$F'(z) = i \left(\frac{s}{i} \cot \frac{s}{i} \right)' |_{s = i \pi (1-z)}
= i (s \coth s)'|_{s = i \pi (1-z)}$,
Lem.~\ref{lem:cothder} yields $|F'(z)|\leq |\pi (1-z)|$. Thus,
\[
|F(x+i y) - F(x)|\leq \pi \int_0^y \sqrt{\frac{1}{4}+r^2}\, dr = \pi y \cdot \frac{1}{y}
\int_0^y \sqrt{\frac{1}{4}+r^2}\, dr \leq \pi y \cdot \frac{1}{\frac{1}{2}} \int_0^{\frac{1}{2}} \sqrt{\frac{1}{4}+r^2}\, dr
\]
since $\sqrt{\frac{1}{4}+r^2}$ is increasing. By $\int_0^{\frac{1}{2}} \sqrt{\frac{1}{4}+r^2}\, dr = 
\frac{\sqrt{2} + \sinh^{-1}(1)}{8}$, we conclude that $|F(x+i y)-F(x)|\leq M |y|$ with $M=\frac{\pi}{4}(\sqrt{2} + \sinh^{-1}(1)) = 1.80294\dotsc$, which implies \eqref{eq:solucky} in this case.

\noindent {\bf Case $0<x\leq \frac{1}{2}$.}
 We see that
$A(z) = - (1-z) \left(\frac{1}{\pi z} - \cot \pi z\right)$
and, by \eqref{eq:mireio},  the Taylor series of
$f(z) = \frac{1}{\pi z} - \cot \pi z$ at $z=0$ is
$\frac{2}{\pi} \sum_n \zeta(2 n) z^{2 n-1}$;
write $f(z) = \frac{\frac{2}{\pi} \zeta(2) z - h(z)}{1-z^2}$, where $h(z) =
\frac{2}{\pi} \zeta(2) z - (1-z^2) \left(\frac{1}{\pi z} - \cot \pi z\right) = \frac{2}{\pi}
\sum_n a_n z^{2 n+1}$ with $a_n=\zeta(2 n)-\zeta(2 n +2)$.
Then
\begin{align} \label{cotaA'}
\!\! A'(z) =  -((1-z) f(z))' = -\left(\frac{2}{\pi} \frac{\zeta(2) z}{1+z} - \frac{h(z)}{1+z}\right)'
= - \frac{2}{\pi} \frac{\zeta(2)}{(1+z)^2} - \frac{h(z)}{(1+z)^2} + \frac{h'(z)}{1+z}.
\end{align}
Since $a_n\geq 0$, it follows that, for 
$z=x+iy$ with $0\leq x\leq \frac{1}{2}$, $|y|\leq \frac{1}{2}$,
\[|A'(z)|\leq\frac{2}{\pi} \zeta(2) + \left|h\left(\frac{1}{\sqrt{2}}\right)\right| +
\left|h'\left(\frac{1}{\sqrt{2}}\right)\right| = 1.77963\dotsc .
\]
So, $|A(x+iy)-A(x)|\leq 1.78 |y|$. 
Now \eqref{eq:solucky} follows by $\left|\frac{1}{x+i y} - \frac{1}{x}\right| = \frac{|y|}{|x (x+i y)|}$. 

Lastly, by \eqref{cotaA'}, 
$A'(x)=\frac{2}{\pi(1+x)^2}\left(-\zeta(2)+\sum_{n}a_n g_n(x)\right)$, where $g_n(x)=(2nx+(2n+1))x^{2n}$, which is increasing.  Since $a_n\geq 0$ and $A'(\frac{1}{2})=\frac{8-\pi^2}{2\pi}<0$, we get that
$A'(0)\leq A'(x)<0$ on $[0,\frac{1}{2})$. Since $A'(0) = - \frac{\pi}{3}$ and
$\lim_{z\to 0} A(z) = 0$, this gives $|A(x)|\leq \frac{\pi}{3} x$ for $x\in [0,\frac{1}{2})$.
\end{proof}

We will need a comparison of coefficients.
\begin{lemma}\label{lem:tremic}
For $n\geq 1$, let
\[a_n(t) = (1-t)^2  -
\left(\frac{1}{2 n} - \frac{2 t}{2 n+1} + 
\frac{t^{2}}{2 n + 2}\right),
\]
\[
    b_n(t)  =  (1-t)^2  -
2 \left(\frac{1}{2 n} - \frac{2 t}{2 n+1} + 
\frac{t^{2}}{2 n + 2}\right)
+ \left(\frac{1}{(2 n)^2} - \frac{2 t}{(2 n+1)^2}
+ \frac{t^{2}}{(2 n + 2)^2}\right).\]
Then $a_n(t)\geq 0$ for all $t\in [0,\frac{1}{2}]$ and $b_n(t)\geq 0$ for all $t\in [0,\frac{1}{3}]$.
\end{lemma}
\begin{proof} 
Let $\alpha, \beta, \theta>0$ such that $\beta^2\geq \alpha\theta$. Then $f(t)=\alpha t^2-2\beta t +\theta\geq 0$ for all $0\leq t\leq \frac{\beta-\sqrt{\beta^2-\alpha\theta}}{\alpha}$.
Thus, $f(t)\geq 0$ for all $0\leq t\leq t_0$ if and only if $\frac{\beta}{\alpha} \geq t_0$ and
$(\beta - \alpha t_0)^2\geq \beta^2-\alpha \theta$, i.e., $\frac{2\beta}{t_0} - \frac{\theta}{t_0^2} \leq \alpha$.
In particular, $f(t)\geq 0$ for all $t\in [0,\frac{1}{2}]$ if $\beta\geq \frac{\alpha}{2}$ and $4(\beta-\theta)\leq \alpha$, and $f(t)\geq 0$ for all $t\in [0,\frac{1}{3}]$ if $\beta\geq \frac{\alpha}{3}$ and $6\beta-9\theta\leq \alpha$.

Now,  $a_n(t)=c_{2n+2}t^2 - 2c_{2n+1}t+c_{2n}$ for $c_n = 1-\frac{1}{n}$. Since $c_{2n+1}\geq \frac{c_{2n}}{2}$ and  $4(c_{2n+1}-c_{2n})\leq c_{2n+2}$ we get the result. Moreover, $b_n(t)=c^2_{2n+2}t^2 - 2c^2_{2n+1}t+c^2_{2n}$, and since $c_{2n+1}^2\geq c_{2 n}^2/2$ and $6 c_{2n+1}^2 -9 c_{2n}^2
 = \frac{6}{2 n (2 n + 1)} \frac{8 n^2 - 1}{2 n (2 n+1)} -3 c_{2 n}^2\leq 
 \frac{6}{2\cdot 3} \frac{7}{6} - \frac{3}{4} < c_4^2\leq c_{2n+2}^2$, we are done.
\end{proof}

\begin{lemma}\label{lem:konstanz}
For $k\geq 1$, let $C_k = \sum_n \zeta(2 n) \left(\frac{1}{(2 n)^k} - \frac{2}{(2 n+1)^k} + \frac{1}{(2 n + 2)^k}\right)$. Then
\[C_1 = 0.168938\dotsc, \quad C_2 = 0.164184\dotsc .\]
\end{lemma}
The real task here is to ensure rapid convergence.
\begin{proof}
For $k\geq 2$, set aside $c_k = \sum_n \left(\frac{1}{(2 n)^k} - \frac{2}{(2 n+1)^k} + \frac{1}{(2 n + 2)^k}\right) = \frac{-1}{2^k} +\sum_n \frac{2}{(2 n)^k} - \sum_n \frac{2}{(2 n + 1)^k} = 
2-\frac{1}{2^k} + \sum_n \frac{4}{(2 n)^k} - \sum_m \frac{2}{m^k}$,
which is $\frac{7}{4} - \frac{\pi^2}{6}$ for $k=2$. For $k=1$, let
$c_1 = \sum_n \left(\frac{1}{2 n} - \frac{2}{2 n+1} + \frac{1}{2 n + 2}\right)$, which
equals the limit as $N\to \infty$ of $2 - \frac{1}{2} + \frac{1}{2 N +2} + 2 H_N
- 2 H_{2 N+1}$, where $H_N = \sum_{n\leq N} \frac{1}{n}$. By
$H_N = \log N + \gamma + o(1)$, this limit equals $\frac{3}{2} - 2\log 2$.

The series $\sum_n (\zeta(2 n)-1) \left(\frac{1}{(2 n)^k} - \frac{2}{(2 n+1)^k} + \frac{1}{(2 n + 2)^k}\right)$
converges exponentially:
for $t\geq 2$, $0<\zeta(t)-1 \leq 2^{-t} \frac{\zeta(2)-1}{2^{-2}}< 3\cdot 2^{-t}$.
We compute $50$ terms of the series for $k=1,2$ with Arb/FLINT.
\end{proof}

\begin{lemma}\label{lem:bettnist}
Let $\Ei(x)$ be the exponential integral. Then, for all $x>0$,
\[\Ei(x) \leq \frac{e^x}{x} \left(1 + \frac{1}{x} +\frac{2}{x^2} +\frac{40/3}{x^3}\right).\]
    \end{lemma}
    In a better world, \cite[\S 6.12]{zbMATH05765058} would give remainder terms for $\Ei(x)$ with optimal constants.
\begin{proof}
 Let $f(x) = \frac{e^x}{x} \left(1+ \frac{1}{x} + \frac{2}{x^2} + \frac{\alpha}{x^3}\right)$ for $\alpha=\frac{40}{3}$.
 By 
 $f'(x) = \frac{e^x}{x} \left(1 + \frac{\alpha - 6}{x^4} - \frac{4\alpha}{x^5}\right)$ and
 $\Ei'(x) = \frac{e^x}{x}$,
 we see that $\sgn(f'(x)-\Ei'(x))=\sgn((\alpha-6) x - 4\alpha)$; thus, the claim is true if and only if it is true at $x=4\alpha/(\alpha-6)$. We check it there numerically, using Arb/FLINT.
\end{proof}

\begin{lemma}\label{lem:rameau} Let $0<\eta\leq e$ and $x\geq e^7$. Write $L=\log x$. Then
\begin{align}
\int_{0}^\infty\dfrac{ux^{-u}}{\sqrt{\eta^2+(\frac{1}{2}-u)^2}}du &\leq
\left(\frac{\frac{1}{2} \Ei(L/2)}{e^{L/2}} - \frac{1}{L}\right) + 
\frac{1}{\sqrt{x}} \left(\frac{1}{2} \log \frac{1}{\eta}+ \frac{1 + \frac{2}{L}}{2 \eta L}\right)
\label{eq:arguc1}\\
&\leq 
\frac{2}{L^2} + \frac{8}{L^3} + \frac{320}{3 L^4} +  
\frac{1}{\sqrt{x}} \left(\frac{1}{2} \log \frac{1}{\eta}+ \frac{1 + \frac{2}{L}}{2 \eta L}\right)
.\label{eq:arguc2}
\end{align}
\end{lemma}
\begin{proof}
We can write our integral $I$ as $I_0+I_++I_-$, where
\[I_0 = \int_0^{\frac{1}{2}} \frac{\frac{1}{2} x^{-\frac{1}{2}}}{\sqrt{\eta^2+\left(\frac{1}{2}-u\right)^2}}du,\;\;\;
 I_+ = \int_{0}^{\frac{1}{2}} \frac{u x^{-u} - \frac{1}{2} x^{-\frac{1}{2}}}{\sqrt{\eta^2+\left(\frac{1}{2}-u\right)^2}}du,\;\;\;I_- = 
\int_{\frac{1}{2}}^\infty \frac{u x^{-u}}{\sqrt{\eta^2+\left(\frac{1}{2}-u\right)^2}}du.\]

\noindent {\bf On $I_0$.} By a change of variables $v=\frac{1}{2}-u$, $I_0 = \frac{1}{2 \sqrt{x}} \int_0^{\frac{1}{2}} \frac{dv}{\sqrt{\eta^2+v^2}}  = \frac{\arsinh \frac{1}{2\eta}}{2\sqrt{x}}$.

\noindent{\bf  On $I_+$.}
We would like to replace the denominator in  $I_+$ by just $(\frac{1}{2}-u)$. That is straightforward for an upper bound when the
numerator is non-negative.
Let $\epsilon=1/L$. For $u\geq 1/L$, since $u x^{-u}$ is decreasing, we know that
$u x^{-u} - \frac{1}{2} x^{-1/2}\geq 0$.  For $u<1/L$, we bound $u x^{-u} - \frac{1}{2} x^{-1/2}\leq u x^{-u}$ first, and then change denominators. Thus
\[I_+\leq \int_{\epsilon}^{\frac{1}{2}} \frac{u x^{-u} - \frac{1}{2} x^{-\frac{1}{2}}}{\frac{1}{2}-u}du
+ \int_{0}^{\epsilon} \frac{u x^{-u}}{\frac{1}{2}-u} du =
I_+^* + \frac{x^{-\frac{1}{2}}}{2} \int_0^\epsilon 
\frac{du}{\frac{1}{2}-u}  = 
I_+^* +\frac{-\log(1-2\epsilon)}{2\sqrt{x}} ,\] 
where $I_+^* = \int_0^{\frac{1}{2}} \frac{u x^{-u} - \frac{1}{2} x^{-\frac{1}{2}}}{\frac{1}{2} - u} du$.
 Now, by a change of variables $v=\frac{1}{2}-u$,
 \[I_+^* 
= \frac{1}{\sqrt{x}} \int_0^{\frac{1}{2}} \left(\frac{\frac{1}{2} (x^v-1)}{v} - x^v\right) dv = \frac{1}{\sqrt{x}}  \left(\frac{1}{2} \left(\Ei\left(\frac{L}{2}\right)
-\log \frac{L}{2} - \gamma\right)
- \frac{\sqrt{x}-1}{L}\right)
\]
by  $\int_0^1 \frac{e^{y t} -1}{t} dt = \Ei(y) - \log y-\gamma $ for $y>0$
\cite[(6.2.3), (6.2.7)]{zbMATH05765058}.

 \noindent {\bf On $I_-$.} We bound simply $I_-\leq \frac{1}{\eta} \int_{\frac{1}{2}}^{{\infty}} u x^{-u} du =
 \frac{1}{\eta} \left(\frac{1}{2L} + \frac{1}{L^2}\right)\frac{1}{\sqrt{x}}$.

We take totals: \[I\leq 
\left(\frac{\frac{1}{2} \Ei(L/2)}{e^{L/2}} - \frac{1}{L}\right) + 
\frac{1}{\sqrt{x}} \left(\frac{\arsinh \frac{1}{2\eta}}{2}
- \frac{\log L}{2}+\frac{\log 2 - \gamma}{2} - \frac{\log\left(1 - \frac{2}{L}\right)}{2}+\frac{1}{L}
+ {\frac{1+\frac{2}{L}}{2\eta L}}
\right).
\]

We know $\arsinh t - \log t$ is decreasing
because $\arsinh' t = \frac{1}{\sqrt{t^2+1}} < \frac{1}{t}$; hence, by $\eta\leq e$, {$\arsinh \frac{1}{2\eta} \leq \log \frac{1}{\eta} + \arsinh \frac{1}{2e} +1$}. 
Since 
$f(y) = \frac{\arsinh \frac{1}{2 e} + 1 - \log y+\log 2 -\gamma}{2} -\frac{\log\left(1 - \frac{2}{y}\right)}{2} + \frac{1}{y}$
is decreasing and $L\geq 7$, $f(L)\leq f(7)<0$.
We conclude \eqref{eq:arguc1} holds. By Lemma \ref{lem:bettnist}, \eqref{eq:arguc2} follows.
\end{proof}

\begin{lemma}\label{lem:convexistan}
Let $G:[0,a]\to \mathbb{R}$ be a $C^2$ function with 
$G(0)=0$. Assume $G''$ is increasing on $[0,a]$. Let $x>1$. Then
$$\int_0^a G(t) x^{-t} dt \leq \frac{G'(0)}{\log^2 x} +
\frac{\Delta}{\log^3 x} -  \left(\frac{a M}{x^a \log x} + \frac{G'(a)}{x^a \log^2 x}
+ \frac{\Delta}{x^a \log^3 x}\right),
$$
where $\Delta = \frac{1}{a} (G'(a) - G'(0))$ and
$M = \frac{1}{2} (G'(0) + G'(a))$. Note that $G(a)\leq a\cdot M$. 
\end{lemma}
\begin{proof}
Since $G'$ is convex on $[0,a]$, $G'(t)\leq \frac{t}{a} (G'(a) - G'(0)) + G'(0)$ for all $0\leq t\leq a$. Hence, $G(t) = \int_0^t G'(u) du\leq q t^2 + G'(0) t$ for
$q =  \frac{G'(a)-G'(0)}{2 a}$. By repeated integration by parts,
\[\int_0^a t x^{-t} dt = 
\frac{1}{\log^2 x} - \frac{a}{x^a \log x} - \frac{1}{x^a \log^2 x},\; \int_0^a t^2 x^{-t} dt =  
\frac{2}{\log^3 x} - \frac{a^2}{x^a \log x} - \frac{2 a}{x^a \log^2 x} - 
\frac{2}{x^a \log^3 x}.\]
Therefore,
\[\begin{aligned}
\int_0^a G(t) x^{-t} dt &\leq \frac{G'(0)}{\log^2 x} + \frac{2 q}{\log^3 x}
- \frac{q a^2 + G'(0) a}{x^a \log x} - \frac{2 q a + G'(0)}{x^a \log^2 x} -
\frac{2 q}{x^a \log^3 x}.
\end{aligned}\]
Here  $2 q = \frac{1}{a} (G'(a)-G'(0))$, $q a^2 + G'(0) a = \frac{a}{2} (G'(0)+G'(a))$ and $2 a q + G'(0) = G'(a)$.
\end{proof}

\subsection{Our weights in terms of a special function}\label{subs:darf2}

Here is a reasonably ``closed-form'' expression for our weights $\widehat{\varphi_{\rho}^+}$,
$\widehat{\varphi_{\rho}^-}$ on the integers. We use it only for plotting Figure \ref{fig:grahamvaaler}.

The {\em Lerch transcendent} $\Phi(z,s,\alpha)$ (not to be confused with $\Phi_\lambda$) is a special function defined by
\begin{equation}\label{eq:lerch}
\Phi(z,s,\alpha) = \sum_{n=0}^\infty \frac{z^n}{(n+\alpha)^s}\end{equation}
for $|z|<1$, provided that $s\in \mathbb{Z}_{>0}$ and
$\alpha\not\in \mathbb{Z}_{\leq 0}$  (or some other conditions that we need not
worry about)
\cite[\S 25.14]{zbMATH05765058}. For $|z|<1$ and $s\in \mathbb{Z}_{>0}$, $\sin \pi z \cdot \Phi(z,s,\alpha)$ tends
to a limit as $\alpha$ approaches a non-positive integer.

\begin{lemma}
 Let $\varphi_{\rho}^+$, $\varphi_{\rho}^-$ be as in \eqref{eq:arnor1}. Then, for $\rho\ne 0$,
$$\widehat{\varphi_{\rho}^+}(z) = \left(\frac{\sin \pi z}{\pi}\right)^2 \left(\Phi(e^{-\rho},2,z) + \rho \Phi(e^{-\rho},1,z)
- \frac{\rho/z}{1 - e^{-\rho}} \right),$$
where $\Phi$ is the Lerch transcendent \eqref{eq:lerch}.
Moreover,  $\widehat{\varphi_{\rho}^-}(z) = \widehat{\varphi_{\rho}^+}(z)
- \frac{\sin^2 \pi z}{(\pi z)^2}$.
\end{lemma}
\begin{proof}
We defined $\varphi_\rho^+(t) = \widehat{M_\rho}(-t)$, and so
$\widehat{\varphi_\rho^+}(z) = M_\rho(-z) = \left(\frac{\sin \pi z}{\pi}\right)^2 f_\rho(-z)$, where
$$\begin{aligned}f_\rho(z) &= \sum_n \left(\frac{e^{-\rho n}}{(n-z)^2} + 
\frac{\rho e^{-\rho n}}{n-z} + \frac{\rho e^{-\rho n}}{z}\right) + \frac{1}{z^2} = \sum_{n=0}^\infty \frac{(e^{-\rho})^n}{(n-z)^2} + \rho
\sum_{n=0}^\infty \frac{(e^{-\rho})^n}{n-z} + \frac{\rho}{z} + 
\frac{\rho/z}{e^{\rho}- 1}\\
&=\Phi(e^{-\rho},2,-z) + \rho \Phi(e^{-\rho},1,-z)
+ \frac{\rho/z}{1 - e^{-\rho}}
.\end{aligned}$$
The statement on $\widehat{\varphi_\rho^-}(t)$ is as in 
 \cite[(3.7)]{zbMATH03758875}.
\end{proof}

\bibliographystyle{alpha}
\bibliography{chirhelf}
\end{document}